\documentclass{amsart}
\usepackage[USenglish,activeacute]{babel}
\usepackage[T1]{fontenc}
\usepackage{tikz-cd}
\usepackage[utf8]{inputenc}
\usepackage{fancyhdr}
\usepackage{yhmath}
\usepackage{mathrsfs}
\usepackage{graphicx}
\usepackage{wrapfig}
\usepackage{caption}
\usepackage{dsfont}
\usepackage[normalem]{ulem}
\usepackage{enumitem}   
\usepackage{cutwin}
\usepackage{amsfonts}
\usepackage{mathtools}
\usepackage{subcaption}
\usepackage[colorlinks=true]{hyperref}
\usepackage{multirow}
\usepackage{amsmath}
\usepackage{amssymb}
\usepackage{accents}
\usepackage{amsthm}
\usepackage{textcomp}
\usepackage{adjustbox}
\makeatletter
\DeclareFontFamily{U}{tipa}{}
\DeclareFontShape{U}{tipa}{m}{n}{<->tipa10}{}
\makeatother

\newtheorem*{not*}{{ Notation}}
\newtheorem{defi}{ Definition}[subsection]
\newtheorem*{defi*}{ Definition}
\newtheorem{teo}[defi]{ Theorem}
\newtheorem*{teo*}{{ Theorem}}
\newtheorem{prop}[defi]{ Proposition}
\newtheorem*{prop*}{{ Proposition}}
\newtheorem{obs}[defi]{{Remark}}
\newtheorem{intteo}{Theorem}

\newtheorem*{Lemma*}{Lemma}
\newtheorem{Lemma}[defi]{Lemma}
\newtheorem*{coro*}{Corollary}
\newtheorem{coro}[defi]{Corollary}

\newcommand{\tarc}{\mbox{\large$\frown$}}
\newcommand{\arc}[2][-3ex]{{#2}{\kern #1{\raisebox{1.5ex}{\tarc}}}}

\newcommand{\op}{\operatorname{op}}
\newcommand{\Sp}{\operatorname{Sp}}
\newcommand{\Aut}{\operatorname{Aut}}
\newcommand{\Spf}{\operatorname{Spf}}
\newcommand{\End}{\operatorname{End}}
\newcommand{\Hom}{\operatorname{Hom}}
\newcommand{\Der}{\operatorname{Der}}
\newcommand{\Mod}{\operatorname{Mod}}
\newcommand{\D}{\mathcal{D}}
\newcommand{\OX}{\mathcal{O}}

\newcommand{\Spec}{\operatorname{Spec}}
\newcommand{\gr}{\operatorname{gr}}
\newcommand{\Rep}{\operatorname{Rep}}
\newcommand{\Spa}{\operatorname{Spa}}

\title{ $p$-adic Cherednik algebras on rigid analytic spaces}
\author{Fernando Peña Vázquez}
\email{fpvmath@gmail.com}
\address{Mathematisch-Naturwissenschaftliche Fakult\"at der Humboldt-Universit\"at zu Berlin, Rudower Chaussee 25, 12489 Berlin, Germany}
\begin{document}
\begin{abstract}
Let $X$ be a smooth rigid space with an action of a finite group $G$ satisfying that $X/G$ is represented by a rigid space. We construct sheaves of $p$-adic Cherednik algebras on the small étale site of the quotient $X/G_{\textnormal{ét}}$, and study some of their properties. The sheaves of $p$-adic Cherednik algebras are sheaves of Fréchet-Stein algebras on $X/G_{\textnormal{ét}}$, which can be regarded as $p$-adic analytic versions of the  sheaves of Cherednik algebras associated to the action of a finite group on a smooth algebraic variety defined by  P. Etingof.
\end{abstract}
\maketitle

\tableofcontents
\section{Introduction}
Double affine Hecke algebras (DAHA) were introduced by I. Cherednik
in \cite{Cherednikmcdonalds} as a tool in his proof of Macdonald’s conjectures about orthogonal polynomials for root systems. Ever since then, DAHA have become prominent objects in representation theory, and have found their way into multiple branches of mathematics, with a plethora of applications ranging from harmonic analysis to deformation theory and the geometric Langlands program. See, for example, the monographs \cite{introduction DAHA}, and \cite{DAHA CUP} for an overview of the theory, and \cite{DAHA-Fourier} for some of its applications in arithmetic, and in particular to the $p$-adic spherical transform.\\
This paper is a part of a series of papers \cite{HochschildDmodules},\cite{HochschildDmodules2},\cite{p-adicCatO}, in which we convey a systematic study of Ardakov-Wadsley's sheaf of completed $p$-adic differential operators $\wideparen{\D}_X$ and its relevant categories of modules (\emph{i.e.} co-admissible, $\mathcal{C}$-complexes \emph{etc}) employing techniques from non-commutative geometry and deformation theory. In this paper, we will study $p$-adic versions of certain rational degenerations of DAHA called rational Cherednik algebras, as well as their generalizations for actions of a finite group on a smooth rigid space, which we will call $p$-adic Cherednik algebras.\bigskip

Before explaining the contents of the paper, let us give an overview of the relevant objects: Let $\mathfrak{h}$ be a finite dimensional complex vector space, and $G\leq \operatorname{GL}(\mathfrak{h})$ be a finite group. We let $S(G)$ be the set of pseudo-reflections in $G$ (\emph{i.e.} the family of elements $s\in G$ satisfying $\operatorname{Rank}(s-\operatorname{Id})=1$), and define a reflection function to be a $G$-invariant function $c:S(G)\rightarrow \mathbb{C}$. We let $\operatorname{Ref}(\mathfrak{h},G)$ be the finite-dimensional vector space of all reflection functions. The family of rational Cherednik algebras associated to the action of $G$ on $\mathfrak{h}$ is the family of filtered $\mathbb{C}$-algebras $H_{c}(\mathfrak{h},G)$,
parameterized by  reflection functions $c\in \operatorname{Ref}(\mathfrak{h},G)$, and defined as the quotient of the tensor algebra 
$G\ltimes T(\mathfrak{h}\oplus \mathfrak{h}^*)$ by the following relations:
\begin{equation}\label{equation presentation of rational Cherednik algebra}
   [v,w]=0, \textnormal{ } [x,y]=0, \textnormal{ } [v,x]-(v,x)+\sum_{g\in S(G)}c(g)(v,\alpha_g)(\alpha_g^{\vee},x)g=0,
\end{equation}
where $v,w\in \mathfrak{h}$, and $x,y\in \mathfrak{h}^*$. These algebras can be regarded as a generalization of Weyl algebras, and their algebraic properties are rather similar. For instance, it is a fundamental result of the theory (known as the PBW Theorem for rational Cherednik algebras), that the canonical filtration of a rational Cherednik algebra (Dunkl-Opdam filtration) yields an isomorphism of graded $K$-algebras:
\begin{equation}\label{equation relations rational Cher algebra}
    G\ltimes \operatorname{Sym}_{\mathbb{C}}(\mathfrak{h}\oplus \mathfrak{h}^*)\rightarrow \gr H_{c}(\mathfrak{h},G).
\end{equation}
In particular, there is a decomposition of $\operatorname{Sym}_{K}(\mathfrak{h}^{*})$-modules:
\begin{equation*}
    H_{c}(\mathfrak{h},G)=\operatorname{Sym}_{\mathbb{C}}(\mathfrak{h}^{*})\otimes_{\mathbb{C}}\mathbb{C}[G]\otimes_{\mathbb{C}}\operatorname{Sym}_{\mathbb{C}}(\mathfrak{h}).
\end{equation*}
This endows $H_{c}(\mathfrak{h},G)$ with a triangular decomposition  (\emph{cf.} \cite[Section 2.1 ]{ginzburg2003category}). Algebras with such a decomposition have a rich representation theory, and it is possible to define a Serre subcategory  $\OX_{c}\subset \Mod(H_{c}(\mathfrak{h},G))$, which is similar to the classical BGG category $\OX$ for reductive Lie algebras. In particular, $\OX_{c}$ is a highest weight category, and its irreducible objects are in one-to-one correspondence with irreducible complex representations of $G$. If $G=\mathbb{W}$ is a Weyl group, then our choice of a reflection function $c\in \operatorname{Ref}(\mathfrak{h},G)$ determines a unique Hecke algebra $\mathcal{H}_c(G)$. The representations of $H_{c}(\mathfrak{h},G)$ and $\mathcal{H}_c(G)$ are related via the $KZ$-functor:
\begin{equation*}
    KZ:\mathcal{\OX}_c\rightarrow \Mod^{\operatorname{f.g.}}(\mathcal{H}_c(G)).
\end{equation*}
This functor is exact, and induces an equivalence of categories:
\begin{equation*}
KZ:\mathcal{\OX}_c/\mathcal{\OX}_{c}^{\operatorname{tor}}\rightarrow \Mod^{\operatorname{f.g.}}(\mathcal{H}_c(G)),
\end{equation*}
where $\mathcal{\OX}_{c}^{\operatorname{tor}}$ is a Serre subcategory of $\mathcal{\OX}_c$. Furthermore, there is a projective object $P_{KZ}\in \OX_c$, with a canonical isomorphism of algebras $\mathcal{H}_c(G)\rightarrow \End_{\OX_c}(P_{KZ})^{\op}$ satisfying that we have an identification of functors:
\begin{equation*}
    KZ=\Hom_{\OX_c}(P_{KZ},-).
\end{equation*}
 It has been known for a while that finite Hecke algebras are an indispensable resource in the (complex) representation theory of reductive groups over finite fields. Examples of this can be found, for example, in the classical works of Howlett-Lehrer \cite{H-L}, and N. Iwahori \cite{Iwahori-Chevalley}. Thus, the discussion above yields a connection between rational Cherednik algebras and arithmetic via Hecke algebras. This will be a common trend in this introduction and, as we will see below, more general versions of Cherednik algebras are related to more general versions of Hecke algebras, and this leads to applications in the  representation theory of reductive algebraic groups.\bigskip

Besides representation theory, rational Cherednik algebras have also been proven to be a powerful tool in algebraic geometry. To wit, rational Cherednik algebras are a special case of symplectic reflection algebras, which were used by P. Etingof and V. Ginzburg in \cite{etingof2002symplectic}  to study resolutions of symplectic quotient singularities on orbifolds via deformation theory. The starting point of their work is the observation that, given a vector space $V$ with an action of a finite group $G$, some of the geometric properties of action of $G$ on $V$ are not correctly captured by the quotient $V/G$. For instance, if $G=\mathbb{Z}/m\mathbb{Z}$ acts on $V=\mathbb{A}^2_{\mathbb{C}}$ by rotations, the fundamental group of the quotient $V/G$, is trivial for every value of $m$. In particular, it is not possible to distinguish the actions of different cyclic groups on $V$
by the finite étale coverings of the quotient $V/G$.\\
The idea then is replacing  $\operatorname{Sym}_{\mathbb{C}}(V^*)^G$, the algebra  of regular functions on $V/G$, by the (non-commutative) skew group algebra $G\ltimes \operatorname{Sym}_{\mathbb{C}}(V^*)$, and trying to infer the algebro-geometric properties of the quotient $V/G$ from the \emph{non-commutative algebraic geometry} of $G\ltimes \operatorname{Sym}_{\mathbb{C}}(V^*)$. In particular, the approach of \cite{etingof2002symplectic} is  studying deformations of $V/G$ through
non-commutative deformations of the skew group-algebra $G\ltimes \operatorname{Sym}_{\mathbb{C}}(V^*)$. In this setting, symplectic reflection algebras appear as deformations of $G\ltimes \operatorname{Sym}_{\mathbb{C}}(V^*)$, and are related to the coordinate ring of a Poisson deformation of
the quotient singularity $V/G$. If $V=\mathfrak{h}\oplus \mathfrak{h}^*$ and $G\leq \operatorname{Gl}(\mathfrak{h})$, then the symplectic reflection algebras obtained in this way are the rational Cherednik algebras described above. Notice that in this situation we have:
\begin{equation*}
    G\ltimes \operatorname{Sym}_{\mathbb{C}}(V^*)=G\ltimes \operatorname{Sym}_{\mathbb{C}}(\mathfrak{h}\oplus \mathfrak{h}^*)=G\ltimes \gr \D(\mathfrak{h}).
\end{equation*}
By the PBW decomposition in $(\ref{equation relations rational Cher algebra})$, this is isomorphic to $\gr H_c(\mathfrak{h},G)$. Hence, we see that rational Cherednik algebras have applications in algebraic geometry, arising from the fact that $H_c(\mathfrak{h},G)$ may be regarded as a non-commutative deformation of the skew-Weyl algebra $G\ltimes \D(\mathfrak{h})$, and both algebras provide quantisations of the skew group algebra of regular functions on the cotangent sheaf of $V$.\bigskip

More generally, P. Etingof extended in \cite{etingof2004cherednik} the construction of Cherednik algebras to all smooth algebraic $\mathbb{C}$-varieties $X$ with a right action of a finite group $G$. These generalized Cherednik algebras are a family of sheaves of filtered associative algebras on $X/G$, which can be seen as a generalization of the skew group algebra of differential operators $G\ltimes \D_X$. Furthermore, the restriction of any Cherednik algebra to the smooth locus of $X/G$ is isomorphic to $G\ltimes \D_X$. As we move on from the linear setting to general smooth varieties, a simple presentation of Cherednik algebras in terms of generators and relations as in equation $(\ref{equation presentation of rational Cherednik algebra})$ is no longer possible. Instead, letting $X^{\operatorname{reg}}$ be the subset of $X$ given by the points with trivial stabilizer
, the Cherednik algebras will be defined as certain subalgebras of $G\ltimes \D_{\omega}(X^{\operatorname{reg}})$, where $\D_{\omega}$ denotes a sheaf of twisted differential operators on $X$. We will postpone the details of this definition until later in the introduction, when we will describe the construction directly in the setting of smooth rigid analytic spaces.\\
One of the main applications
of Cherednik algebras is the construction of geometric invariants of the quotient $X/G$. In particular, the family of Cherednik algebras is a central tool in the construction of flat formal deformations of the orbifold fundamental group of $X/G$. Let us give a few details of this, following \cite{etingof2004cherednik}: The key idea is a generalization of the notion of Hecke algebra to smooth complex varieties with an action of finite group.  Let $X$ be a connected smooth complex variety with an action of a finite group $G$. In analogy with the classical setting of linear actions on a finite-dimensional vector space, we make the following definition:
\begin{defi*}
The braid group of $X/G$ is $B(X/G,x):=\pi_1(X^{\operatorname{reg}}/G,x)$, where $x\in X^{\operatorname{reg}}/G$.   
\end{defi*}
As $X$ is connected, the isomorphism class of $B(X/G,x)$ is independent of the choice of $x\in X^{\operatorname{reg}}/G$, so we will drop the $x$ from the notation. For any $g\in G$, let $X^g$ be the set of points fixed by $g$. 
\begin{defi*}
 A closed subvariety $Y\subset X^g$ is a reflection hypersurface if it is a connected component of $X^g$ which has codimension one in $X$. We let $S(X,G)$ be the set of reflection hypersurfaces of the action of $G$ on $X$, and denote its elements by $(Y,g)$, where $Y\subset X^g$ is a reflection hypersurface.   
\end{defi*}
Notice that $G$ acts on $S(X,G)$ by conjugation, and that $B(X/G)$ only depends on (the complement of) the image of $S(X,G)$ in $X/G$. Let $(Y,g)\in S(X,G)$ be a reflection hypersurface, $G_Y< G$ be the subgroup of elements which fix $Y$, and $n_Y=\vert G_Y\vert$. As a consequence of the Seifert-Van Kampen Theorem, for any $(Y,g)\in S(X,G)$ there is an element in $B(X/G)$ which generates the monodromy around the image of $Y$ in $X/G$. In other words, there is an element in $B(X/G)$ given by the homotopy class of a loop around the image of $Y$ in $X/G$ (See  for example \cite[Appendix 1]{complexrefl} for details). We let $C_Y$ be the conjugacy class in $B(X/G)$ of a generator of the  monodromy around $Y$.
\begin{defi*}
The orbifold fundamental group of $X/G$, denoted by $\pi_1^{\operatorname{orb}}(X/G)$, is the quotient of the braid group $B(X/G)$ by the relations:
    \begin{equation*}
        T^{n_Y}=1, \textnormal{ for all }T\in C_Y, \textnormal{ and } (Y,g)\in S(X,G).
    \end{equation*}
\end{defi*}
We refer to \cite[Definition 1.17]{lecturesorbifolds} for a more detailed account on orbifold fundamental groups. Let us just mention that there is a short exact sequence:
\begin{equation*}
    % https://tikzcd.yichuanshen.de/#N4Igdg9gJgpgziAXAbVABwnAlgFyxMJZABgBpiBdUkANwEMAbAVxiRAEYQBfU9TXfIRQAmclVqMWbADrS0WAPrsAesFkQ0MAE50cELWDoBbGMH0AjLlwAUADQD0AcQCU3XiAzY8BIgGYx1PTMrIggjm58XoJEACwBEsFsnDyRAj4o7PFBUqGy8kp2rlziMFAA5vBEoABmWhBGSJkgekjCKSC19a3ULYi+7Z0NiGTNEEgxA3VDcaONxVxAA
\begin{tikzcd}
1 \arrow[r] & \pi_1(X) \arrow[r] & \pi_1^{\operatorname{orb}}(X/G) \arrow[r] & G \arrow[r] & 1,
\end{tikzcd}
\end{equation*}
which expresses $\pi_1^{\operatorname{orb}}(X/G)$ in terms of $\pi_1(X)$ and $G$. Notice that, unlike the usual notion of fundamental group, the orbifold fundamental group remembers the quotient structure on $X/G$. With these objects at hand, it is time to introduce a version of Hecke algebras adapted to this setting: For any conjugacy class of reflection hypersurfaces $(Y,g)\in S(X,G)$, we introduce the formal parameters $\tau_{Y,1},\cdots,\tau_{Y,n_Y}$, and denote the collection of all these parameters by $\tau$. 
\begin{defi*}[{\cite[Definition 3.3]{etingof2004cherednik}}]
 The Hecke algebra associated to the action of $G$ on $X$, denoted  $\mathcal{H}_{\tau}(X,G)$, is the quotient of the group algebra $\mathbb{C}[B(X/G)][[\tau]]$ by the $\tau$-adically closed ideal generated by the following relations:
\begin{equation*}
    \prod_{j=1}^{n_Y}\left(T-e^{2\pi ij/n_Y}e^{\tau_{Y,j}} \right)=0, \textnormal{ for all }T\in C_Y, \textnormal{ and } (Y,g)\in S(X,G).
\end{equation*}   
\end{defi*} 
Notice that we have an identification of $\mathbb{C}$-algebras:
\begin{equation*}
    \mathcal{H}_{\tau}(X,G)/\tau=\mathbb{C}[\pi_1^{\operatorname{orb}}(X/G)],
\end{equation*}
 so that $\mathcal{H}_{\tau}(X,G)$ is a formal deformation of $\mathbb{C}[\pi_1^{\operatorname{orb}}(X/G)]$. Furthermore, in many situations we have an even stronger result:
\begin{teo*}[{\cite[Theorem 3.7]{etingof2004cherednik}}]
Assume  $\pi_2(X)\otimes_{\mathbb{Z}}\mathbb{Q}=0$. Then $\mathcal{H}_{\tau}(X,G)$ is a flat formal deformation of the group algebra $\mathbb{C}[\pi_1^{\operatorname{orb}}(X/G)]$.
\end{teo*}
The proof crucially uses the fact that Cherednik algebras classify the formal deformations of the skew group algebra of differential operators $G\ltimes \D_X(X)$.\bigskip

Let us now give a few number-theoretic applications of this new (geometric) notion of Hecke algebra: As a first case, when $X$ is a finite-dimensional vector space and $G$ is a Weyl group, the definition of Hecke algebra given above agrees with the usual notion of a Hecke algebra of a Weyl group. As discussed above, this is closely related to the (complex) representation theory of reductive groups over finite fields.\\ 
Next, let $G$ be a simply connected complex simple Lie group, $T\subset G$ be a maximal torus, and $\mathbb{W}(T)$ be its Weyl group. The Hecke algebras $\mathcal{H}_{\tau}(T,\mathbb{W}(T))$ obtained via the procedure above are the affine Hecke algebras. These algebras are a generalization of Iwahori-Hecke algebras, and have been shown to govern the smooth complex representations of  $p$-adic reductive groups. For instance, affine Hecke algebras play a major role in the study of induced cuspidal representations in \cite{Borel-Admiss}, and are essential to our understanding of the decomposition of the category of smooth complex representations of a $p$-adic reductive group in terms of semi-simple types, as showcased in \cite{Bush-Kutz1},\cite{Bush-Kutz2}.\\
Similarly, with $T$ and $\mathbb{W}(T)$ as in the example above, let $R^{\vee}$ be the dual root system of $G$, and let $Q^{\vee}$ be the lattice generated by $R^{\vee}$. Fix a complex elliptic curve $E$. In \cite{Root-SystemsAV}, it is shown that the variety  $X=E\otimes Q^{\vee}$ is an abelian variety, endowed with a canonical action of $\mathbb{W}(T)$. In this case, the Hecke algebra $\mathcal{H}_{\tau}(X,\mathbb{W}(T))$ is the double affine Hecke algebra mentioned at the beginning of the introduction. Thus, we see an emerging pattern, in which Cherednik algebras encode the deformation-theoretic aspects of Hecke algebras, and these algebras are in turn related to the representation theory of reductive groups defined over a finite or $p$-adic field.\bigskip

The goal of this paper is developing $p$-adic analogs to some of the situations described above. Namely, we will define sheaves of $p$-adic Cherednik algebras associated to the action of a finite group $G$ on a smooth rigid analytic space $X$ satisfying mild technical requirements (\emph{cf.} Theorem \ref{teo existence of quotients of rigid spaces by finite groups}). Our hope is that the deformation-theoretic techniques associated to Cherednik algebras can also be applied in the rigid analytic setting. In particular, we expect that the sheaves constructed here can be used to classify the deformations of the skew group algebras of infinite order differential operators $G\ltimes \wideparen{\D}_X(X)$ in some situations, and we see this as a first step towards a deformation theory of co-admissible $G\ltimes\wideparen{\D}$-modules in the sense of Ardakov-Wadsley (\emph{cf.} \cite{ardakov2019}). Furthermore,  it is to be expected that 
the contents of the paper can be generalized to the case where $G$ is a $p$-adic Lie group acting on a smooth rigid analytic space $X$, and we hope that this could lead to new developments in the locally analytic representation theory of  $p$-adic Lie groups. For instance, it would be specially interesting to develop a theory of rational Cherednik algebras for $p$-adic Lie groups.\bigskip

Let us now give an outline of the constructions and main results of the paper: Until the end of the introduction, we let $K$ be a complete discrete valuation field of characteristic zero, and fix a finite group $G$. Let us assume for simplicity that $X=\Sp(A)$ is a smooth affinoid space with an action of $G$, and that there is an étale map $X\rightarrow \mathbb{A}^r_K$. Most of the paper is aimed at the construction of two families of sheaves $H_{t,c,\omega,X,G}$, and $\mathcal{H}_{t,c,\omega,X,G}$ on the quotient space $X/G$. These sheaves are parameterized by units $t\in K^*$, $G$-invariant cohomology classes $\omega \in \operatorname{H}^2_{\operatorname{dR}}(X)^G$, and reflection functions $c\in \operatorname{Ref}(X,G)$ (\emph{cf.} Definition \ref{defi basic components cherednik algebras}).  The sheaves $\mathcal{H}_{t,c,\omega,X,G}$ are called sheaves of $p$-adic Cherednik algebras, and form the main object of study of this paper. We keep the nomenclature from the classical theory, and call the $H_{t,c,\omega,X,G}$ sheaves of Cherednik algebras. The first step is constructing the sheaf of Cherednik algebras
$H_{t,c,\omega,X,G}$, associated to the action of $G$ on $X$. This construction is rather similar  to classical case. The only difficulty is showing that some of the constructions in the algebro-geometric context still make sense for rigid analytic spaces. Next, we want to define the $p$-adic Cherednik algebra:
\begin{equation*}
    \mathcal{H}_{t,c,\omega}(X,G):=\Gamma(X/G,\mathcal{H}_{t,c,\omega,X,G}).
\end{equation*}
The relation of the $p$-adic Cherednik algebra with the Cherednik algebra $H_{t,c,\omega}(X,G)$ discussed above is similar to the relation between the algebra of infinite order differential operators $\wideparen{\D}_X(X)$ defined by Ardakov-Wadsley in \cite{ardakov2019}, and the algebra of differential operators $\D_X(X)$. In particular, $p$-adic Cherednik algebras arise as completions of Cherednik algebras with respect to a countable family of norms. Furthermore, the properties of $\mathcal{H}_{t,c,\omega}(X,G)$ are rather similar to those of $\wideparen{\D}_X(X)$, in the sense that  $\mathcal{H}_{t,c,\omega}(X,G)$ is naturally a Fréchet-Stein algebra, and its elements are in one to one correspondence with rigid analytic functions on the skew group algebra of the cotangent space. That is, there is an isomorphism of Fréchet spaces:
\begin{equation*}
  \mathcal{H}_{t,c,\omega}(X,G) \cong  G\ltimes \OX_X(X)\widehat{\otimes}_K\OX_{\mathbb{A}^r_K}(\mathbb{A}^r_K).
\end{equation*}
We remark that this is in no way an isomorphism of algebras, as elements in $\OX_X(X)$ and $\OX_{\mathbb{A}^r_K}(\mathbb{A}^r_K)$ do not commute in general 
(\emph{cf.} Section \ref{section Completed Cherdnik algebras are Fréchet-Stein}).\bigskip

Let us give a concrete description of this completion process. As mentioned above, Cherednik algebras are determined  by a choice of a unit  $t\in K^*$, a $G$-invariant cohomology class $\omega \in \operatorname{H}^2_{\operatorname{dR}}(X)^G$, and a reflection function $c\in \operatorname{Ref}(X,G)$. In order to further simplify the situation, we will assume that $\omega=0$, and choose $t\in K^*$, $c\in \operatorname{Ref}(X,G)$. The general case leads to considerations regarding group actions on sheaves of infinite order twisted differential operators, which we will postpone until Section \ref{section 2.4}.\\
The geometric notion of Cherednik algebra we are about to introduce is defined in terms of the so-called Dunkl-Opdam operators. These operators have a rather convoluted expression (\emph{cf.} Definition \ref{Basic definitions in Cher algebras}). In order to make their expression simpler, we will also assume that every reflection hypersurface $(Y,g)\in S(X,G)$ satisfies that $Y$ is the zero locus of a rigid function $f_{(Y,g)}\in \OX_X(X)$. We remark that given a smooth affinoid space $X$ equipped with a $G$-action, there is always an admissible affinoid cover of $X$ by $G$-invariant affinoid subdomains, each of which satisfies the assumptions we have made thus far (\emph{cf.} Proposition \ref{G-invariant coverings}). We define the Cherednik algebra $H_{t,c,\omega}(X,G)$ as the subalgebra of $G\ltimes \D_X(X^{\operatorname{reg}})$ generated by $G$, $\OX_X(X)$, and the (standard) Dunkl-Opdam operators: 
\begin{equation*}
    D_{v}=t\mathbb{L}_{v}+\sum_{(Y,g)\in S(X,G)}\frac{2c(Y,g)}{1-\lambda_{Y,g}}\frac{v(f_{(Y,g)})}{f_{(Y,g)}}(g-1),
\end{equation*}
where $v$  is a vector field on $X$. The operators $\lambda_{Y,g}\in K\setminus \{1\}$ are the eigenvalues of the action of $g$ on the conormal bundle of $Y$  (\emph{cf}.  Proposition \ref{constant-eigenfunctions}). The idea is defining $\mathcal{H}_{t,c,\omega}(X,G)$ as the closure of $H_{t,c,\omega}(X,G)$ in $G\ltimes \wideparen{\D}_X(X^{\operatorname{reg}})$. However, the analytical structure of $\wideparen{\D}_X(X^{\operatorname{reg}})$ is not clear, and this definition makes it hard to work with $\mathcal{H}_{t,c,\omega}(X,G)$. In order to solve this, we notice that any big enough $G$-invariant affinoid space $U\subset X^{\operatorname{reg}}$ satisfies:
\begin{equation*}
    H_{t,c,\omega}(X,G)\subset G\ltimes \D_X(U).
\end{equation*}
Furthermore, under our current assumptions, we can use the constructions in \cite{ardakov2019} to produce an explicit Fréchet-Stein presentation of $\wideparen{\D}_X(U)$. This allows us to obtain a characterization of the closure of $H_{t,c,\omega}(X,G)$ in $G\ltimes \wideparen{\D}_X(U)$, which we denote $\mathcal{H}_{t,c,\omega}(X,G)_U$. The crux of the matter is showing that, for big enough $U\subset X^{\operatorname{reg}}$, the algebra $\mathcal{H}_{t,c,\omega}(X,G)_U$ is independent of the choice of  $U$. This is a delicate matter, which we are able to solve by studying the relation between  $\mathcal{H}_{t,c,\omega}(X,G)_U$
and the Shilov boundary of  $X$ (\emph{cf}. Section \ref{section restriction rings}). Our main theorem is as follows:
\begin{intteo}
Let $U\subset X^{\operatorname{reg}}$, and assume $U$ contains the Shilov boundary of $X$. Then  $\mathcal{H}_{t,c,\omega}(X,G)_U$ is a Fréchet-Stein algebra, and is independent of $U$.
\end{intteo}
\begin{proof}
This is shown in Theorems \ref{teo simplification of definition of completed cherednik algebras}, and \ref{teo complete Cher are F-S}.
\end{proof}
Thus, dropping the $U$ from the notation,  we obtain a definition of $\mathcal{H}_{t,c,\omega}(X,G)$.\\
Notice that this theorem allows us to define $p$-adic Cherednik algebras for any $G$-invariant affinoid subdomain $Y\subset X$. Furthermore, the family of $p$-adic Cherednik algebras for varying $Y\subset X$ can be arranged in a canonical way into a sheaf on the quotient space $X/G$. By working in a slightly more general setting, we can even show that they extend to a sheaf on the small étale site $X/G_{\textnormal{ét}}$. This sheaf is defined on étale and affinoid maps $Y\rightarrow X/G$ by the rule:
\begin{equation*}
    \mathcal{H}_{t,c,\omega,X,G}(Y,G):= \mathcal{H}_{t,c,\omega}(Y\times_{X/G}X,G),    
\end{equation*}
and we define the sections of $H_{t,c,\omega,X,G}$ analogously. The restriction maps of $\mathcal{H}_{t,c,\omega,X,G}$ are obtained by continuously extending those of $H_{t,c,\omega,X,G}$, and these maps are, in turn, uniquely determined by restriction of the canonical maps:
\begin{equation*}
    G\ltimes \D_X(Y^{\operatorname{reg}})\rightarrow G\ltimes \D_X(Z^{\operatorname{reg}}),
\end{equation*}
where $Z\rightarrow Y\rightarrow X$ are $G$-equivariant étale maps. Sheaves of $p$-adic Cherednik algebras satisfy the following properties:
\begin{intteo}
Choose $t\in K^*$, $\omega \in \operatorname{H}^2_{\operatorname{dR}}(X)^G$, and $c\in \operatorname{Ref}(X,G)$.  The sheaf of $p$-adic Cherednik algebras $\mathcal{H}_{t,c,\omega,X,G}$  is the unique sheaf on $X/G_{\textnormal{ét}}$ satisfying:
\begin{enumerate}[label=(\roman*)]
    \item Let $Y\rightarrow X/G$ be an étale and affinoid map. Then we have:
\begin{equation*}
    \mathcal{H}_{t,c,\omega,X,G}(Y,G):= \mathcal{H}_{t,c,\omega}(Y\times_{X/G}X,G).    
\end{equation*}
\item Let $Z\rightarrow Y\rightarrow X/G$ be étale and affinoid maps. Then the restriction map:
\begin{equation*}
    \mathcal{H}_{t,c,\omega,X,G}(Y,G)\rightarrow \mathcal{H}_{t,c,\omega,X,G}(Z,G),
\end{equation*}
is a continuous morphism of Fréchet-Stein $K$-algebras.
\item There is an injective map of sheaves of $K$-algebras $H_{t,c,\omega,X,G}\rightarrow \mathcal{H}_{t,c,\omega,X,G}$.
\end{enumerate}
Furthermore, the sheaves $\mathcal{H}_{t,c,\omega,X,G}$ have trivial higher \v{C}ech cohomology.
\end{intteo}
\begin{proof}
This is shown in Theorem \ref{teo sheaves of complete Cherednik algebras for general $G$-varieties}.
\end{proof}
We remark that, in order to make sense of this definition, one needs to show that every reflection function $c\in \operatorname{Ref}(X,G)$ extends in a unique way to a \emph{sheaf of reflection functions }on $X/G_{\textnormal{ét}}$. We shall give more details on this construction in the body of the paper (\emph{cf}. Section \ref{Section 3.1}). Finally, let us mention that the sheaves $\mathcal{H}_{t,c,\omega,X,G}$ generalize $\wideparen{\D}_X$. In particular, by setting $G=1$ we get the sheaves of infinite order twisted differential operators $\wideparen{\D}_{\omega}$, and setting $\omega=0$ we recover  $\wideparen{\D}_X$. 
 \subsection*{Future and related work}
This paper is part of a series of papers \cite{HochschildDmodules},\cite{HochschildDmodules2}, \cite{p-adicCatO} in which we start a study of the deformation theory of algebras $G\ltimes \wideparen{\D}_X(X)$, where $G$ is a finite group acting on a smooth Stein space $X$. In particular, in \cite{HochschildDmodules} we develop a formalism of Hochschild cohomology and homology for $\wideparen{\D}_X$-modules on smooth and separated rigid analytic spaces, and give explicit calculations of the Hochschild cohomology groups $\operatorname{HH}^{\bullet}(\wideparen{\D}_X)$ in terms of the de Rham cohomology of $X$. This is done in the setting of sheaves of Ind-Banach spaces and quasi-abelian categories, as developed in \cite{bode2021operations}. These results are further extended in \cite{HochschildDmodules2}, where we focus on the case where $X$ is a smooth Stein space. Our interest in Stein spaces stems from the fact that they behave like affinoid spaces with respect to sheaf cohomology, but are better behaved analytically. In particular, if  $X$ is a smooth Stein space, then the de Rham complex:  
\begin{equation*}
    \Omega_{X/K}^{\bullet}(X):= \left( 0\rightarrow \OX_X(X)\rightarrow \Omega_{X/K}^1(X)\rightarrow \cdots \rightarrow \Omega_{X/K}^{\operatorname{dim}(X)}\rightarrow 0 \right),
\end{equation*}
is a strict complex of nuclear Fréchet spaces. It can be shown that this implies that $\operatorname{HH}^{\bullet}(\wideparen{\D}_X(X))$ is also a strict complex of nuclear Fréchet spaces, and this property descends to the Hochschild cohomology spaces. In this setting, we can give a more tangible meaning to Hochschild cohomology. In particular, it is shown in 
\cite{HochschildDmodules2} that for every $n\geq 0$ the vector space underlying 
$\operatorname{HH}^{n}(\wideparen{\D}_X)$ classifies the $n$-th degree Yoneda extensions of $I(\wideparen{\D}_X(X))$  by $I(\wideparen{\D}_X(X))$  as $I(\wideparen{\D}_X(X))^e$-modules.\\
Following this line of thought, we then show that we can use the bar resolution to calculate the Hochschild cohomology of $ \wideparen{\D}_X$ in low degrees, and this allows us to compare  our notion of Hochschild cohomology  with J. Taylor's Hochschild cohomology for locally convex algebras (\emph{cf}. \cite{helemskii2012homology}, \cite{taylor1972homology}).  The advantage of this approach is that we obtain explicit interpretations of many Hochschild cohomology spaces in terms of algebraic invariants of $\wideparen{\D}_X(X)$. Furthermore, our results reinforce the heuristic that the spaces $\operatorname{HH}^{\bullet}(\wideparen{\D}_X(X))$ should be regarded as  $p$-adic analytic versions of the classical Hochschild cohomology groups of associative algebras. For instance, the first Hochschild cohomology space $\operatorname{HH}^{1}(\wideparen{\D}_X(X))$ parametrizes the bounded outer derivations of 
$\wideparen{\D}_X(X)$, whereas the classical first Hochschild cohomology group parametrizes arbitrary outer derivations. Finally, we use these results to obtain infinitesimal deformations of $\wideparen{\D}_X(X)$.\\

In the representation-theoretic direction, the results of this paper are continued in \cite{p-adicCatO}, where we study an analog of the category $\OX$ for $p$-adic rational Cherednik algebras. In particular, let $\mathfrak{h}$ be a finite-dimensional $K$-vector space, $G< \operatorname{Gl}(\mathfrak{h})$ a finite group, and $H_{c}(\mathfrak{h},G)$ be a rational Cherednik algebra. As shown in \cite{ginzburg2003category}, one can attach to $H_{c}(\mathfrak{h},G)$ a certain abelian subcategory $\OX_{c}\subset \operatorname{Mod}(H_{c}(\mathfrak{h},G))$, which has many similarities with the classical BGG category $\OX$ for finite-dimensional reductive Lie algebras. In particular, $\OX_{c}$ is a highest weight category, and its irreducible objects are in one-to-one correspondence with irreducible representations of $G$  with coefficients in $\mathbb{C}_p$. In \cite{p-adicCatO}, we define a $p$-adic version of the category $\OX_c$ for $p$-adic rational Cherednik algebras. In particular, let $\mathfrak{h}^{\operatorname{an}}$ be the rigid analytification of $\mathfrak{h}$, and let $\mathcal{H}_c(\mathfrak{h}^{\operatorname{an}},G)$ be a $p$-adic rational Cherednik algebra. In the spirit of \cite{Schmidt2010VermaMO}, we define a $p$-adic version of the category $\OX_c$, which we denote by $\wideparen{\OX}_c$, and show that extension of scalars along the map $H_c(\mathfrak{h},G)\rightarrow\mathcal{H}_c(\mathfrak{h}^{\operatorname{an}},G)$ induces an equivalence of abelian categories $\OX_c\rightarrow \wideparen{\OX}_c$. Thus, obtaining a framework in which rigid-analytic techniques can be used to study the representation theory of finite groups.
\subsection*{Notation and Conventions}
For the rest of this text, we let $K$ be a complete discrete valuation field of characteristic zero, we denote its valuation ring by $\mathcal{R}$, fix a uniformizer $\pi$, and denote its residue field by $k$. For a $\mathcal{R}$-module $\mathcal{M}$ we will often times
 write $\mathcal{M}_K=\mathcal{M}\otimes_{\mathcal{R}}K$. Throughout the text, we will make extensive use of the theory of $\wideparen{\D}$-modules on smooth rigid analytic spaces, as developed in \cite{ardakov2019},\cite{ardakov2015d}, \cite{Ardakov_Bode_Wadsley_2021}. In particular, the notion of Lie Rinehart algebra (\emph{cf.} \cite[Section 2]{ardakov2019}), and the theory of co-admissible modules over Fréchet-Stein algebras (\emph{cf.} \cite{schneider2002algebras}) will be used without further mention. If $A=\varprojlim_n A_n$ is a Fréchet-Stein algebra, we will denote its category of co-admissible modules by $\mathcal{C}(A)$.
\subsection*{Acknowledgments}
This paper was written as a part of a PhD thesis at the Humboldt-Universität zu Berlin under the supervision of Prof. Dr. Elmar Große-Klönne. I would like to thank Prof. Große-Klönne
for pointing me towards this exciting topic, and reading an early draft of the paper.

\bigskip
Funded by the Deutsche Forschungsgemeinschaft (DFG, German Research
Foundation) under Germany´s Excellence Strategy – The Berlin Mathematics
Research Center MATH+ (EXC-2046/1, project ID: 390685689).
 \section{Twisted differential operators on rigid analytic spaces}\label{section Twisted differential operators on Rigid analytic spaces}
 In this chapter we will establish the basic tools needed for the definition of $p$-adic Cherednik algebras. We start by analyzing some geometric aspects of the action of a finite group $G$ on a smooth rigid analytic variety $X$. Later, we will define Atiyah algebras (Picard algebroids) on $X$, and use them to define sheaves of twisted differential operators on $X$. We study Fréchet-Stein completions of sheaves of twisted differential operators, and show that the basic results of \cite{ardakov2019} still hold in this setting. We conclude by studying actions of $G$ on the sheaves of infinite order twisted differential operators, and on their categories of co-admissible modules. 
\subsection{Basic constructions}\label{Section 1.1}
In this section, we will develop the rigid-geometric tools needed for the definition of Cherednik algebras in the rigid analytic setting.\bigskip

As in the classical case, we will ultimately define Cherednik algebras as sheaves on the quotient of $X$ by the action of $G$. Thus, in order to have a consistent theory, we need to understand the conditions under which there is a  rigid analytic quotient $X/G$. The most important result in this direction is the following theorem :
 \begin{teo}[{\cite{hansen2016quotients}}]\label{teo existence of quotients of rigid spaces by finite groups}
 Let $X$ be a rigid analytic space with a right action by a finite group $G$. Assume that this action satisfies the following property: 
 \begin{equation*}
     \textnormal{(G-Aff)}:= \textnormal{There is an admissible cover of } X \textnormal{ by } G\textnormal{-stable affinoid spaces.}
 \end{equation*}
Let $\operatorname{Rig}_K$ be the category of rigid spaces over $K$. The following hold:
 \begin{enumerate}[label=(\roman*)]
     \item The action of $G$ on $X$ admits a categorical quotient $X/G$ in $\operatorname{Rig}_K$.
     \item The canonical projection:
     \begin{equation*}
         \pi: X\rightarrow X/G,
     \end{equation*}
     is finite. Furthermore, $X$ is affinoid if and only if $X/G$ is affinoid.
     \item If $X=\Sp(A)$, then $X/G=\Sp(A^G)$.
     \item The preimage along the projection $\pi:X\rightarrow X/G$ induces a bijection:
     \begin{equation*}
 \pi^{-1}:\left\{ 
				\begin{array}{c} 
					\textnormal{admissible open }\\ 
					\textnormal{subspaces of } X/G
				\end{array}
\right\} \cong \left\{
				\begin{array}{c}
				 G-\textnormal{stable admissible open}\\
				
				  \textnormal{subspaces of } X
				\end{array}
\right\},
\end{equation*}
     satisfying that for every admissible open $U\subset X/G$ we have $U=\pi^{-1}(U)/G$.
 \end{enumerate}
 \end{teo}
 We remark that the result shown in \cite{hansen2016quotients} is much more general than the versions stated here. However, this version is enough for our purposes. Let us now  introduce the following terminology:
\begin{defi}
Let $G$ be a finite group. A $G$-variety is a rigid analytic variety $X$ equipped with a faithful right $K$-linear action of $G$ which satisfies $\operatorname{(G-Aff)}$.
\end{defi}
Let $X$ be a $G$-variety. Our next goal is  describing certain geometric aspects of the action of $G$ on $X$. Namely, we will show that the fixed point locus $X^G$ is a closed smooth subvariety of $X$, introduce reflection hypersurfaces, and show that $G$-varieties admit certain covers with good properties. As most of the material is well-known for algebraic varieties over characteristic zero, we will omit the  proofs of most of the results.   Let us start with the following definition:
\begin{defi}\label{defi fixed locus of G}
Let $X$ be a $G$-variety. We define $X^{G}$ to be the functor of points fixed by the action of $G$. That is, for every rigid space $S$ we have:
    \begin{equation*}
        X^{G}(S)=\{ f:S\rightarrow X \textnormal{ such that } gf=f \textnormal{ for every } g\in G\}.
    \end{equation*}
\end{defi}
Notice that  $X^{G}$ fits into the following cartesian diagram:
\begin{center}
         \begin{tikzcd}
X^{G} \arrow[r, hook] \arrow[d, hook] & X \arrow[d, "\prod\operatorname{Id}"]        \\
X \arrow[r, "\prod g"]       & \displaystyle\prod_{g\in G} X
\end{tikzcd}
    \end{center}
Thus, the functor $X^G$ is representable. Furthermore, as $X$ satisfies $\operatorname{(G-Aff)}$, the lower horizontal map is a closed immersion, and so  the pullback $X^G\rightarrow X$ is also a closed immersion. Furthermore, the following additional result holds: 
\begin{prop}\label{smooth locus}
Let $X$ be a smooth rigid analytic $G$-variety. Then, the fixed point locus $X^{G}$ is a smooth closed subvariety of $X$. In particular, if $g$ is an automorphism of finite order of $X$, then $X^{g}$ is a closed smooth subvariety of $X$.
\end{prop}
Next, we will introduce some geometric objects attached to smooth $G$-varieties. We start with reflection hypersurfaces and their associated eigenvalues:
\begin{defi}\label{defi reflection hypersurfaces}
Let $X$ be a $G$-variety. A reflection hypersurface on $X$ is a pair $(Y,g)$, where $g\in G$ and $Y$ is a connected component of $X^{g}$ satisfying that:
\begin{equation*}
    \operatorname{codim}(Y):=\operatorname{dim}(X)-\operatorname{dim}(Y)=1.
\end{equation*}
We denote the set of all reflection hypersurfaces of the action of $G$ on $X$ by $S(X,G)$.
\end{defi}
Let $X$ be a $G$-variety and choose $(Y,g)\in S(X,G)$. We let $\mathcal{I}_Y$ be the sheaf of ideals associated to $Y$, and let $\mathcal{N}_{Y}^{*}=\mathcal{I}_{Y}/\mathcal{I}_{Y}^{2}$ be the conormal bundle to $Y$. As $Y$ has codimension one in $X$, it follows that $\mathcal{I}_Y$ is a line bundle on $X$. Consequently,  $\mathcal{N}_{Y}^{*}$ is a line bundle on $Y$. As the action of $g:X\rightarrow X$ leaves $Y\subset X^g$ fixed, we have $g(\mathcal{I}_Y)\subset \mathcal{I}_Y$. Hence, we get an $\OX_Y$-linear action $g:\mathcal{N}_{Y}^{*}\rightarrow \mathcal{N}_{Y}^{*}$. 
\begin{defi}
The map $g:\mathcal{N}_{Y}^{*}\rightarrow \mathcal{N}_{Y}^{*}$ is determined by a unique function $\lambda_{Y,g}\in \Gamma(Y,\OX_Y)$. We call $\lambda_{Y,g}$ the eigenvalue of $(Y,g)$.  
\end{defi}
The following proposition justifies the nomenclature used for $\lambda_{Y,g}$:
\begin{prop}\label{constant-eigenfunctions}
Let $X$ be a smooth rigid $G$-variety. For any reflection hypersurface $(Y,g)\in S(X,G)$, the eigenvalue $\lambda_{Y,g}$ is a  constant function on $Y$. In particular, it is a root of unity in $K$. Furthermore, assume $X$ is connected and let $Y\subset X$ be a connected hypersurface. We define the following subgroup:
 \begin{equation*}
     G_Y=\{ g\in G  \textnormal{ }\vert \textnormal{ } Y\subset X^g \}.
 \end{equation*}
 Then $G_Y$ is cyclic and acts $\OX_Y$-linearly and faithfully  on the conormal bundle $\mathcal{N}_Y^*$.
\end{prop}
We finish our preliminaries on $G$-varieties by showing that every $G$-variety admits an admissible cover with particularly good properties. This will be instrumental for our definition of $p$-adic Cherednik algebras in later sections.
\begin{Lemma}\label{Lemma free vector bundles on a finite amount of points}
    Let $X=\Sp(A)$ be an affinoid space with a vector bundle $\mathcal{V}$. For any finite set $S\subset X$ there is an open subspace $S\subset U\subset X$ such that $\mathcal{V}_{\vert U}$ is free. Furthermore, $U$ can be taken to be a  Zariski or affinoid open subspace of $X$.
\end{Lemma}
\begin{proof}
  For each $s\in S$ we can find rigid functions $f_s\in A$ such that $f_s(s')=\delta_{s,s'}$ for $s'\in S$. We let $\mathcal{V}_s=\mathcal{V}(X)\otimes_{\OX_X(X)}k(s)$ denote the residue of $\mathcal{V}$ at $s$, and choose 
  a set of sections $\sigma_{1,s},\dots, \sigma_{n,s}$ on $\mathcal{V}(X)$ which give a basis of the residue $\mathcal{V}_s$. Consider the following sums:
  \begin{equation*}
     \sigma_{k}= \sum_{s\in S}f_s\sigma_{k,s}.
  \end{equation*}
  For each $s'$, the image of $\sigma_{k}$ in $\mathcal{V}_{s'}$ is $ (\sigma_{k})_{s'} = \sum_{s\in S}f_s(s')(\sigma_{k,s})_{s'}=(\sigma_{k,s'})_{s'}$. Hence, it follows that $\sigma_1,\cdots,\sigma_n$ are a family of sections such that their residue classes form a basis of $\mathcal{V}_{s}$ for each $s\in S$. Consider the morphism $\varphi:A^{n}\rightarrow \mathcal{V}(X)$ given by sending $e_{i}$ to $\sigma_{i}$. As both modules are finite, the locus in which $\varphi$ is an isomorphism is a Zariski open $V\subset X$. In particular, $\mathcal{V}$ is free when restricted to $V$, and our calculations above show that $S\subset V$.  As every Zariski open of $X$ admits an admissible cover $V=\cup_{n\geq 0} U_n$ by increasing affinoid subdomains of $X$, it follows that $S\subset U_n$ for big enough $n\geq 0$.
\end{proof}
\begin{prop}\label{G-invariant coverings}
Let $X$ be a smooth $G$-variety . There is an admissible cover $\mathfrak{U}$ of $X$ such that every $U\in\mathfrak{U}$ satisfies the following properties:
\begin{enumerate}[label=(\roman*)]
    \item $U$ is a $G$-invariant admissible affinoid open.
    \item The module of Kähler differentials $\Omega^1_{X/K}(U)$ is a free $\OX_{X}(U)$-module with a basis given by differentials of functions.
    \item Any reflection hypersurface on $U$ is given by the vanishing locus of a single rigid function on $U$.
\end{enumerate}
\end{prop}
\begin{proof}
By definition of $G$-variety, we can assume that $X=\Sp(A)$ is affinoid. Let $x\in X$ be a point. As $G$ is a finite group, the orbit $\textnormal{Orb}(x)=\{ g(x)\}_{g\in G}$ is a finite subset of $X$. Denote this subset by $S$. As $X$ is smooth, every reflection hypersurface $(Y,g)$  is associated to a sheaf of ideals $\mathcal{I}_{(Y,g)}$ which is a line bundle. As $\Omega^1_{X/K}$ is a vector bundle as well, we can use Lemma \ref{Lemma free vector bundles on a finite amount of points} to construct a Zariski open $U\subset X$ which contains $S$ and satisfies conditions $(ii)$ and $(iii)$. As $X$ is affinoid, it follows that $V=\cap_{g\in G}g(U)$ is a $G$-invariant Zariski open of $X$ which contains $x$ and satisfies conditions $(ii)$ and $(iii)$. The previous procedure yields a $G$-invariant Zariski open $V_{x}\subset Y$ satisfying conditions $(ii)$ and $(iii)$. Thus, as Zariski covers are always admissible, we obtain an admissible cover $\mathfrak{U}$ of $X$ such that its elements satisfy conditions $(i)$ to $(iii)$ in the definition, but are Zariski instead of admissible. But then we can apply Theorem \ref{teo existence of quotients of rigid spaces by finite groups} to refine this cover to a cover satisfying our requirements.
\end{proof}
Finally, we will now  give rigid analytic versions of the sheaf  $\OX_{X}(Z)$ and the residue morphism from \cite[Section 2.4]{etingof2004cherednik}. 
\begin{Lemma}\label{residue ses}
Let $X$ be a smooth and separated rigid space with a hypersurface $Z$. There is a line bundle $\OX_{X}(Z)$ on $X$ such that on every affinoid open $U\subset X$ the sections $\OX_{X}(Z)(U)$ are the meromorphic functions in $U$ with a pole of order at most one in $Z$. Furthermore, there is a short exact sequence of coherent modules:
\begin{equation*}
    0\rightarrow \OX_{X}\rightarrow \OX_{X}(Z)\rightarrow \mathcal{N}_{Z}\rightarrow 0,
\end{equation*}
where $\mathcal{N}_{Z}=\mathcal{H}om_{\OX_{Z}}(\mathcal{I}_{Z}/\mathcal{I}_{Z}^{2},\OX_{Z})$ is the normal bundle to $Z$ in $X$.
\end{Lemma}
\begin{proof}
As $X$ is separated, it is enough to show  the case where $X=\Sp(A)$ is affinoid and $Z$ has associated sheaf of ideals $\mathcal{I}_Z$ generated by a single rigid function $f\in A$. Let $\mathcal{M}_{X}$ be the sheaf of meromorphic functions on $X$, as defined in \cite[Section 4.6]{Fresnel2004}. Then we have an injective morphism of $A$-modules:
\begin{equation*}
    m(f^{-1}):=\OX_{X}(X)\rightarrow \mathcal{M}_{X}(X), \textnormal{ } h\mapsto \frac{h}{f}.
\end{equation*}
This lifts to a morphism of $\OX_{X}$-modules, $m(f^{-1}):=\OX_{X}\rightarrow \mathcal{M}_{X}$, and we let $\OX_{X}(Z)=\operatorname{Im}(m(f^{-1}))$. Notice that for any two rigid functions $f,g\in A$ which generate $\mathcal{I}_Z$, there is a unique unit $\lambda \in A$ such that $f=\lambda g$. Thus, the sheaf $\OX_{X}(Z)$ is independent of the choice of generator for $\mathcal{I}_Z$. Furthermore, the sections of $\OX_{X}(Z)$ on affinoid subdomains are meromorphic functions with a pole of order at most one in $Z$. We also have the following identity:
\begin{equation*}
    \mathcal{N}_{Z}= \mathcal{H}om_{\OX_{Z}}(\mathcal{I}_{Z}/\mathcal{I}_{Z}^{2},\OX_{Z})=\Hom_{A/I}(f\cdot A/f^{2}\cdot A,A/f)=\Hom_{A}(f\cdot A,A/f),
\end{equation*}
where we are using the identification between finitely generated $\OX_Z(Z)$-modules and coherent $\OX_Z$-modules. We define the morphism $\OX_{X}(Z)\rightarrow \mathcal{N}_{Z}$ by sending an element $a/f\in \OX_{X}(Z)$ to the following morphism of $A$-modules:
\begin{equation*}
    a/f:f\cdot A \rightarrow A/f, \textnormal{ } fb\mapsto [ab].
\end{equation*}
The map  $\OX_{X}(Z)\rightarrow \mathcal{N}_{Z}$ is clearly surjective, and its kernel is given by elements $a/f\in \OX_{X}(Z)$ such that $a=fh$ for some $h\in A$. In particular, we have $a/f=h\in A$, so we get a short exact sequence:
 \begin{equation*}
     0\rightarrow A\rightarrow A/f\rightarrow \Hom_{A}(f\cdot A,A/f)\rightarrow 0.
 \end{equation*}
 Now we just need to show that these constructions can be glued together. That is, we need to show they do not depend on the choice of a generator of the sheaf of ideals induced by $Z$. This follows from the fact that if $f$ and $g$ are generators of $\mathcal{I}_{Z}$, then there is a unique unit $\lambda$ in $A$ such that $g=\lambda f$.  Thus, one can repeat the process above with both $f$ and $g$, and multiplication by $\lambda$ yields an isomorphism between the two exact sequences. The fact that $\lambda$ is unique implies that the cocycle condition is satisfied, so this construction can be globalized.
\end{proof}
Now that we have defined $\OX_{X}(Z)$, we can construct the residue morphism: 
\begin{prop}\label{residue map}
 Let $X$ be a smooth and separated rigid $K$-variety with a hypersurface $Z\subset X$. There is a unique morphism of coherent modules: 
 \begin{equation*}
     \xi_{Z}:\mathcal{T}_{X/K}\rightarrow \OX_{X}(Z)/\OX_{X}\cong \mathcal{N}_{Z},
 \end{equation*}
 such that on every affinoid open $U\subset X$  satisfying that $U\cap Z$ is defined as the vanishing locus of a rigid function $f\in \OX_{X}(U)$, the morphism is given on any vector field $v\in\mathcal{T}_{X/K}(U)$ by the formula:
 \begin{equation*}
     \xi_{Z}(v)=\left[\frac{v(f)}{f}\right].
 \end{equation*}
\end{prop}
\begin{proof}
Assume first that $X=\Sp(A)$ is such that $Z$ is the vanishing locus of a single rigid function $f$ on $A$. In this situation, notice that by Lemma \ref{residue ses} we have an isomorphism $\varphi:\OX_{X}(Z)/\OX_{X}\rightarrow \mathcal{N}_{Z}=\Hom_{A}(f\cdot A,A/f)$. Consider a morphism of $A$-modules $\xi : f\cdot A\rightarrow A/f$. This morphism is uniquely determined by an element 
$\xi(f)\in A/f$. Let $\overline{\xi(f)}$ be a representative of $\xi(f)$ in $A$. By construction of $\varphi$, we have that $\varphi^{-1}(\xi)=\left[\frac{\overline{\xi(f)}}{f}\right]$. Recall from \cite[Proposition 1.6.3.ii]{huber2013etale} the canonical right exact sequence of Kähler differentials associated to the closed immersion $i:Z\rightarrow X$. Dualizing this sequence, we get a left exact sequence:
\begin{equation*}
    0\rightarrow \mathcal{T}_{Z/K}\rightarrow i^{*}\mathcal{T}_{X/K}\rightarrow \mathcal{N}_{Z}.
\end{equation*}
In particular, there is an $A$-linear morphism $ \mathcal{T}_{X/K}(X)\rightarrow i^{*}\mathcal{T}_{X/K}(Z)\rightarrow \mathcal{N}_{Z}(Z)$. This sends a global derivation $v\in\mathcal{T}_{X/K}(X) $ to the morphism of $A$-modules:
\begin{equation*}
    v:f\cdot A\rightarrow A/f, \,fa \mapsto [v(fa)]. 
\end{equation*}
Hence, we get a morphism $\xi_{Z}:\mathcal{T}_{X}(X)\rightarrow \OX_{X}(Z)(X)/\OX_{X}(X)$, defined on any $v\in\mathcal{T}_{X/K}(U)$ by the formula $\xi_{Z}(v)=\left[\frac{v(f)}{f}\right]$. As all morphisms involved in this construction are defined globally, it follows that the map can be defined for the general smooth and separated $X$.
\end{proof}
\begin{defi}
Let $X$ be a smooth and separated rigid space with a hypersurface $Z$. The map $\xi_{Z}:\mathcal{T}_{X/K}\rightarrow \OX_{X}(Z)/\OX_{X}$ is called the residue map of $Z$.
\end{defi}
In later stages, the residue morphism will be used to define certain generators of the Cherednik algebras. To wit, the residue morphism is an essential ingredient in the definition of the  Dunkl-Opdam operators. 
\subsection{Classification of Atiyah algebras on smooth rigid analytic spaces}\label{Section 2.3}
In this section we will introduce sheaves of infinite order twisted differential operators on smooth rigid spaces. The goal is showing that the theory of \cite{ardakov2019} extends naturally to the twisted setting.
In particular, we will define sheaves of completed $p$-adic twisted differential operators, and show twisted versions of the results in \cite{ardakov2019}. Along the way, we will give a classification of the sheaves of twisted differential operators on a smooth rigid analytic space $X$.\bigskip

Let us start by recalling the definition of an Atiyah algebra:
\begin{defi}[{\cite[Definition 2.1.3]{beilinson1993proof}}]\label{defi Atiyah algebra}
An Atiyah algebra (Picard algebroid) on a smooth rigid space $X$, is a Lie algebroid $(\mathcal{A},[-,-])$, with a short exact sequence:
\begin{equation*}
    0\rightarrow \OX_{X}\rightarrow \mathcal{A}\rightarrow \mathcal{T}_{X/K}\rightarrow 0,
\end{equation*}
satisfying that $\mathcal{A}\rightarrow \mathcal{T}_{X/K}$ is the anchor map of the Lie-Rinehart algebra structure, and that for every affinoid open $U\subset X$, the unit $1\in \OX_{X}(U)$ is  central in $\mathcal{A}(U)$. We will usually denote this element by $1_{\mathcal{A}}$. A morphism of Atiyah algebras is a morphism of Lie algebroids which preserves $1$. We denote the category of Atiyah algebras on $X$ by $\mathscr{PA}(X)$.
\end{defi}
\begin{obs}
Let $A$ be a commutative $K$-algebra. We can define an Atiyah algebra on $A$ as a $(K,A)$-Lie algebra $\mathcal{A}$ satisfying the conditions on Definition \ref{defi Atiyah algebra}, and similarly for morphisms. We call the corresponding category $\mathscr{PA}(A)$. 
\end{obs}
We will start by constructing some examples of Atiyah algebras. Assume first that $X$ is affinoid, and let $\mathcal{L}$ be a line bundle on $X$. By \cite[Lemma 9.2]{ardakov2019}, Atiyah algebras are completely determined by their global sections. In particular, every Atiyah algebra on $\OX_X(X)$ lifts to an Atiyah algebra on $X$. Let $A=\OX_X(X)$ and $L=\mathcal{L}(X)$. As $\mathcal{L}$ is a line bundle on $X$, we have $\Hom_A(L,L)=A$. In particular, we may regard $A$ as a subalgebra of $\Hom_K(L,L)$, and we consider $\Hom_K(L,L)$ as a $K$-Lie algebra with respect to the commutator bracket. We define the $(K,A)$-Lie algebra of first order differential operators on $\mathcal{L}$ as follows:
\begin{equation*}
    \operatorname{Diff}_{A}(L,L)=\{ f\in \Hom_{K}(L,L) \textnormal{ } \vert \textnormal{ }  [a,f]:=af-fa \in A \textnormal{ for all }a\in A \}.
\end{equation*}
The bracket on $\operatorname{Diff}_{A}(L,L)$ is given by the commutator, and the anchor map is given by the map:
\begin{equation*}
    \pi:\operatorname{Diff}_{A}(L,L)\rightarrow \Der_K(A,A), \quad f\mapsto (a\mapsto af-fa),
\end{equation*}
which is well-defined by definition of $\operatorname{Diff}_{A}(L,L)$.  The kernel of $\pi$ is given by 
the $f\in \Hom_{K}(L,L)$ which commute with the action of $A$. Thus, we have:
\begin{equation*}
    \operatorname{ker}(\pi)=\Hom_A(L,L)=A,
\end{equation*}
and we have a left exact sequence of finite $A$-modules:
\begin{equation*}
    0\rightarrow A\rightarrow \operatorname{Diff}_{A}(L,L)\rightarrow \mathcal{T}_{X/K}(X).
\end{equation*}
In order to show that it is right exact, it suffices to show that the corresponding sequence of $\OX_X$-modules is exact, and this holds because 
$\operatorname{Diff}_{A}(L,L)=\Der_K(A,A)$ whenever $L$ is a trivial line bundle. Therefore, $\mathcal{L}$ induces an Atiyah algebra on $X$, which we will denote by $\mathcal{A}_{\mathcal{L}}$. Furthermore, an easy calculation shows that for any affinoid subdomain $U\subset X$ we have:
\begin{equation*}
    \mathcal{A}_{\mathcal{L}}(U)=\OX_X(U)\otimes_{A}\operatorname{Diff}_{A}(L,L)=\operatorname{Diff}_{\OX_X(U)}(\mathcal{L}(U),\mathcal{L}(U))=\mathcal{A}_{\mathcal{L}_{\vert U}}(U).
\end{equation*}
Thus, this construction globalizes, and we have the following proposition:
\begin{prop}\label{prop atiyah algebra from line bundle}
    Let $X$ be a smooth rigid $K$-variety and $\mathcal{L}$ be a line bundle on $X$. Then $\mathcal{L}$ induces a unique Atiyah algebra $\mathcal{A}_{\mathcal{L}}$ on $X$, satisfying that for any affinoid subspace $U\subset X$ we have:
    \begin{equation*}
        \mathcal{A}_{\mathcal{L}}(U)=\operatorname{Diff}_{\OX_X(U)}(\mathcal{L}(U),\mathcal{L}(U)).
    \end{equation*}
\end{prop}
As we will see in the next section, every Atiyah algebra gives rise to a sheaf of filtered algebras, which is called its sheaf of twisted differential operators. For instance, the trivial Atiyah algebra $\OX_X\oplus \mathcal{T}_X$ induces the usual sheaf of differential operators on $X$. Thus, we can see the previous proposition, together with the (to be defined) twisted differential operators construction,  as a way of obtaining sheaves of differential operators with coefficients in line bundles.\\
Now that we have a basic understanding of Atiyah algebras, the next step is giving a classification of their isomorphism classes. Assume first that $X=\Sp(A)$ is affinoid, with an Atiyah algebra $\mathcal{A}$. Then it suffices to classify the global sections. We will now recall the description in  \cite[Chapter I.2.4.4]{ginzburg1998lectures} of the $(K,A)$-Lie algebra structure on $\mathcal{A}(X)$. First, as $X$ is smooth, it follows that $\mathcal{T}_{X/K}(X)$ is a finite projective module. In particular,  we have a decomposition:
\begin{equation*}
    \mathcal{A}(X)=A\oplus \mathcal{T}_{X/K}(X).
\end{equation*}
As the map $\mathcal{A}(X)\rightarrow \mathcal{T}_{X/K}(X)$ is a morphism of $K$-Lie algebras, it follows that for $f,g\in A$, and $v,w\in\mathcal{T}_{X/K}(X)$, we have:
\begin{equation}\label{bracket of tdo}
    [(f,v),(g,w)]=(v(g)-w(f)+\omega(v,w),[v,w]).
\end{equation}
The axioms of a $(K,A)$-lie algebra can be used to show that the map:
\begin{equation*}
    \omega:\mathcal{T}_{X/K}(X)\times \mathcal{T}_{X/K}(X)\rightarrow A,
\end{equation*}
is in fact a closed $2$-form on $X$. On the other hand, any closed $2$-form $\omega\in\Omega_{X/K}^{2,\operatorname{cl}}(X)$ may be used to construct an Atiyah algebra via the formula in equation $(\ref{bracket of tdo})$.
\begin{defi}
 Given  a form $\omega\in\Omega_{X/K}^{2,\operatorname{cl}}(X)$, we let $\mathcal{A}_{\omega}$ be the unique Atiyah algebra on $X$ induced by $\omega$ via the formula in equation $(\ref{bracket of tdo})$.    
\end{defi}
Let $\mathcal{A}_{\omega_1},\mathcal{A}_{\omega_2}$ be two Atiyah algebras and assume we have a morphism:
\begin{equation*}
    \varphi:\mathcal{A}_{\omega_1}\rightarrow \mathcal{A}_{\omega_2}.
\end{equation*}
As $\varphi$ is $A$-linear, commutes with the anchor maps, and satisfies $\varphi(1_{\mathcal{A}_{\omega_1}})=1_{\mathcal{A}_{\omega_2}}$, it is an isomorphism. Furthermore, it has to be given by the expression:
\begin{equation}\label{equation morphism of tdo}
    \varphi([a,v])=[a+d_{\varphi}(v),v], \textnormal{ for } a\in A,\textnormal{ }v\in\mathcal{T}_{X/K}(X).
\end{equation}
As $\varphi$ commutes with the brackets, the following identity holds for $v,w\in\mathcal{T}_{X/K}(X)$:
\begin{equation}\label{equation isomorphisms of Atiyah algebras}
    \omega_1(v,w)-\omega_2(v,w)= v(d_{\varphi}(w))-w(d_{\varphi}(v)).
\end{equation}
Regarding $d_{\varphi}:\mathcal{T}_{X/K}(X)\rightarrow A$ as an element in $\Omega^1_{X/K}(X)$, equation $(\ref{equation isomorphisms of Atiyah algebras})$ shows that two Atiyah algebras $\mathcal{A}_{\omega_1}$ and $\mathcal{A}_{\omega_2}$ are isomorphic if and only if $\omega_1$ and $\omega_2$ represent the same cohomology class in $\operatorname{H}^2_{\operatorname{dR}}(X)=\operatorname{H}^2\left(\Gamma(X,\Omega_{X/K}^{\bullet}) \right)$. We may condense the previous discussion into the following proposition:
\begin{prop}\label{prop classification of atiyah algebras on smooth affinoid varieties}
Let $X$ be a smooth affinoid variety. The following hold:
\begin{enumerate}[label=(\roman*)]
    \item The category $\mathscr{PA}(X)$ is an essentially small groupoid.
    \item The isomorphism classes of $\mathscr{PA}(X)$ are in bijection with $\operatorname{H}^2_{\operatorname{dR}}(X)$.
    \item Every $\mathcal{A}_{\omega}\in \mathscr{PA}(X)$ satisfies $\operatorname{Aut}_{\mathscr{PA}(X)}(\mathcal{A}_{\omega})=\Omega_{X/K}^{1,\operatorname{cl}}(X)$.
\end{enumerate}
\end{prop}
The idea now is extending this classification to all smooth and separated rigid analytic spaces. However, as $X$ is no longer affinoid, the global sections of the de Rham complex no longer need to calculate the de Rham cohomology, and Atiyah algebras are not necessarily determined by their global sections. This leads us to introduce a truncated version of the de Rham complex:
\begin{defi}
Let $X$ be a smooth rigid $K$-variety. We define the truncated de Rham complex of $X$ to be the following complex:
\begin{equation*}
    \Omega_{X/K}^{\geq 1}:=\left(0\rightarrow \Omega_{X/K}^{1}\rightarrow  \Omega^{2,\operatorname{cl}}_{X/K}\rightarrow 0\right).
\end{equation*}
\end{defi}
With this new complex at hand, we can finally state our desired classification:
\begin{teo}\label{class of TDO's Rigid Analytic Spaces}
 Let $X$ be a smooth rigid analytic space. The following hold:
 \begin{enumerate}[label=(\roman*)]
    \item The category $\mathscr{PA}(X)$ is an essentially small groupoid. In particular, the category of isomorphism classes  $\operatorname{Iso}(\mathscr{PA}(X))$ is a pointed set.
    \item There is an isomorphism of pointed sets $\operatorname{Iso}(\mathscr{PA}(X))\rightarrow \mathbb{H}^{1}(X,\Omega_{X/K}^{\geq 1})$.
    \item For an Atiyah algebra $\mathcal{A}\in \mathscr{PA}(X)$, we have $\operatorname{Aut}_{\mathscr{PA}(X)}(\mathcal{A})=\Omega_{X/K}^{1,\operatorname{cl}}(X)$.
 \end{enumerate}
\end{teo}
\begin{proof}
This is essentially the same as \cite[Theorem I.2.4.4.3]{ginzburg1998lectures}.
\end{proof}
By definition of $\Omega^{\geq 1}_{X/K}$, we have the following short exact sequence:
\begin{equation*}
    0\rightarrow \Omega^{2,\operatorname{cl}}_{X/K}[1]\rightarrow \Omega^{\geq 1}_{X/K}\rightarrow \Omega^{1}_{X/K}\rightarrow 0.
\end{equation*}
By taking hypercohomology, we arrive at the left exact  exact sequence:
\begin{equation}\label{equation truncated de Rham complex}
    0\rightarrow \textnormal{H}^{2}(\Omega^{\bullet}_{X/K}(X))\rightarrow \mathbb{H}^{1}(X,\Omega_{X/K}^{\geq 1})\rightarrow \textnormal{H}^{1}(X,\Omega_{X/K}^{1}). 
\end{equation}
As seen in Proposition \ref{prop atiyah algebra from line bundle}, each line bundle $\mathcal{L}$ induces an Atiyah algebra $\mathcal{A}_{\mathcal{L}}$ on $X$.  Hence, we the previous sequence induces a  map:
\begin{equation*}
    \operatorname{Pic}(X)\rightarrow \textnormal{H}^{1}(X,\Omega_{X/K}^{1}),
\end{equation*}
which corresponds to taking the first Chern class. If $X$ is affinoid, then $\textnormal{H}^{1}(X,\Omega_{X/K}^{1})=0$, and $\textnormal{H}^{2}(\Omega^{\bullet}_{X/K}(X))=\textnormal{H}_{\operatorname{dR}}^2(X)$, so we recover the local classification. 
\subsection{Twisted differential operators on smooth rigid spaces} Let $X$ be a smooth rigid space with an Atiyah algebra $\mathcal{A}$. As shown in \cite[Section 2.3]{ardakov2019}, there is a unique sheaf of filtered algebras $U(\mathcal{A})$ on $X$ satisfying the following properties:
\begin{enumerate}[label=(\roman*)]
    \item Let $Y=\Sp(A)\subset X$ be an admissible open, then $\Gamma(Y,U(\mathcal{A}))=U(\mathcal{A}(Y))$, where $U(\mathcal{A}(Y))$ is the universal enveloping algebra of $\mathcal{A}(Y)$ (\emph{cf.} \cite{rinehart1963differential}).
    \item  If $Z=\Sp(B)\subset Y=\Sp(A)$ is an affinoid subdomain, then we have:
    \begin{equation*}
       \Gamma(Z,U(\mathcal{A}))=U(\mathcal{A}(Z))=B\otimes_A U(\mathcal{A}(Y)).
    \end{equation*}
\end{enumerate}
We call $U(\mathcal{A})$ the sheaf of enveloping algebras of $\mathcal{A}$. We are now ready to define the sheaves of twisted differential operators:
\begin{defi}\label{defi tdo}
The sheaf of twisted differential operators (TDO) associated to $\mathcal{A}$ is the unique sheaf of algebras $\D_{\mathcal{A}}$ fitting into the following short exact sequence:
\begin{equation*}
    0\rightarrow \mathcal{I}_{\mathcal{A}}\rightarrow U(\mathcal{A})\rightarrow \D_{\mathcal{A}}\rightarrow 0,
\end{equation*}
where $\mathcal{I}_{\mathcal{A}}$ is the sheaf of two-sided ideals generated by $1_{\mathcal{A}}-1_{U(\mathcal{A})}\in \Gamma(X,U(\mathcal{A}))$. 
If $X$ is affinoid, then $\mathcal{A}=\mathcal{A}_{\omega}$ for some $\omega \in  \Omega^{2,\operatorname{cl}}_{X/K}(X)$. Thus, we will write $\D_{\omega}:=\D_{\mathcal{A}}$,  
 $\mathcal{I}_{\omega}:=\mathcal{I}_{\mathcal{A}}$, and call $\D_{\omega}$ the sheaf of $\omega$-twisted differential operators on $X$.
\end{defi}
We now show some basic properties of twisted differential operators:
\begin{prop}\label{prop PBW for twisted differential operators}
There is a positive, exhaustive filtration  $\Phi_{\bullet}\D_{\mathcal{A}}$ such that there is a canonical isomorphism of sheaves of graded $K$-algebras:
\begin{equation*}
    \operatorname{Sym}_{\OX_X}\left(\mathcal{T}_{X/K}\right)\rightarrow \gr_{\Phi}\D_{\mathcal{A}}.
\end{equation*}
In particular, for every affinoid space $U\subset X$, the filtered algebra $\D_{\mathcal{A}}(U)$ is two-sided noetherian and almost commutative (cf. \textnormal{\cite[Definition 3.4]{ardakov2013irreducible}}).
\end{prop}
\begin{proof}
This is shown in \cite[Lemma 2.1.4]{beilinson1993proof}.
\end{proof}
\begin{defi}\label{defi filtration by order of dif operators}
Let $\D_{\mathcal{A}}$ be a TDO. The filtration $\Phi_{\bullet}\D_{\mathcal{A}}$ constructed in Proposition \ref{prop PBW for twisted differential operators} is called the filtration by order of differential operators.
\end{defi}
For later use, it will be convenient to have an expression of TDO's in terms of quotients of the enveloping algebra of certain $K$-Lie algebras. Assume $X=\Sp(A)$ is affinoid, and let $\mathcal{A}_{\omega}$ be an Atiyah algebra on $X$. Consider the $K$-Lie algebra $\overline{\mathcal{A}_{\omega}}=A\oplus \mathcal{A}_{\omega}$, with bracket given by:
\begin{equation*}
 [(f_1,f_2,v),(g_1,g_2,w)]=(v(g_1)-w(f_1), v(g_2)-w(f_2)+\omega(v,w),[v,w]),    
\end{equation*}
where $(f_1,f_2,v),(g_1,g_2,w)\in \overline{\mathcal{A}_{\omega}}=A\oplus A\oplus \mathcal{T}_{X/K}(X)$. Let $U_K(\overline{\mathcal{A}_{\omega}})$ be the universal enveloping algebra of $\overline{\mathcal{A}_{\omega}}$ with respect to its $K$-Lie algebra structure. We then have the following proposition:
\begin{prop}\label{prop tdo as quotient of universal enveloping algebra}
Let $I$ be the two-sided ideal  in  $U_K(\overline{\mathcal{A}_{\omega}})$ generated by the following elements:
\begin{equation*}
1-(1,0,0), \textnormal{  }  1-(0,1,0), \textnormal{  } (f,0,0)(g,h,v)-(fg,fh,fv), 
\end{equation*}
where $f,g,h\in A$ and $v\in \mathcal{T}_{X/K}(X)$. Then $\D_{\omega}=U_K(\overline{\mathcal{A}_{\omega}})/ I$.   
\end{prop}
\begin{proof}
This follows by the construction of universal enveloping algebras of Lie-Rinehart algebras from \cite[Section 2]{rinehart1963differential}, together with Definition \ref{defi tdo}.
\end{proof}
\begin{obs}
If $\mathcal{A}$ is the trivial Atiyah algebra, then $\D_{\mathcal{A}}=\D_X$.
\end{obs}
The sheaf $\D_{\mathcal{A}}$  corresponds to the sheaf of twisted differential operators of finite order, in the sense that 
for every open affinoid subspace $U\subset X$, the sections of  $\D_{\mathcal{A}}(U)$ can be expressed as polynomials of the vector fields in $\mathcal{T}_{X/K}(U)$. The next goal is defining sheaves of infinite order twisted differential operators. 
\begin{defi}[{\cite[Definition 6.1]{ardakov2019}}]
Let $A$ be affinoid, $\mathcal{A}\subset A$  be an affine formal model, and $L$ be a $(K,A)$-Lie algebra. A smooth $(\mathcal{R},\mathcal{A})$-Lie lattice of $L$ is a smooth $(\mathcal{R},\mathcal{A})$-Lie algebra $\mathcal{L}$ with an isomorphism of $(K,A)$-Lie algebras:
\begin{equation*}
  L=\mathcal{L}\otimes_{\mathcal{R}}K.  
\end{equation*}
\end{defi}

Back to our previous setting, let $A$ be an affinoid $K$-algebra with an affine formal model $\mathcal{A}\subset A$, and let $L$ be a smooth $(K,A)$-Lie algebra with a smooth $(\mathcal{R},\mathcal{A})$-Lie lattice $\mathcal{L}$. For every $n\geq 0$, the $\mathcal{A}$-module $\pi^n\mathcal{L}$ is still a smooth $(\mathcal{R},\mathcal{A})$-Lie lattice. In particular, we can form the enveloping algebra $U(\pi^n\mathcal{L})$. Let $\widehat{U}(\pi^n\mathcal{L})$ be the $\pi$-adic completion of $U(\pi^n\mathcal{L})$. Then $\widehat{U}(\pi^n\mathcal{L})$ is a $\mathcal{R}$-flat, $\pi$-adically complete, and two sided noetherian $\mathcal{R}$-algebra. Furthermore, the canonical maps:
\begin{equation*}
    \widehat{U}(\pi^{n+1}\mathcal{L})_K:=\widehat{U}(\pi^{n+1}\mathcal{L})\otimes_{\mathcal{R}}K\rightarrow \widehat{U}(\pi^n\mathcal{L})_K,
\end{equation*}
are
two sided flat and have dense image (cf. \cite[Theorem 6.7]{ardakov2019}). Hence, we can form the following Fréchet-Stein algebra:
\begin{equation*}
    \wideparen{U}(L)=\varprojlim \widehat{U}(\pi^n\mathcal{L})_K,
\end{equation*}
which we call the Fréchet-Stein enveloping algebra of $L$. As shown in \cite{ardakov2019}, any  Atiyah algebra $\mathcal{A}$ admits a sheaf of Fréchet-Stein enveloping algebras $\wideparen{U}(\mathcal{A})$, defined on open affinoid  subspaces $V=\Sp(A)\subset X$ by the rule:
\begin{equation*}
   \Gamma(V,\wideparen{U}(\mathcal{A}))=\wideparen{U}(\mathcal{A}(V)).
\end{equation*}
Furthermore,  by the contents of \cite{bode2019completed}, for any affinoid subdomain $W=\Sp(B)\subset V$, the following identity holds:
\begin{equation}\label{equation tensor products and sections of completed enveloping algebras}
    \wideparen{U}(\mathcal{A}(W))=B\widehat{\otimes}_A\wideparen{U}(\mathcal{A}(V))=\wideparen{U}(B\otimes_A\mathcal{A}(V)),
\end{equation}
where $\widehat{\otimes}_A$ denotes the completed projective tensor product of locally convex $A$-modules (\emph{cf.} \cite[Section 1]{taylor1972homology}). Now, while the algebras of twisted differential operators are not exactly of the form described above, we will still be able to show that they satisfy analogous properties. Let us start with the following theorem:
\begin{teo}\label{teo complete tdo are fréchet-Stein}
Let $X$ be a smooth rigid space with an Atiyah algebra $\mathcal{A}$. Consider the following sheaves on $X$:
\begin{equation*}
    \wideparen{\D}_{\mathcal{A}}=\wideparen{U}(\mathcal{A})\otimes_{U(\mathcal{A})}\mathcal{D}_{\mathcal{A}}, \textnormal{ } \wideparen{I}_{\mathcal{A}}=\wideparen{U}(\mathcal{A})\otimes_{U(\mathcal{A})}\mathcal{I}_{\mathcal{A}}.
\end{equation*}
There is a short exact sequence of co-admissible $\wideparen{U}(\mathcal{A})$-modules:
\begin{equation}\label{equation ses defining completion of tdo}
    0\rightarrow \wideparen{I}_{\mathcal{A}}\rightarrow \wideparen{U}(\mathcal{A})\rightarrow \wideparen{\D}_{\mathcal{A}}\rightarrow 0.
\end{equation}
Furthermore, the following hold for all affinoid subdomains $V\subset U\subset X$:
\begin{enumerate}[label=(\roman*)]
    \item $\wideparen{\D}_{\mathcal{A}}(U)$ is a Fréchet-Stein algebra.
    \item The canonical map $\D_{\mathcal{A}}(U)\rightarrow\wideparen{\D}_{\mathcal{A}}(U)$ is faithfully flat with dense image.
    \item The map $\wideparen{\D}_{\mathcal{A}}(U)\rightarrow \wideparen{\D}_{\mathcal{A}}(V)$ is continuous and $c$-flat.
    \item The restriction  $\wideparen{\D}_{\mathcal{A}\vert U}$ has trivial higher sheaf cohomology.
    \item $\wideparen{\D}_{\mathcal{A}}$ is a sheaf of complete topological algebras on $X$.
\end{enumerate}
\end{teo}
\begin{proof}
We may assume that $X=\Sp(A)$ is affinoid. In particular, it suffices to show that the sequence is exact on global sections. By \cite[Corollary 3.2]{Ardakov_Bode_Wadsley_2021} the morphism $U(\mathcal{A}(X))\rightarrow \wideparen{U}(\mathcal{A}(X))$ is faithfully flat with dense image. Hence, $(\ref{equation ses defining completion of tdo})$ is exact and $(ii)$ holds. Furthermore, both 
$\wideparen{\D}_{\mathcal{A}}$ and $\wideparen{I}_{\mathcal{A}}$ are finitely presented. Thus, statement $(i)$ follows by  a combination of \cite[Lemma 3.6]{schneider2002algebras} and \cite[Proposition 3.7]{schneider2002algebras}. In order to see that 
$\wideparen{\D}_{\mathcal{A}}(U)\rightarrow \wideparen{\D}_{\mathcal{A}}(V)$ is $c$-flat and continuous, we use the identity
\begin{align*}
\wideparen{\D}_{\mathcal{A}}(V)=\wideparen{\D}_{\mathcal{A}}(U)\wideparen{\otimes}_{\wideparen{U}(\mathcal{A}(U))}\wideparen{U}(\mathcal{A}(V)),
\end{align*}
together with the associativity of $\wideparen{\otimes}$ from \cite[Corollary 7.4]{ardakov2019}, and the fact that the transition maps of $\wideparen{U}(\mathcal{A})$ are $c$-flat shown in \cite[Theorem 7.7]{ardakov2019}. Claim $(iv)$ is a consequence of  \cite[Theorem 8.2]{ardakov2019}.
\end{proof}
For future use, it will be convenient to have a better description of the local sections of the sheaves of infinite order differential operators: Let $X=\Sp(A)$ and assume $\mathcal{T}_{X/K}(X)$ is a free $A$-module. 
As before, we will write $\mathcal{A}_{\omega}=\mathcal{A}_{\omega}(X)$, and follow the same convention for $\mathcal{I}_{\omega}$ and $\D_{\omega}$. Let $A^{\circ}\subset A$ be the subring of power-bounded elements. As $K$ is discretely valued, it follows by \cite[Theorem 3.5.6]{Fresnel2004} that $\mathfrak{X}=\Spf(A^{\circ})$
is an affine formal model of $X$. By multiplying by adequate powers of the uniformizer, we may choose an $A$-basis $v_1,\cdots,v_n$ of $\mathcal{T}_{X/K}(X)$
satisfying that $v_i(A^{\circ})\subset A^{\circ}$ for every $1\leq i \leq n$. In this case, $\mathscr{T}=\bigoplus_{i=1}^nA^{\circ}v_i$ is a free $A^{\circ}$-Lie lattice of $\mathcal{T}_{X/K}(X)$. Furthermore, the closed 2-form:
\begin{equation*}
    \omega:\mathcal{T}_{X/K}(X)\times\mathcal{T}_{X/K}(X)\rightarrow A,
\end{equation*}
is $A$-linear. Hence, we may assume (maybe after multiplying $\mathscr{T}$ by a high enough power of $\pi$)  that $\omega(\mathscr{T},\mathscr{T})\subset A^{\circ}$. In this case, the $A^{\circ}$-module: $\mathscr{A}_{\omega}=A^{\circ}\oplus \mathscr{T}$ is a free $(\mathcal{R},A^{\circ})$-Lie lattice of $\mathcal{A}_{\omega}$, so the inverse limit $\wideparen{U}(\mathcal{A}_{\omega})=\varprojlim_n \widehat{U}(\pi_n\mathscr{A}_{\omega})_K$ is a Fréchet-Stein presentation of $\wideparen{U}(\mathcal{A}_{\omega})$ (\emph{cf.} \cite[Theorem 6.4]{ardakov2019}).\\
As the maps $U(\mathcal{A}_{\omega})\rightarrow \widehat{U}(\pi_n\mathscr{A}_{\omega})_K$ are flat for every $n\geq 0$, 
we can define:
\begin{equation*}
 \widehat{\D}^n_{\omega}:=\widehat{U}(\pi_n\mathscr{A}_{\omega})_K\otimes_{U(\mathcal{A}_{\omega})}\D_{\omega}, \quad \widehat{\mathcal{I}}^n_{\omega}:=\widehat{U}(\pi_n\mathscr{A}_{\omega})_K\otimes_{U(\mathcal{A}_{\omega})}\mathcal{I}_{\omega},   
\end{equation*}
and we obtain a family of short exact sequences:
\begin{equation}\label{equation FS presentation of completed TDO}
   0\rightarrow \widehat{I}^n_{\omega}\rightarrow \widehat{U}(\pi^n\mathscr{A}_{\omega})_K\rightarrow \widehat{\D}^n_{\omega}\rightarrow 0.
\end{equation}
Then $\wideparen{\D}_{\omega}=\varprojlim_n\widehat{\D}^n_{\omega}$ is a Fréchet-Stein presentation of $\wideparen{\D}_{\omega}$ by \cite[Proposition 3.7]{schneider2002algebras}. 
\subsection{Group actions on twisted differential operators}\label{section 2.4}
In order to construct Cherednik algebras on rigid spaces, we need to understand how the action of a finite group on a smooth rigid space $X$ extends to an action on the sheaves of twisted differential operators. In this section, we will identify the class of twisted differential operators which are acted on by $G$, give a $G$-equivariant classification of the twisted differential operators, and discuss the skew group algebra construction.\\
Let $X$ be a $G$-variety and recall from Theorem \ref{teo existence of quotients of rigid spaces by finite groups} the existence of the quotient:
\begin{equation*}
    \pi:X\rightarrow X/G.
\end{equation*}
We will now show how the $G$-action on $X$ induces a $G$-action on the de Rham complex $\Omega_{X/K}^{\bullet}$: The right $G$-action on $X$ induces  a left action of sheaves of $K$-algebras on $\OX_X$. Furthermore, we obtain an action of $G$ on the tangent sheaf $\mathcal{T}_{X/K}$, by setting $g(v)=g\circ \varphi \circ g^{-1}$
for each vector field $v:\OX_X(X)\rightarrow \OX_X(X)$, and each $g\in G$. More generally, for $\omega\in \Omega_{X/K}^{j}(X)$ and $g\in G$, we can regard $\omega$ as a $\OX_{X}(X)$-linear map $\omega:\wedge_{i=1}^{j}\mathcal{T}_{X/K}(X)\rightarrow \OX_{X}(X)$. We define a left action of $G$ on $\Omega_{X/K}^{j}(X)$ by the following formula:
\begin{equation}\label{equation action of G on the de Rham complex}
    g(\omega)(v_{1},\cdots,v_{j})=g(\omega(g^{-1}(v_{1}),\cdots,g^{-1}(v_{j}))), \textnormal{ for }v_{1},\cdots,v_{j}\in\mathcal{T}_{X/K}(X).
\end{equation}
 A hands-on calculation shows that these actions turn $\pi_*\Omega_{X/K}^{\bullet}(X)$ into a complex of $K[G]$-modules. The construction of the actions is functorial, so it follows that $\Omega_{X/K}^{\bullet}$ is a complex of sheaves of  $K$-linear $G$-representations.  We will now investigate the fixed points of this action. As usual, we start with the affinoid case:
\begin{Lemma}\label{extension of the action of G to differential forms}
 Let $X=\Sp(A)$ be a smooth $G$-variety and let $\omega\in\Omega_{X/K}^{2,\operatorname{cl}}(X)$. The canonical action of $G$ on $\OX_X(X)$ and $\mathcal{T}_{X/K}(X)$ extends to an action of $K$-algebras on $\D_{\omega}(X)$ if and only if $\omega\in\Omega_{X/K}^{2,\operatorname{cl}}(X)^G$.
\end{Lemma}
\begin{proof}
Assume that we have an action of $G$ on $\mathcal{D}_{\omega}(X)$ as in the statement. Any $K$-linear action of $G$ on $\mathcal{D}_{\omega}(X)$ which extends the action of $G$ on $\mathcal{T}_{X/K}(X)$ is an action of filtered rings. Indeed, the filtration by order of differential operators has $\Phi_{0}\mathcal{D}_{\omega}(X)=A$, $\Phi_{1}\mathcal{D}_{\omega}(X)=A+\mathcal{T}_{X/K}(X)$, and the rest of the terms are obtained by taking products of $\Phi_{1}\mathcal{D}_{\omega}(X)$. As any such action would induce a $K$-linear automorphism of the first two terms of the filtration, it follows that it induces automorphisms in the rest of the terms. In particular, as $\mathcal{A}_{\omega}=\Phi_{1}\mathcal{D}_{\omega}(X)$, we get a $G$-action of $K$-vector spaces on $\mathcal{A}_{\omega}$. Furthermore, as the Lie bracket in $\mathcal{A}_{\omega}$ is induced by the commutator in $\mathcal{D}_{\omega}(X)$, and the action of $G$ is an action of algebras, it follows that $G$ acts by automorphisms of $K$-Lie algebras on $\mathcal{A}_{\omega}$.\\
 As the $G$-action on $\mathcal{A}_{\omega}$ extends the action of $G$ on $\OX_{X}(X)$ and $\mathcal{T}_{X/K}(X)$, it follows that the $G$-action on $\mathcal{A}_{\omega}=\OX_{X}(X)\oplus\mathcal{T}_{X/K}(X)$ is given by the formula:
 \begin{equation*}
     g(f,v)=(g(f),g(v)), \textnormal{ for }g\in G \textnormal{, and } (f,v)\in \mathcal{A}_{\omega}.
 \end{equation*}
By the definition of the action of $G$ on $\mathcal{T}_{X/K}(X)$ we have:
\begin{equation*}
g([(f,v),(h,w)])=(g(v)(g(h))-g(w)(g(f))+g(\omega(v,w)),[g(v),g(w)]).    
\end{equation*}
On the other hand, we have:
\begin{equation*}
    [g(f,v),g(h,w)]=(g(v)(g(h))-g(w)(g(f))+\omega(g(v),g(w)),[g(v),g(w)]).
\end{equation*}
Thus, any action of $G$ on $\mathcal{D}_{\omega}(X)$ as above satisfies $g(\omega(v,w))=\omega(g(v),g(w))$ for $g\in G$, and $v,w\in\mathcal{T}_{X/K}(X)$. By equation $(\ref{equation action of G on the de Rham complex})$, this is equivalent to $\omega\in\Omega_{X/K}^{2,\operatorname{cl}}(X)^G$.\\
 Conversely, for  $\omega\in\Omega_{X/K}^{2,cl}(X)^G$, and $g\in G$, the formula $g((f,v))=(g(f),g(v))$ defines an action of $K$-Lie algebras of $G$ on $\mathcal{A}_{\omega}$. Hence, we get an induced action of $K$-Lie algebras on $\overline{\mathcal{A}_{\omega}}$, and this lifts to a $G$-action of $K$-algebras on its universal enveloping algebra  $U_{K}(\overline{\mathcal{A}_{\omega}})$. Notice that by Proposition \ref{prop tdo as quotient of universal enveloping algebra}, $\mathcal{D}_{\omega}(X)$ is the quotient of $U_{K}(\overline{\mathcal{A}_{\omega}})$ by a two sided ideal which is invariant under the action of $G$. Thus, we get a $G$-action of filtered $K$-algebras on $\mathcal{D}_{\omega}(X)$, as wanted.
\end{proof}
We will now extend this result to the Fréchet-Stein completions:
\begin{Lemma}\label{G invariant Lie latices}
Let $X=\Sp(A)$ be an affinoid $G$-variety  and choose $\omega \in\Omega_{X/K}^{2,\operatorname{cl}}(X)^G$. If $X$ satisfies that $\Omega^{1}_{X/K}(X)$ has a basis given by differentials of functions, then there is a smooth $(\mathcal{R},A^{\circ})$-Lie lattice $\mathscr{A}_{\omega}$ of $\mathcal{A}_{\omega}$ which is $G$-invariant.
\end{Lemma}
\begin{proof}
Let $df_{1},\cdots df_{n}$ be an $A$-basis of $\Omega^{1}_{X/K}(X)$ given by differentials of functions in $A^{\circ}$ and let $v_{1},\cdots,v_{n}$ be the dual basis of $\mathcal{T}_{X/K}(X)$. Set $\mathscr{T}=\bigoplus_{i=1}^nA^{\circ}v_i$ and $\mathscr{A}_{\omega}=A^{\circ}\oplus \mathscr{T}$. As in the discussion after Theorem \ref{teo complete tdo are fréchet-Stein}, we may assume that $v_{i}(A^{\circ})\subset A^{\circ}$ for each $i$ and $\omega(\mathscr{T},\mathscr{T})\subset A^{\circ}$. In particular,
$\mathscr{A}_{\omega}$ is a finite-free $(\mathcal{R},A^{\circ})$-Lie lattice of $\mathcal{A}_{\omega}$. We just need to show that it is invariant under the action of $G$.\\
First, notice that as $G$ acts on $A$ by continuous maps, the space $A^{\circ}$ is invariant under the action of $G$. Next, as the $v_{1},\cdots,v_{n}$ are an $A$-basis of $\mathcal{T}_{X/K}(X)$, we have that $g(v_{i})=\sum_{j=1}^{n}h_{ij}v_{j}$ for each $i$ and each $g\in G$. We are done if we are able to show that the $h_{ij}$ are elements of $A^{\circ}$. By definition of a dual basis we have:
\begin{equation}
    h_{ij}=df_{j}(g(v_{i}))=g(v_{i})(f_{j})=g(v_i(g^{-1}(f_j)).
\end{equation}
As $G$ induces an isomorphism of $A^{\circ}$, our assumption that the $f_{j}$ are power-bounded and $v_{i}(A^{\circ})\subset A^{\circ}$ implies that $h_{ij}\in A^{\circ}$ for each $i$ and $j$, as we wanted to show.
\end{proof}
\begin{prop}\label{action on operators}
 Let $X=\Sp(A)$ be a smooth $G$-variety and let $\omega\in\Omega_{X/K}^{2,\operatorname{cl}}(X)$. If $X$ satisfies that $\Omega^{1}_{X/K}(X)$ has a basis given by differentials of functions, then the canonical action of $G$ on $\OX_X(X)$ and $\mathcal{T}_{X/K}(X)$ extends to an action of topological $K$-algebras on $\wideparen{\D}_{\omega}(X)$  if and only if $\omega\in\Omega_{X/K}^{2,\operatorname{cl}}(X)^G$.
\end{prop}
\begin{proof}
Choose $\omega\in\Omega_{X/K}^{2,\operatorname{cl}}(X)^G$,  and let $\mathscr{A}_{\omega}$ be a finite-free $G$-invariant $(\mathcal{R},A^{\circ})$-Lie lattice of $\mathcal{A}_{\omega}(X)$ as in Lemma \ref{G invariant Lie latices}.
For each $n\geq 0$, we have an action of $\mathcal{R}$-Lie algebras of $G$ on $\pi^n\mathscr{A}_{\omega}$. Following the procedure in Lemma \ref{extension of the action of G to differential forms}, we get an action of $\mathcal{R}$-algebras on the enveloping algebra $U(\pi^n\mathscr{A}_{\omega})$. As every $\mathcal{R}$-linear map is continuous with respect to the $\pi$-adic topology, this action extends uniquely to a continuous action of $K$-algebras of $G$ on $\widehat{U}(\pi^n\mathscr{A}_{\omega})_K$. Furthermore, the transition maps $\widehat{U}(\pi^{n+1}\mathscr{A}_{\omega})_K\rightarrow \widehat{U}(\pi^n\mathscr{A}_{\omega})_K$ are clearly $G$-equivariant. Hence, we get a continuous action of $G$ on the Fréchet-Stein algebra $\wideparen{U}(\mathcal{A}_{\omega}(X))$. As $G$ leaves the ideal $\wideparen{I}_{\omega}$ invariant, it follows that the action descends to an  action of topological $K$-algebras of $G$ on $\wideparen{\D}_{\omega}(X)$. Furthermore, this action is unique because the map $\D_{\omega}(X)\rightarrow \wideparen{\D}_{\omega}(X)$ has dense image. For the converse, assume that the action of $G$ extends to $\wideparen{\mathcal{D}}_{\omega}(X)$. Then we get an action of $G$ on $\mathcal{D}_{\omega}(X)$, and Lemma \ref{extension of the action of G to differential forms} shows that $\omega$ must be fixed by $G$.
\end{proof}
The next goal is showing that $G$ acts on the sheaves $\pi_*\wideparen{\D}_{\omega}$ for appropriate $\omega$:
\begin{prop}\label{prop transition morphisms are G-equivariant}
Let $X=\Sp(A)$ be a $G$-variety with an étale map $X\rightarrow \mathbb{A}^n_K$ and choose $\omega\in\Omega^{2,\operatorname{cl}}_{X/K}(X)$. The actions of $G$ on $\pi_*\OX_X$  and $\pi_*\mathcal{T}_{X/K}$  extend  to an action of sheaves of topological $K$-algebras on $\pi_*\wideparen{\D}_{\omega}$ if and only if $\omega\in\Omega_{X/K}^{2,\operatorname{cl}}(X)^G$.  If this holds, then for any $G$-invariant subdomain $Y=\Sp(B)\subset X$ the  map:
    \begin{equation*}
      \wideparen{\D}_{\omega}(X)\rightarrow \wideparen{\D}_{\omega}(Y),  
    \end{equation*}
    is a $G$-equivariant and  $c$-flat morphism of Fréchet-Stein algebras. Furthermore, if $(U_i)_{i=1}^n$ is an affinoid cover of $X$, then the associated map:
    \begin{equation*}
        \wideparen{\D}_{\omega}(X)\rightarrow \prod_{i=1}^n\wideparen{\D}_{\omega}(U_i),
    \end{equation*}
    is also $c$-faithfully flat.
\end{prop}
\begin{proof}
As $X$ is separated, it suffices to define a continuous action of $G$ on $\pi_*\wideparen{\D}_{\mathcal{A}}(U)$, where  $U$ ranges over the affinoid subdomains of $X/G$. By Theorem \ref{teo existence of quotients of rigid spaces by finite groups}, this is equivalent to giving the action on $G$-invariant affinoid subdomains of $X$. Thus, it suffices to show the second part of the proposition. The results on $c$-(faithful) flatness follow from the analogous results for $\wideparen{U}(\mathcal{A}_{\omega})$, which were shown in \cite[Theorem 7.7]{ardakov2019}. Similarly, it suffices to show that the map $\wideparen{U}(\mathcal{A}_{\omega})(X)\rightarrow \wideparen{U}(\mathcal{A}_{\omega})(Y)$ is $G$-equivariant.  By Proposition \ref{action on operators}, the action of $G$ on both of these algebras is continuous. Hence, it suffices to see that  $U(\mathcal{A}_{\omega}(X))\rightarrow U(\mathcal{A}_{\omega}(Y))$ is $G$-equivariant. This is equivalent to $\mathcal{A}_{\omega}(X)\rightarrow \mathcal{A}_{\omega}(Y)$ being $G$-equivariant. However, the restriction maps $\OX_X(X)\rightarrow \OX_X(Y)$, and $\mathcal{T}_{X/K}(X)\rightarrow \mathcal{T}_{X/K}(Y)$ are $G$-equivariant. Thus, the result follows from the description of the $G$-action on 
$\mathcal{A}_{\omega}(X)$ and $\mathcal{A}_{\omega}(Y)$ given in the proof of  Lemma \ref{extension of the action of G to differential forms}. 
\end{proof} 
We now deal with the general case:
\begin{Lemma}\label{Lemma first step non-local G-equivariant classification}
Let $X$ be a smooth rigid $G$-variety with an Atiyah algebra $\mathcal{A}$. There is a $G$-action of  sheaves of $K$-Lie algebras on $\mathcal{A}$ such that the short exact sequence:
\begin{equation*}
    0\rightarrow \pi_*\OX_X\rightarrow \pi_*\mathcal{A}\rightarrow \pi_*\mathcal{T}_{X/K}\rightarrow 0,
\end{equation*}
is $G$-equivariant if and only if $\mathcal{A}$ satisfies $[\mathcal{A}]\in \mathbb{H}^1(X,(\Omega_{X/K}^{\geq 1})^{G})$.
\end{Lemma}
\begin{proof}
 As $X$ satisfies $(\operatorname{G-Aff})$, we may choose an admissible cover $\mathfrak{U}$ of $X$ by $G$-invariant admissible open subspaces. By definition we have:
\begin{equation*}
    \Omega_{X/K}^{\geq 1}:=\left(0\rightarrow \Omega_{X/K}^{1}\rightarrow  \Omega^{2,\operatorname{cl}}_{X/K}\rightarrow 0\right).
\end{equation*}
As we have $\operatorname{H}^i(U,\Omega_{X/K}^1)=0$ for all $i\geq 1$ and all $U\in \mathfrak{U}$, we may use the \v{C}ech cohomology spectral sequence associated to the cover $\mathfrak{U}$ (\emph{cf.} \cite[Tag 08BN]{stacks-project}), and it follows that  the isomorphism class $[\mathcal{A}]\in \mathbb{H}^1(X,\Omega_{X/K}^{\geq 1})$ is represented by a cocycle:
\begin{equation*}
    (\{\omega_{U}\}_{U\in\mathfrak{U}}, \{ \delta_{U\cap V}\}_{U,V\in \mathfrak{U}}) \in \prod_{U\in \mathfrak{U}}\Omega^{2,\operatorname{cl}}_{X/K}(U) \oplus \prod_{U,V\in \mathfrak{U}}\Omega_{X/K}^{1}(U\cap V).
\end{equation*}
It follows that $[\mathcal{A}]\in \mathbb{H}^1(X,(\Omega_{X/K}^{\geq 1})^{G})$ if and only if
$\omega_U\in\Omega_{X/K}^{2,\operatorname{cl}}(U)^G$ for all $U\in \mathfrak{U}$, and
$\delta_{U\cap V}\in \Omega_{X/K}^{1}(U\cap V)^G$ for all $U,V\in \mathfrak{U}$. By
Proposition \ref{prop transition morphisms are G-equivariant}, for every $U\in\mathfrak{U}$ there is a $G$-action of sheaves of $K$-Lie algebras on $\mathcal{A}_{\omega_U}$ satisfying the conditions we want if and only if $\omega_U\in\Omega_{X/K}^{2,\operatorname{cl}}(U)^G$ for all $U\in \mathfrak{U}$.  Similarly, given $U,V\in \mathfrak{U}$, the the maps $\delta_{U\cap V}:\mathcal{A}_{\omega_U\vert U\cap V}\rightarrow \mathcal{A}_{\omega_V\vert U\cap V}$  are of the form $\delta_{U,V}((a,v))=(a+\delta_{U,V}(v),v)$. Hence, they are  are $G$-equivariant if and only if for all $v\in\mathcal{T}_{X/K}(X)$ we have: $\delta_{U,V}(v)=g^{-1}\delta_{U,V}(g(v))$, which is equivalent to $\delta_{U,V} \in \Omega_{X/K}^{1}(X)^G$.
\end{proof}
\begin{Lemma}\label{lemma extension ofa ction to TDo}
Let $X$ be a smooth rigid $G$-variety with an Atiyah algebra $\mathcal{A}$ such that $[\mathcal{A}]\in \mathbb{H}^1(X,\Omega_{X/K}^{\geq 1})^{G}$.  Then the $G$-action on $\pi_*\mathcal{A}$ extends uniquely to a $G$-action of sheaves of complete topological $K$-algebras on $\pi_*\wideparen{\D}_{\mathcal{A}}$.
\end{Lemma}
\begin{proof}
Follows at once from Proposition  \ref{prop transition morphisms are G-equivariant} using Proposition \ref{G-invariant coverings}.
\end{proof}
Next, we give a $G$-equivariant version of $\mathscr{PA}(X)$:
\begin{defi}
Let $X$ be a smooth $G$-variety. We define the category of $G$-equivariant Atiyah algebras 
$\mathscr{PA}(X)^G$ as follows:
\begin{enumerate}[label=(\roman*)]
    \item Objects are given by $\pi_*\mathcal{A}$ such that $[\mathcal{A}] \in \mathbb{H}^1(X,\Omega_{X/K}^{\geq 1})^G$.
    \item Morphisms are defined as follows: Let $\mathcal{A}, \mathcal{B}\in \mathscr{PA}(X)^G$,  we define:
    \begin{equation*}
        \Hom_{\mathscr{PA}(X)^G}(\mathcal{A}, \mathcal{B}),
    \end{equation*}
    as space of morphisms $\varphi:\mathcal{A}\rightarrow \mathcal{B}$ in $\mathscr{PA}(X)$ such that the induced map:
    \begin{equation*}
        \pi_*(\varphi):\pi_*\mathcal{A}\rightarrow \pi_*\mathcal{B},
    \end{equation*}
   is a $G$-equivariant morphism of sheaves of $K$-Lie algebras.
\end{enumerate}   
\end{defi}
We can now  obtain a $G$-equivariant classification of the Atiyah algebras on $X$:
\begin{teo}\label{teo global G-equivariant classification}
Let $X$ be a smooth $G$-variety. The following hold:
\begin{enumerate}[label=(\roman*)]
\item $\mathscr{PA}(X)^G$ is an essentially small groupoid.
    \item $\operatorname{Iso}\left(\mathscr{PA}(X)^G\right)=\mathbb{H}^1(X,(\Omega_{X/K}^{\geq 1})^{G})=\mathbb{H}^1(X,\Omega_{X/K}^{\geq 1})^{G}$.
    \item For $\mathcal{A}\in \mathscr{PA}(X)^G$ we have $\Aut_{\mathscr{PA}(X)^G}(\mathcal{A})=\Omega_{X/K}^{1,\operatorname{cl}}(X)^G$ .
\end{enumerate}    
\end{teo}
\begin{proof}
We will first show claim $(ii)$. First, notice that the functor $(-)^G$ is exact. Indeed, as $\operatorname{Char}(K)=0$
and $G$ is finite, it follows from \cite[Corollary 1.6]{fulton2013representation} that $(-)^G$ is exact on finite-dimensional representations. As  every $K$-linear $G$-representation is a filtered colimit of finite-dimensional $G$-representations, and $(-)^G$ commutes with filtered colimits, it follows that $(-)^G$ is exact. Hence, we obtain the following:
\begin{equation*}
    \mathbb{H}^1(X,(\Omega_{X/K}^{\geq 1})^{G})=\mathbb{H}^1(X,\Omega_{X/K}^{\geq 1})^{G}.
\end{equation*}
As both $\operatorname{Iso}\left(\mathscr{PA}(X)^G\right)$ and $\mathbb{H}^1(X,(\Omega_{X/K}^{\geq 1})^{G})$ can be calculated using a cover, we may assume that $X=\Sp(A)$ and admits an étale map $X\rightarrow \mathbb{A}^n_K$.  Thus, we have:
\begin{equation*}
    \mathbb{H}^1(X,\Omega_{X/K}^{\geq 1})^{G}=\operatorname{H}_{\operatorname{dR}}^2(X)^G=\Omega_{X/K}^{2,\operatorname{cl}}(X)^G/d(\Omega_{X/K}^{1}(X)^G).
\end{equation*}
By definition, objects in $\mathscr{PA}(X)^G$ are of the form $\pi_*\mathcal{A}_{\omega}$  for some $\omega\in\Omega_{X/K}^{2,\operatorname{cl}}(X)^G$. By 
Proposition \ref{prop classification of atiyah algebras on smooth affinoid varieties}, two Atiyah algebras $\mathcal{A}_{\omega}$, $\mathcal{A}_{\xi}$  are isomorphic in $\mathscr{PA}(X)$ if and only if we have: $\omega-\xi=d\eta$, for some $\eta \in \Omega_{X/K}^{1}(X)$. Thus, we need to see that the isomorphism 
$\varphi_{\eta}:\mathcal{A}_{\omega}\rightarrow \mathcal{A}_{\xi}$ is $G$-equivariant if and only if $\eta \in \Omega_{X/K}^{1}(X)^G$. However, this was shown in Lemma \ref{Lemma first step non-local G-equivariant classification}. Thus statement $(ii)$  holds, and an analogous argument can be applied to show $(iii)$. The fact that $\mathscr{PA}(X)^G$ is a groupoid stems from the fact that $\mathscr{PA}(X)$ is a groupoid.
\end{proof}
Let us point out that, as a consequence of the discussion above, every object $\pi_*\mathcal{A}$ in $\mathscr{PA}(X)^G$ has attached a sheaf of complete topological algebras $\pi_*\wideparen{\D}_{\mathcal{A}}$ with an action of $G$, and that every map $\pi_*\mathcal{A}\rightarrow \pi_*\mathcal{B}$ lifts to a $G$-equivariant morphism of sheaves of topological algebras $\pi_*\wideparen{\D}_{\mathcal{A}}\rightarrow \pi_*\wideparen{\D}_{\mathcal{B}}$. Additionally, we remark that, if the action of $G$ on $X$ is free, then the quotient map $\pi:X\rightarrow X/G$ is étale, and $X/G$ is a smooth rigid space. In this case, the classification obtained in Theorem \ref{teo global G-equivariant classification} agrees with the classification of Atiyah algebras on $X/G$. Let us now introduce the skew group algebra construction:
\begin{defi}
Let $A$ be a ring with an action of $G$, and let $G\ltimes A:=\bigoplus_{g\in G}A$. We express a family $\{a_{g}\}_{g\in G}\in G\ltimes A$ as $\sum_{g\in G}a_{g}g$. We can endow $G\ltimes A$ with a unital associative algebra structure by setting:
\begin{equation*}
    (\sum_{g\in G}a_{g}g)\cdot (\sum_{h\in G}b_{h}h)=\sum_{g,h\in G}a_{g}g(b_{h})gh. 
\end{equation*}
We call $G\ltimes A$ the skew group algebra of $G$ and $A$.
\end{defi}
We will now extend this to the algebras $\pi_*\wideparen{\D}_{\mathcal{A}}$ defined above.
\begin{Lemma}\label{lemma flatness of skew algebras}
Let $f:A\rightarrow B$ be $G$-equivariant morphism of algebras. If $f$ is (faithfully-)flat then the induced map $G\ltimes A \rightarrow G\ltimes B$ is (faithfully-)flat.
\end{Lemma}
\begin{proof}
Follows from the identity $G\ltimes B=B\otimes_A(G\ltimes A)$.
\end{proof}
We can use this machinery to complete our description of the algebras $G\ltimes \wideparen{\mathcal{D}}_{\omega}(X)$:
\begin{prop}\label{skew tdo are fre-st}
Let $X=\Sp(A)$ be a smooth $G$-variety and  $\omega\in\Omega^{2,\operatorname{cl}}_{X/K}(X)^G$. 
Let $\mathscr{A}_{\omega}$ be a $G$-invariant $(\mathcal{R},A^{\circ})$-Lie lattice  of $\mathcal{A}_{\omega}(X)$, and set:
\begin{equation*}
    \widehat{D}_{\omega}^n:= \widehat{U}(\pi^n\mathscr{A}_{\omega})_K\otimes_{U(\mathscr{A}_{\omega}(X))}\D_{\omega}(X).
\end{equation*}
Then $G\ltimes \wideparen{\mathcal{D}}_{\omega}(X)$ admits a two-sided Fréchet-Stein presentation:
\begin{equation*}
    G\ltimes \wideparen{\mathcal{D}}_{\omega}(X) =\varprojlim_{n} G\ltimes \widehat{\D}_{\omega}^n.
\end{equation*}
\end{prop}
\begin{proof}
By equation  $(\ref{equation FS presentation of completed TDO})$, we have that $\wideparen{\mathcal{D}}_{\omega}(X) =\varprojlim_{n} \widehat{\D}_{\omega}^n$ is a Fréchet-Stein presentation. Furthermore, by the proof of Proposition \ref{action on operators}, the transition maps $\widehat{\D}_{\omega}^{n+1}\rightarrow \widehat{\D}_{\omega}^n$ are $G$-equivariant. Thus, by Lemma \ref{lemma flatness of skew algebras} we have the following two-sided Fréchet-Stein presentation:
\begin{equation*}
    \varprojlim_n G\ltimes \widehat{\D}_{\omega}^n =\varprojlim_n \bigoplus_{g\in G} \widehat{\D}_{\omega}^ng= \bigoplus_{g\in G}\varprojlim_n\widehat{\D}_{\omega}^ng=\bigoplus_{g\in G}\wideparen{\mathcal{D}}_{\omega}(X)g=G\ltimes \wideparen{\mathcal{D}}_{\omega}(X).
\end{equation*}
\end{proof}
Finally, we can condense all these results into the following theorem:
\begin{teo}
Let $X$ be a smooth $G$-variety and $\mathcal{A}\in \mathscr{PA}(X)^G$. The sheaf:
\begin{equation*}
    G\ltimes \wideparen{\D}_{\mathcal{A}}:=G\ltimes \pi_*\wideparen{\D}_{\mathcal{A}},
\end{equation*}
is a sheaf of complete topological algebras such that for every $G$-invariant affinoid open $Y\subset X$ with an étale map $Y\rightarrow \mathbb{A}^n_K$ the following algebra is Fréchet-Stein:
\begin{equation*}
    G\ltimes \wideparen{\D}_{\mathcal{A}}(Y/G)=G\ltimes \wideparen{\D}_{\mathcal{A}}(Y).
\end{equation*}
If $X=\Sp(A)$ wuth an étale map $X\rightarrow \mathbb{A}^n_K$, then $G\ltimes \wideparen{\D}_{\mathcal{A}}$ has trivial higher \v{C}ech cohomology groups, and for any $G$-invariant subdomain $Y=\Sp(B)\subset X$ the  map:
    \begin{equation*}
      G\ltimes\wideparen{\D}_{\omega}(X)\rightarrow G\ltimes\wideparen{\D}_{\omega}(Y),  
    \end{equation*}
    is a $c$-flat morphism of Fréchet-Stein algebras. Furthermore, if $(U_i)_{i=1}^n$ is an affinoid cover of $X$, then the associated map:
    \begin{equation*}
        G\ltimes\wideparen{\D}_{\omega}(X)\rightarrow \prod_{i=1}^nG\ltimes\wideparen{\D}_{\omega}(U_i),
    \end{equation*}
    is also $c$-faithfully flat.

\end{teo}
\begin{proof}
This is a consequence of Lemma \ref{lemma extension ofa ction to TDo}, Propositions \ref{prop transition morphisms are G-equivariant}, and \ref{skew tdo are fre-st}, and the fact that $G\ltimes-$ is an exact functor.
\end{proof}
\subsection{Extension to the étale site}\label{section extension to the étale site}
Given that Cherednik algebras have been successfully employed to study quotient singularities and algebraic fundamental groups of quotient schemes, we feel it is beneficial to define these objects at the étale level. For this reason, we finish the section on twisted differential operators by extending Atiyah algebras and TDO to the étale site of $X$. In particular, we show that any Atiyah algebra $\mathcal{A}$ on $X$ extends canonically to a sheaf of Atiyah algebras on $X_{\textnormal{ét}}$. This induces an extension of the sheaf of twisted differential operators $\D_{\mathcal{A}}$ to $X_{\textnormal{ét}}$. Furthermore, if $[\mathcal{A}]\in \operatorname{H}^2_{\operatorname{dR}}(X)^G$, then we will see that $G\ltimes \D_{\mathcal{A}}$ extends uniquely to a sheaf on $X/G_{\textnormal{ét}}$. Let us start by recalling the definition of the small étale site: 
\begin{defi}[{\cite[Section 3.2]{de1996etale}}]
Let $X$ be a rigid analytic space. We define the small étale site of $X$, $X_{\textnormal{ét}}$, as the site with objects étale morphisms $Y\rightarrow X$, and covers given by families $\{\varphi_i:Y_i\rightarrow W \}_{i\in I}$ of jointly surjective maps such that for each affinoid open subspace $U\subset W$ there is a finite set $J\subset I$ such that the family $\{\varphi_j:\varphi_j^{-1}(U)\rightarrow U \}_{j\in J}$ is jointly surjective.
\end{defi}

\begin{prop}\label{prop extension of atiyah algebras to the small étale site}
Let $X=\Sp(A)$ be a smooth affinoid space with an Atiyah algebra $\mathcal{A}$. Then $\mathcal{A}$ extends uniquely to a coherent sheaf on $X_{\textnormal{ét}}$. This sheaf satisfies the following properties:
\begin{enumerate}[label=(\roman*)]
    \item Let $Y=\Sp(B)\rightarrow X$ be étale. Then $\mathcal{A}(Y)=B\otimes_A\mathcal{A}(X)$ is an Atiyah algebra on $Y$.
    \item The restriction map $\mathcal{A}(X)\rightarrow \mathcal{A}(Y)$ is a morphism of $(K,A)$-Lie algebras.
    \item Assume $\mathcal{A}(X)$ is given by a closed $2$-form $\omega\in \Omega^{2,\operatorname{cl}}_{X/K}(X)$. Then $\mathcal{A}(Y)$ is the Atiyah algebra associated to the image of $\omega$ under the canonical morphism:
    \begin{equation*}
        \Omega^{2}_{X/K}(X)\rightarrow \Omega^{2}_{Y/K}(Y).
    \end{equation*}
\end{enumerate}
\end{prop}
\begin{proof}
By definition, $\mathcal{A}$ is a coherent sheaf on $X$. Thus, by \cite[Corollary 3.2.6]{de1996etale}, it extends canonically to a coherent sheaf on $X_{\textnormal{ét}}$. Furthermore, this extension satisfies that for any étale map $Y=\Sp(B)\rightarrow X$ we have $\mathcal{A}(Y)=B\otimes_A\mathcal{A}(X)$. We need to show that  $\mathcal{A}(Y)$ is an Atiyah algebra on $Y$. As the map $A\rightarrow B$ is flat, we have the following commutative diagram with exact rows:
\begin{equation*}
   % https://tikzcd.yichuanshen.de/#N4Igdg9gJgpgziAXAbVABwnAlgFyxMJZABgBpiBdUkANwEMAbAVxiRGJAF9T1Nd9CKMgEYqtRizYduvbHgJEALOTH1mrROy48QGOQKWlR1NZM3Sde-gpTCVJiRpAAdZwHkAGgH0AmgAofAEptWWtBZAAme3F1NlcAWzocAAsAY0ZgAEFOAOCZXT55cIBmaNMnBKS0jIAVTi9gHwB6AGkcoJCC-RtkO2MYsxd3bw8-DzzLQoMUKP7yuOdElPSGLJzxzqsiolK5xwWl6tW6ho9W9byxGCgAc3giUAAzACcIeKQADmocCCQAVnyLzeSAA7N9fohFIDXu9IeD-tDgYgwSAfp9EbCvqiIQBODH-eGIABs+OJhIipJxhOKlMJJJ0QNhZGxSChDJhSDsLORnAonCAA
\begin{tikzcd}
0 \arrow[r] & \OX_Y(Y) \arrow[r]           & \mathcal{A}(Y) \arrow[r]           & \mathcal{T}_{Y/K}(Y) \arrow[r]           & 0 \\
0 \arrow[r] & \OX_X(X) \arrow[u] \arrow[r] & \mathcal{A}(X) \arrow[u] \arrow[r] & \mathcal{T}_{X/K}(X) \arrow[r] \arrow[u] & 0
\end{tikzcd} 
\end{equation*}
where the leftmost map corresponds to the induced map $A\rightarrow B$, and on the rightmost element on the upper sequence we are using that $\mathcal{T}_{Y/K}(Y)=B\otimes_A\mathcal{T}_{X/K}(X)$. By \cite[Section 2.4]{ardakov2019}, there is a unique $(K,B)$-Lie algebra structure on $\mathcal{A}(Y)$ such that the morphism $\mathcal{A}(X)\rightarrow \mathcal{A}(Y)$ is a morphism of $(K,A)$-Lie algebras. The previous diagram then shows that $\mathcal{A}(Y)$ is an Atiyah algebra on $Y$. Furthermore,  as $Y\rightarrow X$ is étale, we have an identification $\Omega^2_{Y/K}(Y)=B\otimes_A\Omega^2_{X/K}(X)$. It follows by commutativity of the previous diagram and the explicit description of the Lie bracket of an Atiyah algebra given in equation $(\ref{bracket of tdo})$ that statement $(iii)$ also holds.
\end{proof}
\begin{coro}\label{coro extension tdo to the étale site}
Let $X=\Sp(A)$ be a smooth affinoid space with an Atiyah algebra $\mathcal{A}$. There is a unique extension of $\D_{\mathcal{A}}$ to  $X_{\textnormal{ét}}$ satisfying the following properties:
\begin{enumerate}[label=(\roman*)]
    \item Let $Y=\Sp(B)\in X_{\textnormal{ét}}$. Then $\D_{\mathcal{A}}(Y)=B\otimes_A\D_{\mathcal{A}}(X)$. 
    \item For any $Y\in X_{\textnormal{ét}}$, let $Y_{\textnormal{an}}$ denote the small analytic  site of $Y$. There is a canonical isomorphism of sheaves of filtered $K$-algebras:
    \begin{equation*}
        \left(\D_{\mathcal{A}}\right)_{\vert Y_{\textnormal{an}}} \rightarrow \D_{\mathcal{A}_{\vert Y_{\textnormal{an}}}}.
    \end{equation*}
\end{enumerate}
We will also call the extension $\D_{\mathcal{A}}$ the sheaf of $\mathcal{A}$-twisted differential operators.
\end{coro}
\begin{proof}
In light of Definition \ref{defi tdo} and Proposition \ref{prop extension of atiyah algebras to the small étale site}, we only need to show that we have an isomorphism of filtered $K$-algebras:
\begin{equation*}
    B\otimes_A\D_{\mathcal{A}(X)}(X)\rightarrow\D_{\mathcal{A}(Y)}(Y).
\end{equation*}
By definition of TDO, there is a short exact sequence of $\OX_X(X)$-modules:
\begin{equation*}
    0\rightarrow \mathcal{I}_{\mathcal{A}}(X)\rightarrow U(\mathcal{A}(X))\rightarrow \D_{\mathcal{A}(X)}(X),
\end{equation*}
By \cite[Proposition 2.3]{ardakov2019}, we have $U(\mathcal{A}(Y))=B\otimes_A U(\mathcal{A}(X))$. Furthermore, by construction, $\mathcal{I}_{\mathcal{A}}(X)$ is the two-sided ideal generated by $1_{\mathcal{A}(X)}-1_{U(\mathcal{A}(X))}$. Hence, we have $\mathcal{I}_{\mathcal{A}}(Y)=B\otimes_A\mathcal{I}_{\mathcal{A}(X)}(X)$. As $A\rightarrow B$ is flat, applying $B\otimes_A-$ to this sequence yields:
\begin{equation*}
    0\rightarrow \mathcal{I}_{\mathcal{A}}(Y)\rightarrow U(\mathcal{A}(Y)) \rightarrow B\otimes_A \D_{\mathcal{A}(X)}(X),
\end{equation*}
which shows $\D_{\mathcal{A}}(Y)=B\otimes_A\D_{\mathcal{A}(X)}(X)$, as we wanted to show.
\end{proof}
We can translate these results to the $G$-equivariant setting without much change.
\begin{Lemma}
    Let $X$ be a $G$-variety and choose an étale map $Y=\Sp(B)\rightarrow X/G$. The pullback $Z:=Y\times_{X/G}X$ admits a unique action of $G$ satisfying the following:
    \begin{enumerate}[label=(\roman*)]
        \item The projection $Z\rightarrow X$ is $G$-equivariant.
        \item $Z$ is a $G$-variety.
        \item $Z/G =Y$.
    \end{enumerate}
    We will always regard such pullbacks as $G$-varieties with respect to this action.
\end{Lemma}
\begin{proof}
Choose $g\in G$, we define the action $g:Z\rightarrow Z$ by setting $g$ to act as the identity on $Y$, and act on $X$ via the original action.
\end{proof}
\begin{Lemma}\label{lemma extension G-equivariant affinoid case}
Let $X=\Sp(A)$ be a smooth $G$-variety with an Atiyah algebra $\mathcal{A}$, and assume that $[\mathcal{A}]\in \operatorname{H}_{\operatorname{dR}}^2(X)^G$. Then the sheaf $G\ltimes \D_{\mathcal{A}}$ on $X/G$ extends uniquely to a sheaf  on $X/G_{\textnormal{ét}}$. Moreover, for any étale map $Y=\Sp(B)\rightarrow X/G$, let $Z=Y\times_{X/G}X$. Then we have an isomorphism of sheaves of filtered algebras:
\begin{equation*}
    (G\ltimes \D_{\mathcal{A}})_{\vert Y_{\textnormal{an}}}=G\ltimes \D_{\mathcal{A}\vert Z}:=G\ltimes \pi_*\D_{\mathcal{A}\vert Z}. 
\end{equation*}
\end{Lemma}
\begin{proof}
In light of Corollary \ref{coro extension tdo to the étale site}, we only need to show that for any étale map $Y=\Sp(B)\rightarrow X/G$, the $G$-action on $\mathcal{T}_{Z/K}$ extends uniquely to a $G$-action on the algebra of twisted differential operators $\D_{\mathcal{A}\vert Z}$. By Corollary \ref{extension of the action of G to differential forms}, it suffices to show that $[\mathcal{A}_{\vert Z}]\in \operatorname{H}_{\operatorname{dR}}^2(Z)^G$. However, by assumption we have  $[\mathcal{A}]\in \operatorname{H}_{\operatorname{dR}}^2(X)^G$, and by statement $(iii)$ in Proposition \ref{prop extension of atiyah algebras to the small étale site} the class $[\mathcal{A}_{\vert Z}]$ agrees with the image of $[\mathcal{A}]$ under  $\operatorname{H}_{\operatorname{dR}}^2(X)\rightarrow \operatorname{H}_{\operatorname{dR}}^2(Z)$. As this map is $G$-equivariant, the result holds.
\end{proof}
As usual, the uniqueness of these constructions allows us to globalize the results to more general smooth $G$-varieties:
\begin{teo}
Let $X$ be a smooth and separated $G$-variety with an Atiyah algebra $\mathcal{A}$ satisfying that $[\mathcal{A}]\in \mathbb{H}^1(X,\Omega_{X/K}^{\geq 1})^G$. Then $G\ltimes \D_{\mathcal{A}}$ extends uniquely to a sheaf of filtered $K$-algebras on $X/G_{\textnormal{ét}}$. Moreover, for any étale map $Y\rightarrow X/G$, let $Z=Y\times_{X/G}X$. Then we have an isomorphism of sheaves of filtered algebras:
\begin{equation*}
    (G\ltimes \D_{\mathcal{A}})_{\vert Y_{\textnormal{an}}}=G\ltimes \D_{\mathcal{A}\vert Z}:=G\ltimes \pi_*\D_{\mathcal{A}\vert Z}. 
\end{equation*}
\end{teo}
\begin{proof}
As $X$ is separated and $\pi:X\rightarrow X/G$ is finite, we may assume that $X/G$ is affinoid. It suffices to show that $G\ltimes \D_{\mathcal{A}}$ extends to the full subsite $X/G_{\textnormal{ét,aff}}$ given by the étale affinoid maps $Y\rightarrow X$. This was done in Lemma \ref{lemma extension G-equivariant affinoid case}.
\end{proof}
We remark that the sheaves $G\ltimes \wideparen{\D}_{\mathcal{A}}$ can also be extended to the small étale site. However, we will delay this until Section \ref{sheaf of p-adic}, where we obtain this extension as a consequence of the corresponding fact for $p$-adic Cherednik algebras.
\section{Cherednik algebras on rigid analytic spaces}\label{section CHerednik algerbas}
We will now use the material of the previous sections to construct sheaves of Cherednik algebras on a smooth $G$-variety $X$ which satisfies $\operatorname{(G-Aff)}$. We will start by constructing Cherednik algebras in the case where $X$ is affinoid and has an étale map $X\rightarrow \mathbb{A}^r_K$, and then globalize the construction to obtain a presheaf on $X/G_{\textnormal{ét}}$. In order to show that these presheaves are in fact sheaves, we will show a PBW Theorem for Cherednik algebras, following the procedures in \cite{etingof2004cherednik}. 
\subsection{Cherednik algebras on rigid analytic spaces}\label{Section 3.1}
In this section, we will translate the classical theory of Cherednik algebras from \cite{etingof2004cherednik} to the rigid analytic setting. It is important to point out that, although all results in \cite{etingof2004cherednik} use the complex numbers as the base field, most of them also hold when the base field is an arbitrary field of characteristic zero. Thus, even if we only focus on the rigid-geometric context, the results of this section also have validity for algebraic $K$-varieties. Furthermore, the proofs are completely analogous.
\begin{obs}
From now on and until the end of the section, we will only consider smooth rigid $G$-varieties satisfying $(\operatorname{G-Aff})$ (cf. Theorem \ref{teo existence of quotients of rigid spaces by finite groups}).
\end{obs}
\begin{defi}\label{defi basic components cherednik algebras}
Let $X$ be a smooth $G$-variety. We define the following objects:
\begin{enumerate}[label=(\roman*)]
    \item The action of $G$ on $S(X,G)$ by conjugation is given by the formula:
    \begin{equation*}
      h(Y,g)=(h(Y),h^{-1}gh) \textnormal{ for }   h\in G \textnormal{ and }(Y,g)\in S(X,G).  
    \end{equation*}   
    \item A reflection function on $X$ is a $G$-invariant function $c:S(X,G)\rightarrow K$. We denote the $K$-vector space of reflection functions by $\operatorname{Ref}(X,G)$. 
\end{enumerate}
\end{defi}
Assume $X=\Sp(A)$ is affinoid and $\Omega^1_{X/K}(X)$ has a basis given by differentials of functions. As $G$ is finite and $X$ is quasi-compact, $S(X,G)$ is a finite set. Hence, reflection functions form a finite dimensional vector space. We want to define a family of algebras parameterized by the following elements:
\begin{enumerate}
    \item Non-zero constants $t\in K^*$.
    \item Reflection functions $c:S(X,G)\rightarrow K$.
    \item $G$-invariant closed 2-forms $\omega\in \Omega^{2,\operatorname{cl}}_X(X)^{G}$.
\end{enumerate}
In particular, let us make the following definition:
\begin{defi}\label{Basic definitions in Cher algebras}
Let $X=\Sp(A)$ be a smooth affinoid $G$-variety. Choose $t\in K^{*}$, $c\in \operatorname{Ref}(X,G)$, and $\omega\in\Omega_{X/K}^{2,\operatorname{cl}}(X)^G$. We define the following objects:
\begin{enumerate}[label=(\roman*)]
\item $\overline{X}=X\setminus \bigcup_{(Y,g)\in S(X,G)} Y$. 
\item Let $U\subset \overline{X}$ be a $G$-invariant affinoid subspace meeting all connected components of $X$. A Dunkl-Opdam operator for a derivation $v\in\mathcal{T}_{X/K}(X)$ is an element of the algebra $G\ltimes \mathcal{D}_{\omega/t}(U)$ given by the expression:
\begin{equation*}
    D_{v}=t\mathbb{L}_{v}+\sum_{(Y,g)\in S(X,G)}\frac{2c(Y,g)}{1-\lambda_{Y,g}}\overline{\xi_{Y}}(v)(g-1),
\end{equation*}
where $\lambda_{Y,g}$ is the eigenvalue of $g$ on $\mathcal{N}_Y^*$ (cf. Proposition \ref{constant-eigenfunctions}), and $\overline{\xi_{Y}}(v)$ is a representative in $\OX_{Z}(X)$ of the coset  of $\xi_{Y}(v)$ (cf. Proposition \ref{residue map}).
\item The Cherednik algebra $H_{t,c,\omega}(X,G)$ is the subalgebra of $G\ltimes\mathcal{D}_{\omega/t}(U)$ generated by $G, \OX_{X}(X)$, and a  Dunkl-Opdam operator $D_v$ for each $v\in \mathcal{T}_{X/K}(X)$.
\end{enumerate}
\end{defi}
\begin{obs}\label{obs scaling the parameter}
Notice that  for any $\lambda\in K^{*}$ We have $H_{t,c,\omega}(X,G)=H_{\lambda t,\lambda c,\lambda\omega}(X,G)$.
\end{obs}
A priory, the Cherednik algebra $H_{t,c,\omega}(X,G)$ depends on several choices. However, we will see in the following lemmas that it only depends on $t,c$, and $\omega$:
\begin{Lemma}\label{lemma independence of cher alg of U}
The following hold:
\begin{enumerate}[label=(\roman*)]
    \item $\overline{X}$ is a  $G$-invariant Zariski open subspace of $X$.
    \item $H_{t,c,\omega}(X,G)$ does not depend on the choice of $U$.
\end{enumerate}
\end{Lemma}
\begin{proof}
Let $Z$ be the union of all reflection hypersurfaces. As this is a $G$-invariant closed subvariety of $X$, it follows that there are some $f_1,\cdots,f_r\in A^G$ such that $Z=\mathbb{V}(f_1,\cdots,f_r)$. Hence, we have an admissible cover $\overline{X}=\bigcup_{n\geq 0}X\left(\frac{\pi^n}{f_1},\cdots, \frac{\pi^n}{f_r}\right)$, and
 every $\overline{X}_n=X\left(\frac{\pi^n}{f_1},\cdots, \frac{\pi^n}{f_r}\right)$ is a $G$-invariant affinoid subdomain of $X$. Choose two $G$-invariant affinoid subdomains $U,V\subset \overline{X}$ meeting all connected components of $X$. There is some $n\geq 0$ such that  $U,V\subset \overline{X}_n$, and we may assume
 that $U$ and $V$ meet all connected components of $\overline{X}_n$. Thus, the following maps are injective:
\begin{equation*}
    \OX_X(\overline{X}_n)\rightarrow \OX_X(U), \textnormal{ }\OX_X(\overline{X}_n)\rightarrow \OX_X(V),
\end{equation*}
But then, Proposition \ref{prop PBW for twisted differential operators} implies that the maps $G\ltimes\mathcal{D}_{\omega/t}(\overline{X}_n)\rightarrow G\ltimes\mathcal{D}_{\omega/t}(U)$ and $G\ltimes\mathcal{D}_{\omega/t}(\overline{X}_n)\rightarrow G\ltimes\mathcal{D}_{\omega/t}(V)$ are injective as well. As the operators defining $H_{t,c,\omega}(X,G)$ are already contained in $G\ltimes\mathcal{D}_{\omega/t}(\overline{X}_n)$, it follows that $H_{t,c,\omega}(X,G)$ does not depend on $U$.
\end{proof}
\begin{obs}
For the rest of the section, we fix a $G$-invariant affinoid subdomain $U\subset \overline{X}$ which meets all connected components of $X$.
\end{obs}
Notice that definition $(ii)$ does not yield a single element, but a family of them. Namely, the different choices of representative of $\xi_{Y}(v)$ will give rise to different Dunkl-Opdam operators associated to the same derivation $v\in\mathcal{T}_{X/K}(X)$. This does not produce any ambiguity, as our definition of Cherednik algebra does not depend on the choice of representatives. 
\begin{Lemma}\label{independence of representatives of DO}
The Cherednik algebra $H_{t,c,\omega}(X,G)$ is independent on the choices of representatives of the Dunkl-Opdam operators.
\end{Lemma}
\begin{proof}
Consider a derivation $v\in\mathcal{T}_{X/K}(X)$, and let $D_v,D_v'$ be two Dunkl-Opdam operators associated to it. Then we have:
\begin{equation*}
    D_v-\overline{D}_v=\sum_{(Y,g)\in S(X,G)}\frac{2c(Y,g)}{1-\lambda_{Y,g}}h_{Y}(g-1).
\end{equation*}
where $h_{Y}=\overline{\xi_{Y}}(v)-\overline{\xi_{Y}}(v)'$.  As $\overline{\xi_{Y}}(v)$ and $\overline{\xi_{Y}}(v)'$ are representatives in $\OX_X(Z)$ of $\xi_Y(v)$, it follows that $h_{Y}\in \OX_X(X)$ (cf. Proposition \ref{residue map}). The eigenvalues $\lambda_{Y,g}$ are non-trivial units in $K$. Hence, it follows that $ D_v-\overline{D}_v$ is contained in  $ G\ltimes \OX_X(X)\subset  G\ltimes \mathcal{D}_{\omega/t}(U)$. Thus, $H_{t,c,\omega}(X,G)$ contains all the differences $ D_v-\overline{D}_v$, and therefore it contains all choices of Dunkl-Opdam operators.
\end{proof}
\begin{obs}\label{remark standard Dunkl-Opdam}
    In general, we do not have a canonical choice of Dunkl-Opdam operators. However, in some special situations there is a natural way of making such a choice. Namely, assume that every $(Y,g)\in S(X,G)$ is the zero locus of a rigid function $f_Y\in\OX_X(X)$. Then we have $\xi_{Y}(v)=[\frac{v(f_Y)}{f_Y}]$. In this case, for every $v\in\mathcal{T}_{X/K}(X)$ We can define its standard Dunkl-Opdam operator:
    \begin{equation*}
        D^{s}_{v}=t\mathbb{L}_{v}+\sum_{(Y,g)\in S(X,G)}\frac{2c(Y,g)}{1-\lambda_{Y,g}}\frac{v(f_{Y})}{f_{Y}}(g-1),
    \end{equation*}
    and these provide  representatives for each family of Dunkl-Opdam operators. 
\end{obs}
 Now that we have defined Cherednik algebras, our next step will be showing some of their basic properties. We start by the following very useful relations:
\begin{prop}\label{initial relations in a Cherednik algebra}
The following commutation relations hold in $H_{t,c,\omega}(X,G):$
\begin{enumerate}[label=(\roman*)]
\item $D_{v}g=gD_{g^{-1}(v)}$ for $g\in G$, and  $v\in\mathcal{T}_{X/K}(X)$.
\item $D_{fv+hw}=fD_{v}+hD_{w}$ for $f,h\in \OX_{X}(X)$, and $v,w\in \mathcal{T}_{X/K}(X)$.
\item Consider $f\in \OX_{X}(X)$ and a Dunkl-Opdam operator $D_{v}$. Then we have:
\begin{equation*}
    [D_{v},f]=tv(f) +\sum_{(Y,g)\in S(X,G)}\frac{2c(Y,g)}{1-\lambda_{Y,g}}\overline{\xi_{Y}}(v)(g(f)-f)g.
\end{equation*}
In particular, $[D_{v},f]\in G\ltimes \OX_{X}(X)$.
\end{enumerate}
\end{prop}
\begin{proof}
We only show $(iii)$. We have the following identities inside $G\ltimes \mathcal{D}_{\omega/t}(U)$:
\begin{multline*}
    D_{v}f= t\mathbb{L}_{v}f+\sum_{(Y,g)\in S(X,G)}\frac{2c(Y,g)}{1-\lambda_{Y,g}}\overline{\xi_{Y}}(v)(g-1)f\\
    = t\mathbb{L}_{fv}+tv(f)+\sum_{(Y,g)\in S(X,G)}\frac{2c(Y,g)}{1-\lambda_{Y,g}}\overline{\xi_{Y}}(v)(gf-f).
\end{multline*}
We can simplify the last term by considering the following calculation:
\begin{align*}
    \sum_{(Y,g)\in S(X,G)}\frac{2c(Y,g)}{1-\lambda_{Y,g}}\overline{\xi_{Y}}(v)(gf-f)    
    \\=\sum_{(Y,g)\in S(X,G)}\frac{2c(Y,g)}{1-\lambda_{Y,g}}\overline{\xi_{Y}}(v)(g(f)g-fg+fg-f)\\=
    \sum_{(Y,g)\in S(X,G)}\frac{2c(Y,g)}{1-\lambda_{Y,g}}\overline{\xi_{Y}}(v)(g(f)-f)g +f\sum_{(Y,g)\in S(X,G)}\frac{2c(Y,g)}{1-\lambda_{Y,g}}\overline{\xi_{Y}}(v)(g-1).
\end{align*}
Combining both expressions and rearranging the terms, we have shown that: 
\begin{equation*}
    [D_{v},f]=tv(f) +\sum_{(Y,g)\in S(X,G)}\frac{2c(Y,g)}{1-\lambda_{Y,g}}\overline{\xi_{Y}}(v)(g(f)-f)g.
\end{equation*}
We  need to show that $[D_{v},f]\in G\ltimes \OX_{X}(X)$. It suffices to show that for each $(Y,g)$, we have $\overline{\xi_{Y}}(v)(g(f)-f)\in \OX_{X}(X)$. This can be checked locally, so we may assume that $Y$ is the zero locus of some $h_Y\in \OX_X(X)$. By Proposition \ref{residue map}, it suffices to show that $h_Y$ divides $g(f)-f$ for any $f\in\OX_X(X)$. This holds because $Y\subset X^g$.
\end{proof}
The next goal will be endowing $H_{t,c,\omega}(X,G)$ with a filtered algebra structure: 
\begin{defi}\label{defi Dunkl-OPdam filtration}
The Dunkl-Opdam filtration on $H_{t,c,\omega}(X,G)$ is the unique positive and exhaustive $K$-algebra filtration satisfying the following properties:
\begin{enumerate}[label=(\roman*)]
    \item $F_{0}H_{t,c,\omega}(X,G)=G\ltimes \OX_{X}(X).$
    \item $F_{1}H_{t,c,\omega}(X,G)$ is the left $G\ltimes \OX_{X}(X)$-module generated by $1$ and the Dunkl-Opdam operators.
\end{enumerate}
\end{defi}
\begin{obs}
Notice that $H_{t,0,\omega}(X,G)=G\ltimes \D_{\omega/t}(X)$. Hence, the Dunkl-Opdam filtration endows  $G\ltimes \D_{\omega/t}(X)$ with a filtration. Furthermore, this filtration restricts to the filtration by order of differential operators on $\D_{\omega/t}(X)$. For this reason, we will denote this filtration by $\Phi_{\bullet}G\ltimes \D_{\omega/t}(X)$. 
\end{obs}
\subsection{The presheaf of Cherednik algebras}
The definition of Cherednik algebras depends on three parameters: A unit $t\in K^*$, a reflection function $c\in\operatorname{Ref}(X,G)$, and a $G$-invariant closed $2$-form $\omega\in\Omega_{X/K}^{2,\operatorname{cl}}(X)^G$. In order to construct a sheaf of Cherednik algebras on $X/G_{\textnormal{ét}}$, we need to understand how this choice of parameters determines  a unique Cherednik algebra on every affinoid object in $X/G_{\textnormal{ét}}$. Thus, we will start the section by showing that there is a sheaf $\mathscr{R}(-,G):X/G_{\textnormal{ét}}^{\op}\rightarrow \operatorname{Vect}_K$ satisfying that for any étale map $Y\rightarrow X/G$ we have:
\begin{equation*}
    \mathscr{R}(Y,G)=\operatorname{Ref}(Y\times_{X/G}X,G).
\end{equation*}
This result, together with the extension of Atiyah algebras to the étale site obtained in Section \ref{section extension to the étale site}, will allow us to build presheaves of Cherednik algebras on $X/G_{\textnormal{ét}}$.\\
Let us start the section with the following technical lemmas:
\begin{Lemma}\label{Lemma 1 sheaf of cher algebras }
Let $X$ be a smooth rigid $G$-variety, with an étale map  $f:V\rightarrow X/G$, and let $W=V\times_{X/G}X$. For any $(Y,g)\in S(X,G)$ consider the following pullback:
\begin{equation*}
   \begin{tikzcd}
 Z\arrow[d] \arrow[r] & Y \arrow[d] \\
W \arrow[r]           & X          
\end{tikzcd} 
\end{equation*}
Then $Z\subset W^g$ is either empty or a disjoint union of reflection hypersurfaces in $S(W,G)$. Furthermore, for each $g\in G$ we have $W^g=X^g\times_X W$.
\end{Lemma}
\begin{proof}
Assume $Z$ is not empty. For any rigid space $T$ we have:
\begin{equation*}
    W(T)=V(T)\times_{X/G(T)}X(T).
\end{equation*}
By construction, for any $g\in G$, the action of $g$ on $W(T)$ is given by $(v,x)g=(v,xg)$.
Let $(v,x)\in Z(T)$, then we must have $x\in Y(T)\subset X^g$. Hence, $(v,x)\in W^g(T)$.
Thus, we have $Z\subset W^g$. Therefore, in order to show that $Z$ is a disjoint union of reflection hypersurfaces, we only need to show that each of the connected components of $Z$ is a hypersurface of $W$. Let $Z_i$ be one of these connected components. Then the composition $Z_i\rightarrow Z\rightarrow Y$ is étale. By Proposition \ref{smooth locus}, $Y$ is smooth. Hence, both $Z_i$ and $Y$ have the same dimension, and we have:
\begin{equation*}
    \textnormal{dim}(Z_i)=\textnormal{dim}(Y)=\textnormal{dim}(X)-1=\textnormal{dim}(W)-1.
\end{equation*}
Thus, it follows that $Z$ is a disjoint union of reflection hypersurfaces. For the second statement, notice that the fact that the map $W\rightarrow X$ is $G$-equivariant implies that we have $W^g\subset X^g\times_X W$. The argument above shows the converse.
\end{proof}
\begin{Lemma}\label{Lemma 2 sheaf of cher algebras}
In the setting of Lemma \ref{Lemma 1 sheaf of cher algebras } let $(T,g)\in S(W,G)$. There is a unique $(Y,g)\in S(X,G)$ such that $T\subset Y\times_X W$. Moreover, if $(Z,h)\in S(X,G)$, satisfies 
$T\subset Z\times_X W$, then $Y=Z$, and there is some $\omega\in G$ such that $g,h \in\langle \omega \rangle$.
\end{Lemma}    
\begin{proof}
 By Lemma \ref{Lemma 1 sheaf of cher algebras } we have $W^g=X^g\times_X W$. In particular, the projection $W^g\rightarrow X^g$ is étale. As $T$ is a connected component of $W^g$, the map $F:T\rightarrow X^g$ is étale as well. As $T$ is connected and $X^g$ is smooth, the image of this map must be contained in only one of the connected components of $X^g$. Denote this connected component by $Y$. Then the map $T\rightarrow Y$ is étale, and both spaces are smooth. Thus, both spaces have the same dimension, so  $(Y,g)\in S(X,G)$.\\
 Consider $(Z,h)\in S(X,G)$ such that 
$T\subset Z\times_X W$. Then, by Lemma \ref{Lemma 1 sheaf of cher algebras }, the map $T\rightarrow Z$ is étale. As $T$ is smooth, we get the following composition:
\begin{equation*}
    F:T\rightarrow Y\cap Z\rightarrow Y,
\end{equation*}
where the second map is a Zariski-closed immersion. As $F$ is étale, and Zariski-closed immersions are unramified, it follows that  $Y\cap Z\rightarrow Y$ must be étale as well. But $Y$ is smooth and connected, so this is only possible if $Y=Y\cap Z$. Analogously, it follows that $Z\subset Y$, so $Z=Y$. The second part follows by Proposition \ref{constant-eigenfunctions}.
\end{proof}
Let $V_2\rightarrow V_1$ be a morphism in $X/G_{\textnormal{ét}}$, and let $W_i=V_i\times_{X/G}X\rightarrow X$ for $i=1,2$. Notice that the previous two lemmas allow us to construct a map:
\begin{equation*} \varphi_{V_2,V_1}:S(W_2,G)\rightarrow S(W_1,G).
\end{equation*}
Indeed, let $(T,g)\in S(W_2,G)$. By Lemma \ref{Lemma 2 sheaf of cher algebras}, there is a unique $(Y,g)\in S(W_1,G)$ with an étale map $T\rightarrow Y$. We define $\varphi_{V_2,V_1}((T,g))=(Y,g)$. As the pullback map $W_2\rightarrow W_1$ is $G$-equivariant, the map $\varphi_{V_2,V_1}$ is also $G$-equivariant for the conjugation action on $S(W_2,G)$ and $S(W_1,G)$ respectively. Thus, we obtain a functor:
\begin{equation*}
    S(-,G):  X/G_{\textnormal{ét}}\rightarrow G-\operatorname{Sets},\textnormal{ } V\mapsto S(V\times_{X/G}X,G).
\end{equation*}
There is a natural faithful functor $F:G-\textnormal{Sets}\rightarrow \operatorname{Rep}^r_K(G)$ from the category of $G$-sets to the category of right  $K$-linear $G$-representations. It is given by taking a set $S$ to the free vector space with basis $S$. Additionally, we define: 
\begin{equation*}
    \Psi:=\Hom_K(-,K)^G:\Rep^r_K(G)\rightarrow \operatorname{Vect}_K,
\end{equation*}
and notice that $\Psi$ is exact on finite-dimensional representations. Notice that for an étale map $V\rightarrow X/G$ with $W=V\times_{X/G}X$ We have:
\begin{align*}    \Psi(F(S(W,G)))=\Hom_K(F(S(W,G)),K)^G=&\Hom_{G-\textnormal{Sets}}(S(W,G),K)\\=&\Hom_{\textnormal{Sets}}(S(W,G)/G,K).
\end{align*}
In particular, we have shown that:
\begin{equation*}
    \Psi(F(S(W,G)))=\operatorname{Ref}(W,G):=\{\textnormal{Reflection functions } c:S(W,G)\rightarrow K\}.
\end{equation*}
This simple observation will simplify some of the calculations below.
\begin{defi}\label{defi sheaf of reflection functions}
We define the 
 presheaf of reflection functions on $X/G_{\textnormal{ét}}$ by:
\begin{equation*}
    \mathscr{R}(-,G):X/G_{\textnormal{ét}}\rightarrow \textnormal{Vect}_K,\textnormal{ }
    (V\rightarrow X/G)\mapsto  \textnormal{Ref}(V\times_{X/G}X,G).
    \end{equation*}    
\end{defi}
\begin{obs}
To simplify notation simple, for any $c\in  \textnormal{Ref}(X,G)$ and any étale map $V\rightarrow X/G$, we will denote the image of $c$ in $\textnormal{Ref}(V\times_{X/G}X,G)$ also by $c$.
\end{obs}
\begin{prop}\label{prop reflection functions are a sheaf}
    Let $X$ be a smooth  $G$-variety. Then $ \mathscr{R}(-,G)$ is a sheaf of $K$-vector spaces on $X/G_{\textnormal{ét}}$.
\end{prop}
\begin{proof}
 Let $\{V_i\}_{i\in I}$ be a  cover of $X/G$ in the étale topology. For each $i\in I$, define  $W_i=V_i\times_{X/G}X$. It suffices to show that the sequence:
\begin{equation}\label{equation sheaf of reflection functions}
    \bigcup_{i,j\in I}S(W_i\times_X W_j,G)\rightrightarrows \bigcup_{i\in I} S(W_i,G)\rightarrow S(X,G),
\end{equation}
is a coequalizer diagram in the category of $G$-Sets. Let $V=\cup_i V_i$ be the disjoint union of all the $V_i$. Then We have an étale covering $V\rightarrow X/G$. Its pullback along the projection $X\rightarrow X/G$ yields an étale covering 
$f:W\rightarrow X$. Let $(Y,g)\in S(X,G)$ be a reflection hypersurface of the action of $G$ on $X$.\\
By Lemmas \ref{Lemma 1 sheaf of cher algebras } and \ref{Lemma 2 sheaf of cher algebras}, there is a unique $(T,g)\in S(W,G)$
such that $f(T)\subset Y$. By construction, we have $W=\cup_i W_i$, so there is some $i\in I$ such that $W_i\cap T=T_i$ is not empty. Replace $T_i$ by any of its connected components. Then $T_i$ is a connected hypersurface of $W_i$ such that $T_i\subset W_i^g$. In particular, $(T_i,g)\in S(W_i,G)$. By construction, we have $f(T_i)\subset Y$, so We have $\varphi_{V_i,X/G}((T_i,g)))=(Y,g)$. Hence, the last arrow in diagram (\ref{equation sheaf of reflection functions}) is surjective.\\
Choose $(T_1,g)\in S(W_i,G)$ and $(T_2,g)\in  S(W_j,G)$ such that:
\begin{equation*}
    \varphi_{V_i,X}((T_1,g))=\varphi_{V_j,X}((T_2,g))=(Y,g).
\end{equation*}
Then we have étale maps $T_1\rightarrow Y$ and $T_2\rightarrow Y$. Thus, we have a rigid space:
\begin{equation*}
    Z=T_1\times_Y T_2\subset (W_1\times_X W_2 )^g, 
\end{equation*}
together with étale maps $Z\rightarrow T_1$ and $Z\rightarrow T_2$. By the same arguments as above, replacing $Z$ by one of its connected components, we have that $(Z,g)$ is a reflection hypersurface. Furthermore, by construction We have the following identities:
\begin{equation*}
    \varphi_{V_{ij},V_i}((Z,g))=   (T_1,g)      \textnormal{, } \varphi_{V_{ij},V_j}((Z,g))=(T_2,g).
\end{equation*}
Therefore, (\ref{equation sheaf of reflection functions}) is a coequalizer diagram, as we wanted to show.
\end{proof}
Next, we want to understand the behavior of the residue map with respect to étale base-change. Let us start with some technical lemmas:
\begin{Lemma}\label{Lemma residue map in étale topology}
Let $X$ be a smooth $K$-variety and $f:Y\rightarrow X$ be an étale affinoid map. Consider a smooth hypersurface $Z\subset X$ such that $f^{-1}(Z):=Z\times_X Y$
    is non-empty. Then We have a commutative diagram of $\OX_X$-modules:
    \begin{equation*}
        \begin{tikzcd}
0 \arrow[r] & \OX_X \arrow[r] \arrow[d] & \OX_X(Z) \arrow[r] \arrow[d]  & \mathcal{N}_Z \arrow[r] \arrow[d]    & 0 \\
0 \arrow[r] & f_*\OX_Y \arrow[r]        & f_*\OX_Y(f^{-1}(Z)) \arrow[r] & f_*\mathcal{N}_{f^{-1}(Z)} \arrow[r] & 0
\end{tikzcd}
    \end{equation*}
Furthermore, this map induces a commutative diagram of $\OX_X$ modules:
\begin{equation*}
    \begin{tikzcd}
\mathcal{T}_{X/K} \arrow[d, ] \arrow[r, "\xi_{Z}"]  & \OX_{X}(Z)/\OX_{X} \arrow[d]  \\
f_*\mathcal{T}_{Y/K} \arrow[r, "\xi_{f^{-1}(Z)}"] & f_*\OX_{Y}(f^{-1}(Z))/\OX_{Y}
\end{tikzcd}
\end{equation*}
\end{Lemma}
\begin{proof}
We may assume that both $X=\Sp(A)$ and $Y=\Sp(B)$ are  affinoid and that $Z=\mathbb{V}(g)$ for some rigid function $g\in\OX_X(X)$. Let $\varphi:A\rightarrow B$ be the morphism induced by $f$. As $f^{-1}(Z)$ is non-empty, we have 
$\Gamma(X,\OX_X(Z))=A\frac{1}{g}$, and $\Gamma(Y,\OX_Y(f^{-1}(Z)))=B\frac{1}{\varphi(g)}$. Thus, as $X$ is affinoid and $\OX_X(Z)$ is coherent, we get an induced map $\OX_X(Z)\rightarrow f_*\OX_Y(f^{-1}(Z))$ as wanted. In order to show commutativity of the second diagram, we will also use $\varphi$ to denote the extension of the map $A\rightarrow B$ to $\mathcal{T}_{X/K}(X)\rightarrow \mathcal{T}_{Y/K}(Y)$.
    Let $v\in \mathcal{T}_{X/K}(X)$, then we have:
    \begin{equation*}
        \varphi(\xi_{Z}(v))=\varphi\left(\left[\frac{v(g)}{g}\right]\right)=\left[\frac{\varphi(v(g))}{\varphi(g)}\right]=\left[\frac{\varphi(v)(\varphi(g))}{\varphi(g)}\right]=\xi_{f^{-1}(Z)}(\varphi(v)).
    \end{equation*}
\end{proof}
\begin{Lemma}\label{Lemma residue map in étale topology 2}
    Let $X$ be a smooth $K$-variety with a smooth hypersurface $Z\subset X$. Assume that $Z=\cup_{i=1}^r Z_i$ is the decomposition of $Z$ into its connected components, and that each $Z_i$ is also a hypersurface. The following hold:
    \begin{enumerate}[label=(\roman*)]
        \item We have a canonical decomposition $\OX_{Z}(X)/\OX_{X}=\bigoplus_{i=1}^r\OX_{Z_i}(X)/\OX_{X}$.
        \item For each $v\in \mathcal{T}_{X/K}(X)$ we have: $\xi_{Z}(v)=\sum_{i=1}^r\xi_{Z_i}(v)$. 
    \end{enumerate} 
\end{Lemma}
\begin{proof} 
As the normal bundle $\mathcal{N}_{Z}$ is a coherent sheaf on $Z$, there is a canonical decomposition $\mathcal{N}_{Z}=\bigoplus_{i=1}^r\mathcal{N}_{Z_i}$. Thus, the  first assertion follows from the equation:
  \begin{equation*}
      \OX_{Z}(X)/\OX_{X}=\mathcal{N}_{Z}=\bigoplus_{i=1}^r\mathcal{N}_{Z_i}=\bigoplus_{i=1}^r\OX_{Z_i}(X)/\OX_{X},
  \end{equation*}
  where the first identity  is Lemma \ref{residue ses}. For claim $(ii)$, we may assume that $X=\Sp(A)$, and that $Z_i=\mathbb{V}(f_i)$ for each $1\leq i\leq r$. Setting $f=\prod_{i=1}^r f_i$, it follows that $Z=\mathbb{V}(f)$, so we get the following identity:
\begin{equation*}
    \xi_{Z}(v)=\left[\frac{v(\prod_{i=1}^r f_i)}{\prod_{i=1}^r f_i}\right]=\left[\frac{\sum_{i=1}^r\prod_{j\neq i} f_jv(f_i) }{\prod_{i=1}^r f_i}\right]=
    \sum_{i=1}^r\left[\frac{v(f_i)}{f_i}\right]=\sum_{i=1}^r\xi_{Z_i}(v).
\end{equation*} 
\end{proof}
\begin{Lemma}\label{Lemma 3 restriction map cher algebras}
 Let $X$ be a smooth affinoid $K$-variety and let $f:Y=\Sp(B)\rightarrow X=\Sp(A)$ be an étale affinoid  map, with associated map of $K$-algebras $\varphi:A\rightarrow B$. Consider a smooth hypersurface $Z\subset X$, such that $f^{-1}(Z)=Z\times_X Y$ is non-empty and has a decomposition $f^{-1}(Z)=\cup_{i=1}^r T_i$ as a disjoint union of hypersurfaces. Choose a vector field $v\in \mathcal{T}_{X/K}(X)$, and let $\overline{\xi_{Z}(v)}$ be a representative in  $\OX_Z(X)(X)$ of the residue map. Then we have:
 \begin{equation*}
     \varphi\left(\overline{\xi_{Z}(v)}\right)=\sum_{i=1}^r\overline{\xi_{T_i}(\varphi(v))} + h,
 \end{equation*}
 where each $\overline{\xi_{T_i}(\varphi(v))}$ is a representative of $\xi_{T_i}(\varphi(v))$ in $\OX_f^{-1}(Z)(Y)$, and $h\in B$.
\end{Lemma}
\begin{proof}
Denote the map $\OX_{X}(Z)/\OX_{X}\rightarrow  f_*\OX_{Y}(f^{-1}(Z))/\OX_{Y}$ constructed in Lemma \ref{Lemma residue map in étale topology} by $\varphi$.
By Lemmas \ref{Lemma residue map in étale topology} and \ref{Lemma residue map in étale topology 2} we have:
    \begin{equation*}
        \varphi(\xi_{Z}(v))= \xi_{f^{-1}(Z)}(\varphi(v))=\sum_{i=1}^r\xi_{T_i}(\varphi(v)).
    \end{equation*}
Therefore, $\varphi\left(\overline{\xi_{Z}(v)}\right)$ is a representative in    $O_f^{-1}(Z)(Y)$ of $\sum_{i=1}^r\xi_{T_i}(\varphi(v))$. As $Y$ is affinoid, we have the following short exact sequence:
\begin{equation*}
    0 \rightarrow  \OX_Y(Y) \rightarrow \Gamma(Y,\OX_Y(f^{-1}(Z))) \rightarrow  \Gamma(Y,\mathcal{N}_{f^{-1}(Z)}) \rightarrow 0.
\end{equation*}
In particular, for $1\leq i\leq r$ there is an element $\overline{\xi_{T_i}(\varphi(v))}$ representing $\xi_{T_i}(\varphi(v))$. But then $h=\varphi\left(\overline{\xi_{Z}(v)}\right)-\sum_{i=1}^r\overline{\xi_{T_i}(\varphi(v))}$ is an element in $\OX_Y(Y)=B$.
\end{proof}
We are finally ready to define presheaves of Cherednik algebras on  $X/G_{\textnormal{ét}}$:
\begin{prop}\label{prop restriction maps presheaf of cherednik algebras}
    Let $X=\Sp(A)$ be a smooth $G$-variety and $V_2\rightarrow V_1$ be a map in $X/G_{\textnormal{ét,aff}}$.  Let $W_i=V_i\times_{X/G}X$ for $i=1,2$, and $f:W_2\rightarrow W_1$ be the induced map. Choose affinoid subdomains $U_i\subset \overline{W}_i$ such that 
    each $U_i$ meets all connected components of $W_i$
   and $f(U_2)\subset U_1$. There is a unique map of filtered $K$-algebras:
    \begin{equation*}
        H_{t,c,\omega}(W_1,G)\rightarrow H_{t,c,\omega}(W_2,G),
    \end{equation*}
    which is a restriction of the canonical morphism  $G\ltimes\mathcal{D}_{\omega/t}(U_1)\rightarrow G\ltimes \mathcal{D}_{\omega/t}(U_2)$.
\end{prop}
\begin{proof}
 As $U_2\rightarrow U_1$ is étale and $G$-equivariant, we have a map of filtered $K$-algebras:
    \begin{equation*}
       \varphi: G\ltimes \mathcal{D}_{\omega/t}(U_1)\rightarrow G\ltimes \mathcal{D}_{\omega/t}(U_2).
    \end{equation*}
    This map sends $G\ltimes \OX_X(U_1)$ to $G\ltimes \OX_X(U_2)$ by construction. Hence, we only need to show that the image of any Dunkl-Opdam operator under $\varphi$  is contained in $F_1 H_{t,c,\omega}(W_2,G)$. Consider a derivation $v\in \mathcal{T}_{X/K}(W_1)$, and let $D_v$ be a Dunkl-Opdam operator associated to $v$.
    We have the following:
     \begin{equation*}
       \varphi( D_{v})=t\mathbb{L}_{\varphi(v)}+\sum_{(Y,g)\in S(W_1,G)}\frac{2c(Y,g)}{1-\lambda_{Y,g}}\varphi(\overline{\xi_{Y}(v)})(g-1).
    \end{equation*}
    By Lemma \ref{Lemma 1 sheaf of cher algebras }  we know that for $(Y,g)\in S(W_1,G)$, the preimage $f^{-1}(Y)$ is either empty or a disjoint union of reflection hypersurfaces. If $f^{-1}(Y)$ is empty, then $\OX_{W_2}(f^{-1}(Y))=\OX_{W_2}$. In particular, $\frac{2c(Y,g)}{1-\lambda_{Y,g}}\varphi(\overline{\xi_{Y}(v)})(g-1)\in G\ltimes \OX_{W_2}(W_2)$. 
    Consider the following differential operator:
    \begin{equation*}
        \Delta(v):=\sum_{(Y,g)\in S(W_1,G), f^{-1}(Y)= \varnothing }\frac{2c(Y,g)}{1-\lambda_{Y,g}}\varphi(\overline{\xi_{Y}(v)})(g-1).
    \end{equation*}
    The discussion above shows $\Delta(v)\in G\ltimes \OX_X(W_2)$, and we have the following identity:
    \begin{equation*}
       \varphi( D_{v})=t\mathbb{L}_{\varphi(v)}+\sum_{(Y,g)\in S(W_1,G), f^{-1}(Y)\neq \varnothing}\frac{2c(Y,g)}{1-\lambda_{Y,g}}\varphi(\overline{\xi_{Y}(v)})(g-1) +\Delta(v).
    \end{equation*}
    By Lemma \ref{Lemma 2 sheaf of cher algebras}, each $(Z,g)\in S(W_2,G)$ is a connected component of the preimage of a unique $(Y,g)\in S(W_1,G)$. Let $(Y,g)\in S(W_1,G)$, and let $f^{-1}(Y)=\cup_{i=1}^r Z_i$ be its decomposition as a union of reflection hypersurfaces. By Lemma \ref{Lemma 3 restriction map cher algebras} we have:
    \begin{equation*}
        \varphi(\overline{\xi_{Y}(v)})=\sum_{i=1}^r\overline{\xi_{Z_i}(\varphi(v))} + h_Y.
    \end{equation*}
As $\mathscr{R}(-,G)$ is a sheaf, we have  $c(Y,g)=c(Z_i,g)$  for all $1\leq i\leq r$. Hence, we have:
    \begin{equation}\label{equation last equation in prop restriction map cherednik alg}
       \varphi( D_{v})=D_{\varphi(v)}+\sum_{(Y,g)\in S(W_1,G)}h_Y+\Delta(v).
    \end{equation} 
    In particular, $\varphi( D_{v})=\in F_1 H_{t,c,\omega}(V_2,G)$, as we wanted to show.
\end{proof}
\begin{defi}\label{defi presheaf of CHer algebras}
    Let $X=\Sp(A)$ be a smooth $G$-variety, and choose $t\in K^*$, $c\in \operatorname{Ref}(X,G)$, and $\omega\in\Omega_{X/K}^{2,\operatorname{cl}}(X)^G$. The presheaf of Cherednik algebras $ H_{t,c,\omega, X,G}$ is the unique presheaf on $X/G_{\textnormal{ét,aff}}$ defined on affinoid étale maps $V\rightarrow X/G$ by:
    \begin{equation*}
        H_{t,c,\omega, X,G}(V)=H_{t,c,\omega}(V\times_{X/G} X,G),
    \end{equation*}  
    and restriction maps are defined as in Proposition \ref{prop restriction maps presheaf of cherednik algebras}.
\end{defi}
\subsection{The PBW Theorem for Cherednik algebras}\label{section PBW Theorem for CHer algebras in the rigid setting}
The main goal of this section is showing that the PBW Theorem holds for Cherednik algebras on small enough smooth affinoid spaces. From now and until the end of the section, let $X=\Sp(A)$ be a $G$-variety with an étale map $X\rightarrow \mathbb{A}^n_K$. We fix a $G$-stable open affinoid subspace $U\subset \overline{X}$ meeting all connected components of $X$, a unit $t\in K^*$, a $G$-invariant closed $2$-form $\omega\in\Omega_{X/K}^{2,\operatorname{cl}}(X)^G$, and a reflection function $c\in \textnormal{Ref}(X,G)$.
\begin{teo}[{ PBW Theorem for affinoid Cherednik algebras}]\label{Affinoid PBW Theorem}\label{Classic PBW Theorem}
The Dunkl-Opdam filtration induces a canonical isomorphism of $G\ltimes \OX_X(X)$ algebras:
\begin{equation*}
  \Psi:  \gr_{F}H_{t,c,\omega}(X,G)\rightarrow G\ltimes \operatorname{Sym}\mathcal{T}_{X/K}(X).
\end{equation*}
\end{teo}
\begin{proof}
The proof is analogous to \cite[Theorem 2.17]{etingof2004cherednik}.
\end{proof}
\begin{obs}
The morphism $\Psi$ does not depend on the choice of $U\subset \overline{X}$.
\end{obs}
This theorem allows us to deduce important structural results for $H_{t,c,\omega}(X,G)$.
\begin{coro}\label{uncompleted CHerednik algebras are free modules}
Let $v_1,\cdots,v_n$ be a basis of $\mathcal{T}_{X/K}(X)$ as an $\OX_X(X)$-module
and $D_{v_1},\cdots,D_{v_n}$ be a choice of Dunkl-Opdam operators associated to it. The Cherednik algebra 
   $H_{t,c,\omega}(X,G)$ is a free $G\ltimes \OX_X(X)$-module, with a basis given by:
   \begin{equation*}
       D_{v_{1}}^{\alpha_{1}}\cdots D_{v_{n}}^{\alpha_{n}},
   \end{equation*}
   where the $\alpha_{1},\cdots,\alpha_{n}$ are non-negative integers . 
\end{coro}
\begin{coro}\label{clasical local version cherednik algebras are sheaves}
Let $X=\Sp(A)$ be a $G$-variety with an étale map $X\rightarrow \mathbb{A}^n_K$. Every
presheaf of Cherednik algebras $H_{t,c,\omega,X,G}$ satisfies:
\begin{enumerate}[label=(\roman*)]
    \item $H_{t,c,\omega,X,G}$ is a locally free sheaf of $\pi_{*}\OX_{X}$-modules on $X/G_{\textnormal{ét,aff}}$.
    \item For any $V\in X/G_{\textnormal{ét,aff}}$ there is an isomorphism of filtered algebras:
    \begin{equation*}
        H_{t,c,\omega}(V\times_{X/G}X,G)=\OX_X(V\times_{X/G}X)\otimes_{\OX_X(X)}H_{t,c,\omega}(X,G).
    \end{equation*}
    \item $H_{t,c,\omega,X,G}$ has trivial higher \v{C}ech cohomology groups on $X/G_{\textnormal{ét,aff}}$.
\end{enumerate}
\end{coro}
\begin{proof}
As $X$ is affinoid, statement $(ii)$ implies statement $(iii)$.  By Corollary \ref{uncompleted CHerednik algebras are free modules}, the algebra $H_{t,c,\omega}(X,G)$ is a free module over $G\ltimes \OX_X(X)$. Hence, statement $(ii)$ also implies $(i)$. Thus, we only need to show that for every affinoid étale map $V\rightarrow X/G$ and $W=V\times_{X/G}X$ there is an identification of filtered algebras:
\begin{equation*}
    H_{t,c,\omega,X,G}(V):=H_{t,c,\omega}(W,G)=\OX_W(W)\otimes_{\OX_X(X)}H_{t,c,\omega}(X,G).
\end{equation*}
By Proposition \ref{prop restriction maps presheaf of cherednik algebras}, we have a morphism of $G\ltimes \OX_W(W)$-modules:
\begin{equation*}
    \psi:\OX_W(W)\otimes_{\OX_X(X)}H_{t,c,\omega}(X,G)\rightarrow H_{t,c,\omega}(W,G).
\end{equation*}
There is a canonical filtration on  $\OX_W(W)\otimes_{\OX_{X}(X)}H_{t,c,\omega}(X,G)$ given by:
\begin{equation*}
    F_{i}\left(\OX_W(W)\otimes_{\OX_{X}(X)}H_{t,c,\omega}(X,G)\right)=\OX_W(W)\otimes_{\OX_X(X)}F_{i}H_{t,c,\omega}(X,G).
\end{equation*}
This filtration is exhaustive and positive, and makes $\psi$ a morphism of filtered $\OX_W(W)$-modules. By Theorem \ref{Classic PBW Theorem}, passing to the associated graded yields:
\begin{align*}
    \gr_{F}\left(\OX_W(W)\otimes_{\OX_X(X)}H_{t,c,\omega}(X,G)\right)=&\OX_W(W)\otimes_{\OX_X(X)}\gr_{F}H_{t,c,\omega}(X,G)\\
    =&\OX_W(W)\otimes_{\OX_X(X)}\operatorname{Sym}(\mathcal{T}_{X/K}(X))\\=&\operatorname{Sym}(\mathcal{T}_{W/K}(W))\\=&\gr_{F}H_{t,c,\omega}(W,G).
\end{align*}
\end{proof}
As the categories of sheaves on $X/G_{\textnormal{ét,aff}}$ and 
$X/G_{\textnormal{ét}}$ are equivalent, we may extend $H_{t,c,\omega,X,G}$ to a sheaf of filtered algebras on $X/G_{\textnormal{ét}}$. This can be used to define Cherednik algebras for non-affinoid smooth $G$-varieties:
\begin{teo}\label{globalization of Cherednik algebras}
Let $X$ be a smooth $G$-variety. Fix $t\in K^{*}$, $c\in \operatorname{Ref}(X,G)$, $[\omega]\in \mathbb{H}^{1}(X,\Omega_{X/K}^{\geq 1})^{G}$. There is a unique sheaf of filtered $K$-algebras $H_{t,c,\omega,X,G}$ on $X/G_{\textnormal{ét}}$ such that for any étale map $V=\Sp(A)\rightarrow X/G$,  satisfying that $W=V\times_{X/G}X$ admits an étale map $W\rightarrow \mathbb{A}^n_K$ we have:
\begin{equation*}
    H_{t,c,\omega,X,G}(V)=H_{t,c,\omega}(W,G).
\end{equation*}
\end{teo}
\section{p-adic Cherednik algebras}\label{p-adic CHerednik algebras}
In this chapter we will construct sheaves of $p$-adic Cherednik algebras on a smooth $G$-variety $X$. Locally, these arise as Fréchet-Stein completions of the Cherednik algebras defined in the previous chapter. In order to solve some of the technical issues that appear in this construction, we give a (very short) introduction to adic spaces, focusing on the specialization map and its relation with the Shilov boundary of a rigid analytic variety. We then use these results to give a local definition of $p$-adic Cherednik algebras, and show that they are indeed  Fréchet-Stein algebras. Finally, we construct sheaves of $p$-adic Cherednik algebras on $X/G_{\textnormal{ét}}$, and show that their higher \v{C}ech cohomology groups vanish. We finish the chapter by studying the $c$-flatness of the transition maps of the sheaf of $p$-adic Cherednik algebras, and giving  some applications of this analysis to the theory of co-admissible modules over a sheaf of $p$-adic Cherednik algebras.
\subsection{Restriction algebras}\label{section restriction rings}
In this section, we introduce some technical machinery that will be indispensable for the theory of $p$-adic Cherednik algebras. In particular, we will define the restriction algebras, the Shilov boundary of an adic  affinoid space, and show the connection between these two concepts. Along the way, we give a small recap on the specialization map and its relation with the Shilov boundary. We start by making the following definition:
\begin{defi}\label{defi restriction rings}
   Let $f:A\rightarrow B$ be a morphism of affinoid $K$-algebras. The restriction algebra of $f$ is the following $\mathcal{R}$-algebra:
\begin{equation*}
    A^{\circ}_{f:A\rightarrow B}=f^{-1}(B^{\circ})=\{ a\in A \textnormal{ } \vert \textnormal{ }f(a)\in B^{\circ}\}.
\end{equation*}
If $f:A\rightarrow B$ induces an open immersion $\Sp(B)\rightarrow \Sp(A)$, we  write $A^{\circ}_B:=A^{\circ}_{f:A\rightarrow B}$.
\end{defi}
Restriction algebras associated to open immersions will be instrumental in the theory of $p$-adic Cherednik algebras. This  section is devoted to solving some technical issues regarding these algebras. In particular,
for an open immersion $\Sp(B)\rightarrow \Sp(A)$,
we will give conditions which assure that $A^{\circ}_B=A^{\circ}$. To do this, we will need to make use of the theory of adic spaces. As the material we need is rather standard, we will not introduce these spaces here. The reader can consult \cite{huber2013etale}, and 
\cite{Wedhorn2019adic}, for a thorough introduction to Huber's theory of adic spaces.
 \begin{defi}
 Let $X=\Spa(A)$ be an affinoid space. The Shilov boundary of $X$ (or of $A$) is the unique minimal subset of rank one points $S(X)\subset X$
 such that for every function $f\in A$ we have:
 \begin{equation*}
     \underset{x\in S(X)}{\textnormal{max}}\{v_{x}(f)\} =\underset{x\in X, \textnormal{rank}(v_{x})=1}{\textnormal{max}}\{v_{x}(f)\}. 
 \end{equation*}
 The points $x\in S(X)$ are called Shilov points.
 \end{defi}
 Notice that this notion makes sense, as every valuation of rank one has image contained in the totally ordered multiplicative abelian monoid $(\mathbb{R}_{\geq 0},\cdot)$ (\emph{cf.} \cite[Proposition 1.14.]{Wedhorn2019adic}). The Shilov boundary can be defined for arbitrary Banach $K$-algebras. However, it is in general not well-behaved, and its not even clear it exists (due to the minimality and uniqueness requirements). However, we have the following result due to V. Berkovich:
 \begin{prop}[{\cite[Corollary 2.4.5]{berkovich2012spectral}}]
Let $A$ be an affinoid $K$-algebra, and let $X=\Spa(A,A^{\circ})$. Then the Shilov boundary $S(X)$ exists and it is finite.
 \end{prop}
In order to make use of this notion, we will first recall the basic properties of the specialization map. For a complete treatment of the specialization map one can check Section 2.4 in \cite{berkovich2012spectral} or \cite{le2007rigid}. Let $A^{\circ}\subset A$ be the algebra of power-bounded elements of $A$, and $A^{\circ\circ}\subset A^{\circ}$ be the ideal of topologically nilpotent elements. We define  $\mathfrak{X}=\Spec(A^{\circ}/A^{\circ\circ})$. Notice that $\mathfrak{X}$ is a finite type $k$-variety which agrees with the reduction of $\Spec(A^{\circ}/\pi)$. The specialization map:
 \begin{equation*}
     \textnormal{sp}:X\rightarrow \mathfrak{X},
 \end{equation*}
 is defined as follows: Every point $x\in X$ is given by a continuous valuation $v_{x}$ in $A$, and has associated residue field $(K(x),K(x)^+)$. We remark that if the valuation $v_{x}$ has rank one, then it can be shown that $K_{x}^{+}$ agrees with the ring of integers of $K_{x}$. 
 \begin{defi}
 Let $x\in X$ be an adic point, and $\mathfrak{m}_{x}^{+}$ be the maximal ideal of $K(x)^+$. The secondary residue field of $X$ at $x$ is the field: $\Tilde{K}_{x}=K_{x}^{+}/\mathfrak{m}_{x}^{+}$.
 \end{defi}
For every $x\in X$ there is a morphism
 $(A,A^{\circ})\rightarrow (K_{x},K^{+}_{x})$, which induces a map:
 \begin{equation*}
     A^{\circ}\rightarrow K^{+}_{x}\rightarrow \Tilde{K}_{x}.
 \end{equation*}
 By continuity,
 topologically nilpotent elements are mapped to topologically nilpotent elements. Therefore, we get a map:
 \begin{equation*}
     A^{\circ}/A^{\circ\circ}\rightarrow \Tilde{K}_{x}.
 \end{equation*}
 Equivalently, we have a map of schemes $\Spec(\Tilde{K}_{x})\rightarrow\mathfrak{X}$. We define $\textnormal{sp}(x)\in \mathfrak{X}$ to be the image of this map. We notice that this construction can be performed in more general adic spaces, and that the specialization map is continuous, spectral and closed. The idea now is to use the specialization map to obtain characterization of the Shilov boundary of $X$. In particular, we have:
\begin{prop}[{\cite[Proposition 2.2]{bhatt2022six}}]\label{characterization of Shilov points}
Fix an affinoid $K$-algebra $A$ and a rank one point $x \in X= \Spa(A,A^{\circ})$. The following are equivalent:    \begin{enumerate}[label=(\roman*)]
\item $\textnormal{sp}(x)$ is the generic point of an irreducible component of $\mathfrak{X}$.
\item $\{x\}=\textnormal{sp}^{-1}(\textnormal{sp}(x))$.
\item  Let $\nu=\textnormal{sp}(x)\in\mathfrak{X}\subset \Spec(A^{\circ})$. There is an identity of topological rings:
\begin{equation*}
    K^+_{x}= \widehat{A^{\circ}_{\nu}}=\varprojlim_n\widehat{A^{\circ}_{\nu}}/\pi^n,
\end{equation*}
where $A^{\circ}_{\nu}$ is localization at the prime ideal $\nu$.
\item The valuation $v_{x}$ is in the Shilov boundary $S(X)$.
\end{enumerate}
\end{prop}
The specific construction of the residue field $K_{x}$ allows us to give a very straightforward characterization of the functions that are in the support of a Shilov point. 
\begin{Lemma}\label{center of shilov points}
    Let $X=\Spa(A,A^{\circ})$ be affinoid, $x\in S(X)$, and $f\in A$. Then $v_x(f)=0$ if and only if there is some open subspace $U$ satisfying $x\in U\subset \mathbb{V}(f)$.
\end{Lemma}
 \begin{proof}
 Let $x\in S(X)$ be a Shilov point and let $\nu=\operatorname{sp}(x)\in\mathfrak{X}\subset \Spec(A^{\circ})$ be the prime ideal representing its image under the specialization map. After multiplying by a power of the uniformizer, we may assume that $f\in A^{\circ}$. Assume that $v_{x}(f)=0$. By Proposition \ref{characterization of Shilov points}, the map $A^{\circ}\rightarrow K_{x}$ is given by the following composition:
 \begin{equation*}
     A^{\circ}\rightarrow A^{\circ}_{\nu}\rightarrow \widehat{A^{\circ}_{\nu}}\rightarrow \widehat{A^{\circ}_{\nu}}[\pi^{-1}]=K_{x},
 \end{equation*}
 where the first map is a localization, and the second map is a $\pi$-adic completion. As the valuation on $K$ is discrete, the ring $A^{\circ}$ is noetherian. In particular, the map $A^{\circ}_{\nu}\rightarrow \widehat{A^{\circ}_{\nu}}$ corresponds to a completion of a noetherian local ring, and therefore is injective. Furthermore, as $A^{\circ}$ is an admissible $\mathcal{R}$ algebra it does not have $\pi$-torsion. Hence, the algebra $\widehat{A^{\circ}_{\nu}}$ does not have $\pi$-torsion. Thus, the map  $\widehat{A^{\circ}_{\nu}}\rightarrow K_{x}$ is also injective. By assumption we have $v_{x}(f)=0$, so the image of $f$ under the map $A^{\circ}\rightarrow K_{x}$ is zero. It follows that $f$ is in the kernel of the localization map $A^{\circ}\rightarrow A^{\circ}_{\nu}$. We have an identity:
 \begin{equation*}
     A^{\circ}_{\nu}=\varinjlim_{\nu\in D(g)} A^{\circ}[g^{-1}].
 \end{equation*}
 As this is a filtered colimit, it follows that there is an element $g\in A^{\circ}$ such that $f$ is in the kernel of $A^{\circ}\rightarrow A^{\circ}[g^{-1}]$. As $\nu \in D(g)$, it follows that the intersection of $D(g)$ with the special fiber:
 \begin{equation*}
     \Spec(A^{\circ}/\pi)=\Spec(k)\times_{\Spec(\mathcal{R})}\Spec(A^{\circ}),
 \end{equation*}
 is not empty. Hence, it makes sense to consider the completed localization:
 \begin{equation*}
     A^{\circ}\rightarrow A^{\circ}[g^{-1}]\rightarrow A^{\circ}\langle g^{-1}\rangle,
 \end{equation*}
 and $f$ is mapped to zero under this completed localization. Taking the rigid generic fiber of this yields a map $A\rightarrow A\langle g^{-1}\rangle$, which corresponds to the inclusion of a Laurent subdomain. As $f$ is in the kernel of this map, we have $X(\frac{1}{g})\subset V(f)$. 
 \end{proof}
\begin{Lemma}
Consider affinoid spaces $X=\Spa(A,A^{\circ})$,$Y=\Spa(B,B^{\circ})$, with a morphism $f:Y\rightarrow X$, which is induced by a morphism of affinoid $K$-algebras $f:A\rightarrow B$.
If $S(X)\subset f(Y)$, then $A^{\circ}_{f:A\rightarrow B}=A^{\circ}$ as topological rings.
\end{Lemma}
\begin{proof}
By definition,
the image of any $g\in A^{\circ}_{f:A\rightarrow B}$ is power-bounded in $Y$. Thus, for every rank one point $y\in Y$ we have $v_{y}(g)\leq 1$. By assumption, we have $S(X)\subset f(Y)$. Hence, by definition of $S(X)$, we have $v_{x}(g)\leq 1$ for every rank one point $x\in X$. As this includes all classical points, this is equivalent to $g\in A^{\circ}$.
\end{proof}
 Our focus will be on the following situation:  Let $X=\Spa(A,A^{\circ})$ be an affinoid space with a nowhere-dense closed analytic subspace $Z\subset X$ given by an ideal:
\begin{equation*}
    \mathcal{I}_Z=f_{1}A +\cdots + f_{r} A.
\end{equation*}
Let $\overline{X}=X\setminus Z$ be the complement of $Z$. We have a family of rational subdomains:
\begin{equation*}
    \overline{X}_{n}=X\left(\frac{\pi^{n}}{f_{1}},\cdots,\frac{\pi^{n}}{f_{r}}\right)=\Spa(B_{n},B_n^{\circ})\subset \overline{X},
\end{equation*}
and we have $\overline{X}=\bigcup_n \overline{X}_n$. In this setting, we have the following proposition:
\begin{prop}\label{restriction rings are trivial}
 For sufficiently high $n\geq 0$ we have $S(X)\subset \overline{X}_n$
\end{prop}
 \begin{proof}
 As the Shilov boundary of $X$ is a finite set and $\overline{X}_{n}$ is an increasing family, it is enough to show that  $S(X)\subset \overline{X}$. Choose a Shilov boundary point $x\in X$, with valuation $v_x$. We have $x\in Z$ if and only if  $v_{x}(f_i)=0$ for all $1\leq i\leq r$. If this holds, then by Lemma \ref{center of shilov points} for each $i$ there is an open subspace $x\in U_i\subset \mathbb{V}(f_i)$. But then we have $x\in U=\cap_{i=1}^r U_i\subset Z$, which is a contradiction.
 \end{proof}
\begin{obs}
    Notice that any open subspace $U\subset X$ which contains the Shilov boundary of $X$ has non-empty intersection with all connected components of $X$.
\end{obs}

\subsection{\texorpdfstring {Local definition of p-adic Cherednik algebras}{}}\label{sect local definition cher alg}
We will now use the theory developed in previous sections to define $p$-adic Cherednik algebras. These algebras, denoted by $\mathcal{H}_{t,c,\omega}(X,G)$, play the same role in the setting of Cherednik algebras on rigid analytic spaces as the algebras of $p$-adic differential operators $\wideparen{\mathcal{D}}_{\omega}$ 
do in the setting of $\mathcal{D}_{\omega}$-modules. In particular, most of the good properties of algebras of $p$-adic differential operators are still present in $p$-adic Cherednik algebras. For the rest of this section, we assume that $X=\Sp(A)$ is a $G$-variety with an étale map $X\rightarrow \mathbb{A}^r_K$. As before, we fix a unit $t\in K^*$, a reflection function $x\in \textnormal{Ref}(X,G)$ and a closed and $G$-invariant $2$-form $\omega\in\Omega_{X/K}^{2,\operatorname{cl}}(X)^G$. In particular, by Lemma \ref{G invariant Lie latices}, there are $G$-invariant $(\mathcal{R},A^{\circ})$-Lie lattices for $\mathcal{A}_{\omega/t}(X)$, and for any such Lie lattice $\mathscr{A}_{\omega/t}$, we get a Fréchet-Stein presentation of $G\ltimes\wideparen{\D}_{\omega/t}(X)$. Indeed, let $\mathcal{I}_n$ be the two-sided ideal in $U(\pi^n\mathscr{A}_{\omega/t})$ 
 generated by $1_{A^{\circ}}-1_{\widehat{U}(\pi^n\mathscr{A}_{\omega/t})}$. Then we define:
 \begin{equation*}
\D(\pi^{n}\mathscr{A}_{\omega/t})=U(\pi^n\mathscr{A}_{\omega/t})/\mathcal{I}_n,
 \end{equation*}
 and we have a Fréchet-Stein presentation:
 \begin{equation*}
   G\ltimes\wideparen{\D}_{\omega/t}(X)=\varprojlim_n G\ltimes\widehat{\D}(\pi^{n}\mathscr{A}_{\omega/t})_K.  
 \end{equation*}
 The idea is obtaining a completion of the Cherednik algebra $H_{t,c,\omega}(X,G)$ using a similar process. However, $H_{t,c,\omega}(X,G)$ is not the skew group algebra of an enveloping algebra, nor a quotient of such an algebra. Thus, the techniques we have been using so far are not available to us. Nevertheless, as shown in Lemma \ref{lemma independence of cher alg of U}, if $U\subset \overline{X}$ is a $G$-invariant affinoid subdomain which meets all connected components of $X$, there is a canonical injection:
 \begin{equation*}
    H_{t,c,\omega}(X,G)\rightarrow G\ltimes \D_{\omega/t}(U).
 \end{equation*}
The idea is defining the completion of $H_{t,c,\omega}(X,G)$ as its closure in  $G\ltimes \wideparen{\D}_{\omega/t}(U)$. 
\begin{defi}[{\cite[Definition 3.1]{ardakov2019}}]
Let $\varphi:A\rightarrow B$ be an étale morphism of affinoid algebras, with affine formal models $\mathcal{A}$ and $\mathcal{B}$. Let $\mathcal{L}$ be a $(\mathcal{R},\mathcal{L})$-Lie algebra. We say $\mathcal{B}$ is $\mathcal{L}$-stable if $\varphi(\mathcal{A})\subset \mathcal{B}$, and if the $(\mathcal{R},\mathcal{A})$-Lie algebra structure on $\mathcal{L}$ extends to a $(\mathcal{R},\mathcal{B})$-Lie algebra structure on $\mathcal{B}\otimes_{\mathcal{A}}\mathcal{L}$.
\end{defi}
We remark that if the previous extension exists, it is unique. Furthermore, there is an explicit criterion of existence of such extensions. Indeed, consider the anchor map $\rho:\mathcal{L}\rightarrow \Der_{\mathcal{R}}(\mathcal{A})$.  This induces a $B$-linear map of $K$-Lie algebras:
\begin{equation*}
  \rho_B:B\otimes_A(K\otimes_{\mathcal{R}}\mathcal{L})\rightarrow B\otimes_A(K\otimes_{\mathcal{R}}\Der_{\mathcal{R}}(\mathcal{A}))\cong B\otimes_A \Der_K(A)\cong \Der_K(B),
\end{equation*}
 By \cite[Corollary 2.4]{ardakov2019}, the $(\mathcal{R},\mathcal{A})$-Lie algebra structure on $\mathcal{L}$ extends to a $(\mathcal{R},\mathcal{B})$-Lie algebra structure on $\mathcal{B}\otimes_{\mathcal{A}}\mathcal{L}$ if and only if  the following holds:
\begin{equation}\label{equation conditions on extensions}
    \rho_B(\mathcal{B}\otimes_{\mathcal{A}}\mathcal{L})\subset \Der_{\mathcal{R}}(\mathcal{B}).
\end{equation}
Consider the following definition:
\begin{defi}\label{defi U-admissible Cherednik lattice}
Let $U=\Sp(B)\subset \overline{X}$ be a $G$-stable affinoid subdomain meeting all connected components of $X$, and let $\mathscr{A}_{\omega/t}$ be a finite-free $G$-invariant $(\mathcal{R},A^{\circ})$-Lie lattice of $\mathcal{A}_{\omega/t}(X)$. Let $\mathscr{T}_{\mathscr{A}_{\omega/t}}$ be the image of $\mathscr{A}_{\omega/t}$ under $\mathcal{A}_{\omega/t}(X)\rightarrow \mathcal{T}_{X/K}(X)$. We say  $\mathscr{A}_{\omega/t}$ is a $U$-Cherednik lattice if it satisfies the following:
   \begin{enumerate}[label=(\roman*)]
       \item $B^{\circ}$ is $\mathscr{A}_{\omega/t}$-stable, and we let $\mathscr{B}_{\omega/t}=B^{\circ}\otimes_{A^{\circ}}\mathscr{A}_{\omega/t}$.
       \item  For each $v\in \mathscr{T}_{\mathscr{A}_{\omega/t}}$ there is a Dunkl-Opdam operator $D_{v/t}\in G\ltimes\D(\mathscr{B}_{\omega/t})$.
   \end{enumerate}
\end{defi}
Let us make a few comments on the definition:
 \begin{obs}\label{remark properties of CHerednik lattices}
 Let $U\subset \overline{X}$ be a $G$-invariant affinoid subdomain meeting all connected components, and let $\mathscr{A}_{\omega/t}$ be a $U$-Cherednik lattice. The following hold:
 \begin{enumerate}[label=(\roman*)]
     \item $\mathscr{T}_{\mathscr{A}_{\omega/t}}$ is a finite-free $(\mathcal{R},A^{\circ})$-Lie lattice of $\mathcal{T}_{X/K}(X)$.
     \item $B^{\circ}$ is $\pi^{n}\mathscr{T}_{\mathscr{A}_{\omega/t}}$-stable for all $n\geq 0$.  
     \item Let $\mathscr{T}_{\mathscr{B}_{\omega/t}}=B^{\circ}\otimes_{A^{\circ}}\mathscr{T}_{\mathscr{A}_{\omega/t}}$. Then $\omega/t(\mathscr{T}_{\mathscr{B}_{\omega/t}},\mathscr{T}_{\mathscr{B}_{\omega/t}})\subset B^{\circ}$.
     \item $\pi^n\mathscr{A}_{\omega/t}$ is a $U$-Cherednik lattice for all $n\geq 0$.
 \end{enumerate}    
 \end{obs}
For the future developments, it will be convenient to have a good understanding of Cherednik lattices. Let us start by recalling a rather important technical lemma:  
\begin{Lemma}[{\cite[Lemma 7.6]{ardakov2019}}]\label{Lemma 7.6 ardakov}
Let $X=\Sp(A)$ be an affinoid space with a $(K,A)$-Lie algebra $L$ admitting a smooth $(\mathcal{R},A^{\circ})$-Lie lattice $\mathcal{L}$. Every affinoid subdomain $U=\Sp(B)\subset X$ satisfies that $B^{\circ}$ is $\pi^n\mathcal{L}$-stable for sufficiently high $n\geq 0$.
\end{Lemma}
This result is a consequence of the condition for the existence of extensions given in equation (\ref{equation conditions on extensions}). We now show a version of this result for Cherednik lattices:
\begin{prop}\label{prop Cherednik lattices}
    For $1\leq i\leq m$, let $U_i\subset \overline{X}$  be $G$-invariant affinoid subdomains meeting all connected components of $X$. There is a  finite-free $G$-invariant $(\mathcal{R},A^{\circ})$-Lie lattice $\mathscr{A}_{\omega/t}\subset \mathcal{A}_{\omega/t}(X)$, which is a $U_i$-Cherednik lattice for all $i$.
\end{prop}
\begin{proof}
 Let $\mathscr{A}_{\omega/t}$ be a finite-free $G$-invariant $(\mathcal{R},A^{\circ})$-Lie lattice of $\mathcal{A}_{\omega/t}(X)$, which exists by Lemma \ref{G invariant Lie latices}, and let $U_i=\Sp(B_i)$ for all $1\leq i\leq m$. Notice that for each $n\geq 0$, the $A^{\circ}$-module $\pi^n\mathscr{A}_{\omega/t}$ is also a $G$-invariant $(\mathcal{R},A^{\circ})$-Lie lattice of $\mathcal{A}_{\omega/t}(X)$. As the family $\{U_i\}_{i=1}^m$  is finite, we can use Lemma \ref{Lemma 7.6 ardakov} to assume that $B^{\circ}_i$ is $\mathscr{A}_{\omega/t}$-stable for all $i$.  Furthermore, for each non-negative integer $l\geq 0$, the lattice $\pi^l\mathscr{A}_{\omega/t}$ also satisfies condition $(i)$
 for all $U_i$. Hence, it suffices to show that for sufficiently high $l\geq 0$, the lattice $\pi^l\mathscr{A}_{\omega/t}$ also satisfies condition $(ii)$ for all $U_i$. It is enough to assume $m=1$, so we drop the subscript $i$ from the notation. Let $\mathscr{B}_{\omega/t}=B^{\circ}\otimes_{A^{\circ}}\mathscr{A}_{\omega/t}$. Notice that $G\ltimes \D(\pi^{n}\mathscr{B}_{\omega/t})$ is an $A^{\circ}$-module for each $n$. Furthermore, by Proposition \ref{initial relations in a Cherednik algebra}, for each $v\in \mathscr{T}_{\mathscr{A}_{\omega/t}}$ and each $f\in A^{\circ}$, we have:
\begin{equation*}
    D_{f v}=f D_v.
\end{equation*}
 Hence, it suffices to show that for any $A^{\circ}$-basis $\{v_1,\cdots v_r\}$ of $\mathscr{T}_{\mathscr{A}_{\omega/t}}$, there is a choice of Dunkl-Opdam operators $D_{v_j}$, and a non-negative integer $l$ such that:
\begin{equation*}
    D_{\pi^{l}v_j/t}\in G\ltimes\mathcal{D}(\pi^{l}\mathscr{B}_{\omega/t}), \textnormal{ and all }1\leq j\leq r.
\end{equation*}
Let $D_{v_j}$ be a Dunkl-Opdam operator for $v_j$, and consider the following equation:
    \begin{equation*}
        D_{\pi^{l}v_j/t}=\pi^lD_{ v_j/t}=\pi^l\mathbb{L}_{v_j}+ \pi^l\left(\sum_{(Y,g)\in S(X,G)}\frac{2c(Y,g)}{1-\lambda_{Y,g}}\overline{\xi_{Y}}(v_j/t)(g-1)\right).
    \end{equation*}
Notice that  $\pi^l \mathbb{L}_{v_j} \in \mathcal{D}(\pi^{l}\mathscr{B}_{\omega})$ for $l\geq 0$. We need to show that the other terms are  contained in  $G\ltimes\mathcal{D}(\pi^{l}\mathscr{B}_{\omega})$ for high $l\geq 0$. As $U$ is $G$-invariant and $U\subset \overline{X}$, we have:
\begin{equation*}
    \sum_{(Y,g)\in S(X,G)}\frac{2c(Y,g)}{1-\lambda_{Y,g}}\overline{\xi_{Y}}(v_j/t)(g-1) \in G\ltimes B.
\end{equation*}
As $G\ltimes B=G\ltimes B^{\circ}[\pi^{-1}]$, there is a non-negative integer $l\geq 0$ such that:
\begin{equation*}
  \pi^l\left(\sum_{(Y,g)\in S(X,G)}\frac{2c(Y,g)}{1-\lambda_{Y,g}}\overline{\xi_{Y}}(v_j/t)(g-1)\right)\in G\ltimes B^{\circ}, \textnormal{ for each $1\leq j\leq r$.} 
\end{equation*}
 Hence, $\pi^l \mathscr{A}_{\omega/t}$ is a $U$-Cherednik lattice, as we wanted to show.
\end{proof}
Fix some $G$-invariant affinoid subdomain $U=\Sp(B)\subset \overline{X}$ be an  which meets all connected components of $X$. In this situation, we can make the following definitions:
\begin{defi}\label{defi rigid cher algebras}
Let $\mathscr{A}_{\omega/t}$ be a $U$-Cherednik lattice, and let $\mathscr{B}_{\omega/t}=B^{\circ}\otimes_{A^{\circ}}\mathscr{A}_{\omega/t}$. We define the following algebras:
\begin{enumerate}[label=(\roman*)]
\item $H_{t,c,\omega}(X,G)_{U,\mathscr{A}_{\omega/t}} :=G\ltimes \mathcal{D}(\mathscr{B}_{\omega/t})\cap H_{t,c,\omega}(X,G)$.
\item $\mathcal{H}_{t,c,\omega}(X,G)_{U,\mathscr{A}_{\omega/t}}:=\widehat{H}_{t,c,\omega}(X,G)_{U,\mathscr{A}_{\omega/t}}\otimes_{\mathcal{R}}K.$  
\end{enumerate}
\end{defi}
We will now show some basic properties of these algebras:
\begin{Lemma}\label{Lemma basic propertis complete cherednik algebras}
Let $\mathscr{A}_{\omega/t}$ be a $U$-Cherednik lattice. The following hold:
  \begin{enumerate}[label=(\roman*)]
      \item $\mathcal{H}_{t,c,\omega}(X,G)_{U,\mathscr{A}_{\omega/t}}$ is  a Banach $K$-algebra.
      \item There is an injection of $K$-algebras $i:H_{t,c,\omega}(X,G)\rightarrow \mathcal{H}_{t,c,\omega}(X,G)_{U,\mathscr{A}_{\omega/t}}$.
      \item $H_{t,c,\omega}(X,G)$ is dense in $\mathcal{H}_{t,c,\omega}(X,G)_{U,\mathscr{A}_{\omega/t}}$.
  \end{enumerate}
\end{Lemma}
\begin{proof}
Claim $(i)$ holds because $\widehat{H}_{t,c,\omega}(X,G)_{U,\mathscr{A}_{\omega/t}}$ is complete with respect to the $\pi$-adic topology. claim $(ii)$ holds because 
$H_{t,c,\omega}(X,G)_{U,\mathscr{A}_{\omega/t}}$ is separated with respect to the $\pi$-adic topology because it is a subspace of $G\ltimes \mathcal{D}(\mathscr{B}_{\omega/t})$, together with the identity $H_{t,c,\omega}(X,G)=H_{t,c,\omega}(X,G)_{U,\mathscr{A}_{\omega/t}}\otimes_{\mathcal{R}}K$. The last statement is clear.
\end{proof}
\begin{prop}\label{Hm is saturated}
Let $\mathscr{A}_{\omega/t}$ be a $U$-Cherednik lattice. There is an injective and continuous morphism of $K$-algebras with closed image:
\begin{equation*}
    \mathcal{H}_{t,c,\omega}(X,G)_{U,\mathscr{A}_{\omega/t}}\rightarrow G\ltimes \widehat{\mathcal{D}}(\mathscr{B}_{\omega/t})_{K}.
\end{equation*}
\end{prop}
\begin{proof}
By construction we have a monomorphism of $K$-algebras
\begin{equation*}
    H_{t,c,\omega}(X,G)_{U,\mathscr{A}_{\omega/t}}\rightarrow G\ltimes \mathcal{D}(\mathscr{B}_{\omega/t}).
\end{equation*}
Thus, it suffices to show that $H_{t,c,\omega}(X,G)_{U,\mathscr{A}_{\omega/t}}$ is saturated in $G\ltimes \mathcal{D}(\mathscr{B}_{\omega/t})$. That is, we need to show that for all $f\in G\ltimes \mathcal{D}(\mathscr{B}_{\omega/t})$, if there is some $n\geq 0$ such that $\pi^nf\in H_{t,c,\omega}(X,G)_{U,\mathscr{A}_{\omega/t}}$, then $f\in H_{t,c,\omega}(X,G)_{U,\mathscr{A}_{\omega/t}}$. Choose $f\in G\ltimes \mathcal{D}(\mathscr{B}_{\omega/t})$ such that $\pi^nf\in H_{t,c,\omega}(X,G)_{U,\mathscr{A}_{\omega/t}}$ for high enough $n\geq 0$. Then we have $f\in H_{t,c,\omega}(X,G)$. Therefore, we have:
\begin{equation*}
   f\in H_{t,c,\omega}(X,G)_{U,\mathscr{A}_{\omega/t}} =G\ltimes \mathcal{D}_{\omega/t}(\mathscr{B}_{\omega/t})\cap H_{t,c,\omega}(X,G). 
\end{equation*}
Thus, $H_{t,c,\omega}(X,G)_{U,\mathscr{A}_{\omega/t}}$ is saturated in $G\ltimes \mathcal{D}(\mathscr{B}_{\omega/t})$, and the claim holds.
\end{proof}
For each $m\geq 0$ we have the following commutative diagram:
\begin{center}
    % https://tikzcd.yichuanshen.de/#N4Igdg9gJgpgziAXAbVABwnAlgFyxMJZABgBoBGAXVJADcBDAGwFcYkQAdDgW3pwAsAxk2AAJAL4A9YGHEB9YDlKDxACgAapAOIBKBQE1SAFQUBBaV0FYATioD0XXgOGNgAJXGeQ40uky58QhQyYmo6JlZ2Rz4hEQlpMABqcnlFZTVNXQNjMwsOK1txBx4Yl3dPcW9fEAxsPAIiclJQmgYWNkQQLS5GPG54AAIuAHcsWH4+YABVBQBpNS40LATk8QAhPIKVLgg++FzgSxsVcRNgc0P846Lo5xEPcR1Knz86wMaKMLbIzu6OXqw-TgQw4o3GkxmwHmqkWyxk6021x2ezgByOhVOaKuhWKTlirgeT28YRgUAA5vAiKAAGbWCDcJBNEA4CBIYgvEC0+lIABMNBZSAAzByuQzEEyBYgeSK6WKyMzWYhhZRxEA
\begin{tikzcd}
{\mathcal{H}_{t,c,\omega}(X,G)_{U,\pi^{m+1}\mathscr{A}_{\omega/t}}} \arrow[d] \arrow[r] & G\ltimes \widehat{\mathcal{D}}(\pi^{m+1}\mathscr{B}_{\omega/t})_{K} \arrow[d] \\
{\mathcal{H}_{t,c,\omega}(X,G)_{U,\pi^{m}\mathscr{A}_{\omega/t}}} \arrow[r]             & G\ltimes \widehat{\mathcal{D}}(\pi^{m}\mathscr{B}_{\omega/t})_{K}            
\end{tikzcd}
\end{center}
where horizontal maps are injective. This leads us to the following definition:
\begin{defi}
Let $\mathscr{A}_{\omega/t}$ be a $U$-Cherednik lattice.
The $U$-Cherednik algebra of the $G$-variety $X$ (relative to our choice of $t$,$c$,$\omega$) is the following inverse limit: 
\begin{equation*}
    \mathcal{H}_{t,c,\omega}(X,G)_{U}=\varprojlim_{m}\mathcal{H}_{t,c,\omega}(X,G)_{U,\pi^{m}\mathscr{A}_{\omega/t}}.
\end{equation*}
We regard $\mathcal{H}_{t,c,\omega}(X,G)_{U}$ as a topological $K$-algebra with the inverse limit topology.
\end{defi}
We need to show this is well-defined. That is, that it does not depend on $\mathscr{A}_{\omega/t}$. Chose another $U$-Cherednik lattice  $\widetilde{\mathscr{A}}_{\omega/t}$. As both $\mathscr{A}_{\omega/t}$ and $\widetilde{\mathscr{A}}_{\omega/t}$ are finite-free $(\mathcal{R},A^{\circ})$-Lie lattices of $\mathcal{A}_{\omega/t}$,
there are $s_{1},s_{2}\geq 0$ such that for $m\geq 0$ we have:
\begin{equation*}
    \pi^{s_{2}+m}\mathscr{A}_{\omega/t}\subset \pi^{s_{1}+m}\widetilde{\mathscr{A}}_{\omega/t}\subset \pi^{m}\mathscr{A}_{\omega/t}.
\end{equation*}
Applying $B^{\circ}\otimes_{A^{\circ}}-$ to this sequence, we get:
\begin{equation}\label{equation showing U-cher algebras are independent of lattice}
    \pi^{s_{2}+m}\mathscr{B}_{\omega/t}\rightarrow \pi^{s_{1}+m}\widetilde{\mathscr{B}}_{\omega/t}\rightarrow  \pi^{m}\mathscr{B}_{\omega/t}.
\end{equation}
By definition of a $U$-Cherednik lattice, we have:
\begin{equation*}
   K\otimes_{\mathcal{R}} \pi^{m}\mathscr{B}_{\omega/t}=B\otimes_{A}\mathcal{A}_{\omega}(X)=\mathcal{A}_{\omega}(U). 
\end{equation*}
Thus, any element in the kernel of one of the maps $\pi^{s_{2}+m}\mathscr{B}_{\omega/t}\rightarrow \pi^{s_{1}+m}\widetilde{\mathscr{B}}_{\omega/t}$ must be $\pi$-torsion. As these are free modules over the $\mathcal{R}$-flat algebra $B^{\circ}$, the maps are injective. Furthermore, the maps in $(\ref{equation showing U-cher algebras are independent of lattice})$ are $G$-equivariant maps of $(\mathcal{R},B^{\circ})$-Lie algebras. Thus, we get a sequence of $G$-equivariant injective maps of $K$-algebras:
\begin{equation*}
    \mathcal{D}(\pi^{s_{2}+m}\mathscr{B}_{\omega/t})\rightarrow  \mathcal{D}(\pi^{s_{1}+m}\widetilde{\mathscr{B}}_{\omega/t})\rightarrow  \mathcal{D}(\pi^{m}\mathscr{B}_{\omega/t}).
\end{equation*}
 After applying $\pi$-adic completion, tensoring by $K$, and taking the skew group ring with respect to the action of $G$, we arrive at the following commutative diagram:
\begin{equation*}
    % https://tikzcd.yichuanshen.de/#N4Igdg9gJgpgziAXAbVABwnAlgFyxMJZABgBoBGAXVJADcBDAGwFcYkQACAcQB0fG8AW3gcAqgAoAQgD1gfAMZYATvIC+fCEPgB9YAEFZC5WvU80WWXF0AmVQGowqgCq6DcnopWqA9H0H0cAAt5JmAAJVVVAEoQVVJ0TFx8QhRyCmo6JlZ2Xn4tODEpQw9jU00sYSt9Ys8TPnNLXXJ7Rw1aGCVGLDAYYCdVVxrS3x5-IJDGcMiYuITsPAIia3SaBhY2RBBcgQqRCRl3WrL8wcPS0wbgRxdqs68RseDQiOjY+JAMeeSiMmIMteymwAEo1gLYHANgDhSGpxAANUhcKK6USkG5uIz3PwBJ6TF6qN5zJKLVKkP6rLIbEAg4BVZoQ3TQ2EIpEo0htDpdHp9SEYkpY0Y4iZTSKEj6JBYpZDLcmZdbsGmORkw1TwxHI4Co9FDAWPYX42IZGBQADm8CIoAAZkoIIIkGQQDgIEhmu9rbaXTQnUhbG6bXbEABWL3OxC+q3+pAAFhDLtmIHdAYAzLHEMR44mkCnHaGoxnI4gYzmkIHVJRVEA
\begin{tikzcd}
{\mathcal{H}_{t,c,\omega}(X,G)_{U,\pi^{s_{2}+m}\mathscr{A}_{\omega/t}}} \arrow[d] \arrow[r]                & {\mathcal{H}_{t,c,\omega}(X,G)_{U,\pi^{s_{1}+m}\widetilde{\mathscr{A}}_{\omega/t}}} \arrow[d] \arrow[r]               & {\mathcal{H}_{t,c,\omega}(X,G)_{U,\pi^{m}\mathscr{A}_{\omega/t}}} \arrow[d]               \\
 G\ltimes \widehat{\mathcal{D}}(\pi^{s_{2}+m}\mathscr{B}_{\omega/t})_{K} \arrow[r] & G\ltimes \widehat{\mathcal{D}}(\pi^{s_{1}+m}\widetilde{\mathscr{B}}_{\omega/t})_{K} \arrow[r] & G\ltimes \widehat{\mathcal{D}}(\pi^{m}\mathscr{B}_{\omega/t})_{K}
\end{tikzcd}
\end{equation*}
 Taking inverse limits we get two maps:
 \begin{equation*}
    \varprojlim_{m} \mathcal{H}_{t,c,\omega}(X,G)_{U,\pi^m\mathscr{A}_{\omega/t}}\rightarrow  \varprojlim_{m}\mathcal{H}_{t,c,\omega}(X,G)_{U,\pi^m\widetilde{\mathscr{A}}_{\omega/t}}\rightarrow  \varprojlim_{m}\mathcal{H}_{t,c,\omega}(X,G)_{U,\pi^m\mathscr{A}_{\omega/t}},
 \end{equation*}
which are mutually inverse. All in all, we have shown the following:
\begin{Lemma}
Let $\mathscr{A}_{\omega/t},\widetilde{\mathscr{A}}_{\omega/t}$ be two $U$-Cherednik lattices. There is a unique isomorphism of topological $K$-algebras: 
\begin{equation*}
    \varprojlim_{m} \mathcal{H}_{t,c,\omega}(X,G)_{U,\pi^m\mathscr{A}_{\omega/t}}\rightarrow  \varprojlim_{m}\mathcal{H}_{t,c,\omega}(X,G)_{U,\pi^m\widetilde{\mathscr{A}}_{\omega/t}}.
\end{equation*}
Furthermore, this isomorphism is the identity on $H_{t,c,\omega}(X,G)$.
\end{Lemma}

Hence, we have shown that $\mathcal{H}_{t,c,\omega}(X,G)_{U}$ is well-defined, and can be calculated using any $U$-Cherednik lattice. The following consequences arise from the definition:
\begin{prop}\label{Lemma basic porperties of U-Cher algebras}
Let $\mathcal{H}_{t,c,\omega}(X,G)_{U}$ be the $U$-Cherednik algebra. We have:
\begin{enumerate}[label=(\roman*)]
    \item The map $i:H_{t,c,\omega}(X,G)\rightarrow\mathcal{H}_{t,c,\omega}(X,G)_{U}$ is injective with dense image.
    \item The map $ \mathcal{H}_{t,c,\omega}(X,G)_{U}\rightarrow G\ltimes \wideparen{\mathcal{D}}_{\omega/t}(U)$ is injective with closed image.
\end{enumerate}
In particular, $\mathcal{H}_{t,c,\omega}(X,G)_{U}$  is the closure of $H_{t,c,\omega}(X,G)$ in $G\ltimes \wideparen{\mathcal{D}}_{\omega/t}(U)$.
\end{prop}
\begin{proof}
Claim $(i)$  is a consequence of Lemma \ref{Lemma basic propertis complete cherednik algebras} and the definition of the inverse limit topology. Claim $(ii)$ follows by Proposition \ref{Hm is saturated}, commutativity of the diagram above, and the fact that inverse limits are left exact.
\end{proof}
Ideally, we would like that  $\mathcal{H}_{t,c,\omega}(X,G)_{U}$ is also independent of the $G$-stable affinoid subspace $U\subset \overline{X}$. This will in general not hold, as for a pair of affinoid subspaces $V=\Sp(C)\subset U=\Sp(B)\subset X$, it could happen that $B^{\circ}\cap A\neq C^{\circ}\cap A$. Thus, $\mathcal{H}_{t,c,\omega}(X,G)_{V}$ will contain more rigid functions than $\mathcal{H}_{t,c,\omega}(X,G)_{U}$. 
\begin{Lemma}\label{transition maps in definition}
Let $V\subset U\subset \overline{X}$ be two $G$-stable affinoid subdomains which meet all connected components of $X$. Then there is a unique continuous morphism of topological $K$-algebras:
\begin{equation*}
    \mathcal{H}_{t,c,\omega}(X,G)_{U}\rightarrow \mathcal{H}_{t,c,\omega}(X,G)_{V},
\end{equation*}
which is the identity on $H_{t,c,\omega}(X,G)$.
\end{Lemma}
\begin{proof}
 By construction, the map $G\ltimes \wideparen{\D}_{\omega/t}(U)\rightarrow G\ltimes \wideparen{\D}_{\omega/t}(V)$ is continuous, and is the identity when restricted to $H_{t,c,\omega}(X,G)$. Hence, it induces a map between the closures of $H_{t,c,\omega}(X,G)$ in both spaces, which restricts to the identity on  $H_{t,c,\omega}(X,G)$. The lemma now follows by the last part of Proposition \ref{Lemma basic porperties of U-Cher algebras}.
\end{proof}
In order to find a definition which does not depend on any choice, we notice that $\overline{X}$ has a canonical increasing affinoid cover. Indeed, let $Z$
be the union of all reflection hypersurfaces. This is a closed analytic subspace associated to an ideal:
\begin{equation*}
    \mathcal{I}_Z=f_1 A+\cdots + f_r A.
\end{equation*}
Thus, we have an admissible cover of $\overline{X}$  given by the Laurent subdomains:
\begin{equation*}
    \overline{X}_{n}=X\left( \frac{\pi^n}{f_1},\cdots,\frac{\pi^n}{f_r} \right):=\{x\in X \vert \textnormal{ } \vert f_{i}(x)\vert\geq \vert \pi^{n}\vert \textnormal{, for all  }  1\leq i\leq r \}.
\end{equation*}
As $G$ acts on $S(X,G)$ by conjugation, each of the $\overline{X}_{n}$ is a $G$-stable affinoid subdomain of $X$, which we can assume to meet all the connected components of $X$.
\begin{defi}
The  $p$-adic Cherednik algebra of the $G$-variety $X$ (relative to our choice of $t$,$c$,$\omega$) is the following inverse limit: 
\begin{equation*}
    \mathcal{H}_{t,c,\omega}(X,G)=\varprojlim_{n}\mathcal{H}_{t,c,\omega}(X,G)_{\overline{X}_{n}}.
\end{equation*}
We regard it as a topological $K$-algebra with respect to its inverse limit topology.
\end{defi}
\begin{obs}\label{scalateing the parameters of cherednik algebras}
We make the following observations:
\begin{enumerate}[label=(\roman*)]
    \item For any $\lambda\in K^{*}$ we have $\mathcal{H}_{t,c,\omega}(X,G)=\mathcal{H}_{\lambda t,\lambda c,\lambda\omega}(X,G)$. 
    \item $\mathcal{H}_{t,c,\omega}(X,G)$ is the closure of $H_{t,c,\omega}(X,G)$ in $G\ltimes \wideparen{\mathcal{D}}_{\omega/t}(\overline{X})$.
\end{enumerate}
\end{obs}

In light of this remark, and the fact that $t\in K^*$, it would be possible to set $t=1$ in the definition of $p$-adic Cherednik algebra and remove it from the notation. However, it is important to mention that, in the classical case, the parameter $t=0$ is allowed, and leads to Cherednik algebras of a slightly different nature (\emph{cf.} \cite[Section 2.10]{etingof2010lecture},  \cite[pp. 8]{etingof2004cherednik}). Thus, we will restrain from dropping the $t$ from the notation, and simply make the assumption that $t=1$ when it simplifies 
calculations. 
\subsection{p-adic Cherednik algebras as Fréchet-Stein algebras}\label{section Completed Cherdnik algebras are Fréchet-Stein}
Our next goal will be endowing $p$-adic Cherednik algebras with a Fréchet-Stein structure. Let $X=\Sp(A)$ be an affinoid $G$-variety with an étale map $X\rightarrow \mathbb{A}^r_K$. From now and until the end of the section we fix a $G$-stable affinoid open subspace $U=\Sp(B)\subset \overline{X}$ which meets all connected components of $X$, together with $t\in K^*$, $c\in \textnormal{Ref}(X,G)$, and $\omega\in\Omega_{X/K}^{2,\operatorname{cl}}(X)^G$. By Proposition \ref{prop Cherednik lattices}, we may choose some $U$-Cherednik lattice $\mathscr{A}_{\omega/t}$, and let $\mathscr{T}_{\mathscr{A}_{\omega/t}}$ be its image under the anchor map $\mathcal{A}_{\omega/t}(X)\rightarrow \mathcal{T}_{X/K}(X)$. Let us start by extending Theorem \ref{Classic PBW Theorem} to the integral setting:
\begin{defi}
We define the following filtrations:
\begin{enumerate}[label=(\roman*)]
    \item The Dunkl-Opdam filtration on $H_{t,c,\omega}(X,G)_{U,\mathscr{A}_{\omega/t}}$ is defined by:
 \begin{equation*}
    F_{i}H_{t,c,\omega}(X,G)_{U,\mathscr{A}_{\omega/t}}=F_{i}H_{t,c,\omega}(X,G)\cap H_{t,c,\omega}(X,G)_{U,\mathscr{A}_{\omega/t}}, 
 \end{equation*}   
where $F_{\bullet}H_{t,c,\omega}(X,G)$ is the Dunkl-Opdam filtration.
    \item The filtration by order of differential operators on $\mathcal{D}(\mathscr{B}_{\omega/t})$  is defined by;
\begin{equation*}
    \Phi_i\mathcal{D}(\mathscr{B}_{\omega/t})=\mathcal{D}(\mathscr{B}_{\omega/t})\cap \Phi_i\D_{\omega/t}(U), \textnormal{ for all } i\geq 0,
\end{equation*}
where $\Phi_{\bullet} \D_{\omega/t}(U)$ is the filtration by order of differential operators.
\end{enumerate}
\end{defi}
We can use these filtrations to obtain a PBW Theorem at the integral level:
\begin{prop}\label{pbw uncompleted microlocal level}
 There is an isomorphism of graded $G\ltimes A^{\circ}_B $algebras:
    \begin{equation*}
       \overline{\Psi}: \gr_{F}H_{t,c,\omega}(X,G)_{U,\mathscr{A}_{\omega/t}}\rightarrow  G\ltimes \operatorname{Sym}_{A_B^{\circ}}(A_B^{\circ}\otimes_{A^{\circ}}\mathscr{T}_{\mathscr{A}_{\omega/t}}),
    \end{equation*}  
    which sends a Dunkl-Opdam operator $D_{v/t}$ to $v$.
\end{prop}
\begin{proof}
Let $\mathcal{B}=\{v_{1},\cdots,v_{r}\}$ be an $A^{\circ}$-basis of $\mathscr{T}_{\mathscr{A}_{\omega/t}}$ and 
fix any $m\geq 0$. By construction, we have the following commutative square of filtered algebras:
    \begin{center}
\begin{tikzcd}
{H_{t,c,\omega}(X,G)} \arrow[r]                                                            & G\ltimes\mathcal{D}_{\omega}(U)                                          \\
{H_{t,c,\omega}(X,G)_{U,\mathscr{A}_{\omega/t}}} \arrow[u] \arrow[r] & G\ltimes\mathcal{D}(\mathscr{B}_{\omega/t}) \arrow[u]
\end{tikzcd}
    \end{center}
Passing to the associated graded, we get a commutative diagram of graded algebras:
\begin{center}
    \begin{tikzcd}
{\gr_{F}H_{t,c,\omega}(X,G)} \arrow[rr, "\Psi"]                                                                       &  & G\ltimes \operatorname{Sym}_{B}\mathcal{T}_{X/K}(U)                          \\
{\gr_{F}H_{t,c,\omega}(X,G)_{U,\mathscr{A}_{\omega/t}}} \arrow[u] \arrow[rr, "\overline{\Psi}"] &  & G\ltimes \operatorname{Sym}_{B^{\circ}}(\mathscr{T}_{\mathscr{B}_{\omega/t}}) \arrow[u]
\end{tikzcd}
\end{center}
By Theorem \ref{Affinoid PBW Theorem}, $\Psi:\gr_{F}H_{t,c,\omega}(X,G)\rightarrow G\ltimes \operatorname{Sym}_{B}\mathcal{T}_{X/K}(U)$ is injective and has image $G\ltimes \operatorname{Sym}_{A}\mathcal{T}_{X/K}(X)$. Furthermore, $\Psi$ maps a Dunkl-Opdam operator $D_{v}$ to $tv\in \mathcal{T}_{X/K}(X)$, and is the identity on $G\ltimes A$. As all other maps are injective, it follows that $\overline{\Psi}$ is injective.
Thus, it is an isomorphism onto its image. We just need to show that its image is $G\ltimes \operatorname{Sym}_{A_B^{\circ}}(A_B^{\circ}\otimes_{A^{\circ}}\mathscr{T}_{\mathscr{A}_{\omega/t}})$. The algebra  $G\ltimes \operatorname{Sym}_{A_B^{\circ}}(\mathscr{T}_{\mathscr{A}_{\omega/t}}\otimes_{A^{\circ}}A_B^{\circ})$ is generated as an algebra over $G\ltimes A^{\circ}_B$
by the symbols $v_i$, for $1\leq i\leq r$. By definition of a $U$-Cherednik lattice, $H_{t,c,\omega}(X,G)_{U,\mathscr{A}_{\omega/t}}$ contains Dunkl-Opdam operators $D_{v/t}$ for each $v\in\mathcal{B}$. Hence, 
its image under $\overline{\Psi}$ contains the symbols $v_1,\cdots,v_r$. By the definition of the restriction rings given in  Definition \ref{defi restriction rings}, and the fact that the map $A\rightarrow B$ is injective (as $U$ meets all connected components of $X$), we have $A^{\circ}_B=B^{\circ}\cap A$. Hence, we have:
\begin{equation*}
    G\ltimes A^{\circ}_B=(G\ltimes A)\cap (G\ltimes B^{\circ})\subset H_{t,c,\omega}(X,G) \cap  G\ltimes\mathcal{D}_{\omega}(\mathscr{B}_{\omega/t}) =H_{t,c,\omega}(X,G)_{U,\mathscr{A}_{\omega/t}}.
\end{equation*}
 As $\Psi$ is the identity on $G\ltimes A$, it follows that $G\ltimes A^{\circ}_B$ is contained in the image of $\overline{\Psi}$. Thus, the graded algebra $G\ltimes \operatorname{Sym}_{A_B^{\circ}}(\pi^{m}\mathscr{T}_{\mathscr{A}_{\omega/t}}\otimes_{A^{\circ}}A_B^{\circ})$ is contained in the image of $\overline{\Psi}$. On the other hand:
 \begin{multline*}
 \overline{\Psi}(H_{t,c,\omega}(X,G)_{U,\mathscr{A}_{\omega/t}})\subset  G\ltimes \operatorname{Sym}_A\mathcal{T}_{X/K}(X)\cap   G\ltimes \operatorname{Sym}_{B^{\circ}}(\mathscr{T}_{\mathscr{B}_{\omega/t}})\\
 =G\ltimes \operatorname{Sym}_{A_B^{\circ}}(A_B^{\circ}\otimes_{A^{\circ}}\mathscr{T}_{\mathscr{A}_{\omega/t}}).
 \end{multline*}
\end{proof}
Just as in the non-integral case, we get the following corollary:
\begin{coro}\label{coro basis of Cherednik algebras at the tintegral level}
   $H_{t,c,\omega}(X,G)_{U,\mathscr{A}_{\omega/t}}$ is a free $G\ltimes A^{\circ}_B$-module, with basis given by:
   \begin{equation*}
       D_{v_{1}/t}^{\alpha_{1}}\cdots D_{v_{r}/t}^{\alpha_{r}},
   \end{equation*}
   for non-negative integers $\alpha_{1},\cdots,\alpha_{r}\geq 0$. 
\end{coro}
Although the rings $A^{\circ}_B$ can be rather strange, the material in Section \ref{section restriction rings} gives a precise description whenever $U$ contains the Shilov boundary of $X$. In particular, in this case we have $A^{\circ}_B=A^{\circ}$. We may use this to show the following theorem:
\begin{teo}\label{teo simplification of definition of completed cherednik algebras}
Assume that $S(X)\subset U\subset X$ is a $G$-invariant affinoid subdomain. Then we have a canonical isomorphism of topological $K$-algebras:
\begin{equation*}
    \mathcal{H}_{t,c,\omega}(X,G)\rightarrow \mathcal{H}_{t,c,\omega}(X,G)_{U}.
\end{equation*}
Furthermore, this isomorphism is the identity on  $H_{t,c,\omega}(X,G)$.
\end{teo}
\begin{proof}
    By definition we have $ \mathcal{H}_{t,c,\omega}(X,G)=\varprojlim_{n}\mathcal{H}_{t,c,\omega}(X,G)_{\overline{X}_{n}}$ as topological algebras. As the cover $\overline{X}=\cup_n\overline{X}_{n}$ is admissible and $U$ is affinoid, there is some $n\geq 1$ such that $U\subset \overline{X}_{n}$. It suffices to show that the canonical map:
    \begin{equation*}
       \varphi: \mathcal{H}_{t,c,\omega}(X,G)_{\overline{X}_{n}}\rightarrow \mathcal{H}_{t,c,\omega}(X,G)_{U},
    \end{equation*}
    is an isomorphism of topological $K$-algebras which is the identity on $H_{t,c,\omega}(X,G)$. Let $\mathscr{A}_{\omega/t}$ be a Cherednik lattice for $U$ and $\overline{X}_n$. It suffices to show that the map:
    \begin{equation*}
        \varphi_m:H_{t,c,\omega}(X,G)_{\overline{X}_{n},\pi^m\mathscr{A}_{\omega/t}}\rightarrow H_{t,c,\omega}(X,G)_{U,\pi^m\mathscr{A}_{\omega/t}}
    \end{equation*}
    is an isomorphism for $m\geq 0$. However, as $U$ contains the Shilov boundary of $X$, it follows that $\overline{X}_n$ also contains it. Hence, Corollary \ref{coro basis of Cherednik algebras at the tintegral level} shows that both algebras are free modules over $G\ltimes A^{\circ}$ with the same basis, so $\varphi_m$ is an isomorphism.
\end{proof}
\begin{obs}
  Let $\mathscr{A}_{\omega/t}$  be a $(\mathcal{R},A^{\circ})$-Lie lattice of $\mathcal{A}_{\omega}(X)$, and consider $G$-invariant affinoid subdomains $U,V\subset \overline{X}$ which contain $S(X)$ and  such that 
  $\mathscr{A}_{\omega/t}$ is a Cherednik lattice for $U$ and $V$. Then we have:
  \begin{equation*}
      H_{t,c,\omega}(X,G)_{V,\mathscr{A}_{\omega/t}}= H_{t,c,\omega}(X,G)_{U,\mathscr{A}_{\omega/t}}.
  \end{equation*}
Thus, we can drop the choice of affinoid subdomain from the notation, and write
\begin{equation*}
 H_{t,c,\omega}(X,G)_{\mathscr{A}_{\omega/t}}:=H_{t,c,\omega}(X,G)_{U,\mathscr{A}_{\omega/t}}, \textnormal{ }   \mathcal{H}_{t,c,\omega}(X,G)_{\mathscr{A}_{\omega/t}}:=\mathcal{H}_{t,c,\omega}(X,G)_{U,\mathscr{A}_{\omega/t}}.
\end{equation*}
We will follow this convention until the end of the paper.
\end{obs}
Thus, the constructions in Section \ref{sect local definition cher alg} are independent of the choice of $G$-invariant affinoid subdomain $U\subset \overline{X}$. This motivates the following definition:
\begin{defi}
Let $\mathscr{A}_{\omega/t}\subset \mathcal{A}_{\omega/t}(X)$ be a $G$-invariant finite-free $(\mathcal{R},A^{\circ})$-Lie lattice. We say $\mathscr{A}_{\omega/t}$ is a Cherednik lattice of $\mathcal{A}_{\omega/t}(X)$ if there is a $G$-invariant affinoid subdomain $U\subset X$ satisfying:
\begin{equation*}
    S(X)\subset U\subset \overline{X},
\end{equation*}
and such that $\mathscr{A}_{\omega/t}$ is a $U$-Cherednik lattice.
\end{defi}
One of the main features of Cherednik lattices is that they share many good properties with the smooth lattices of Lie algebroids:
\begin{Lemma}\label{Lemma scalating the lattices}
Let $\mathscr{A}_{\omega/t}$ be a $G$-invariant finite-free $(\mathcal{R},A^{\circ})$-Lie lattice  of $\mathcal{A}_{\omega/t}(X)$. There is some non-negative integer $m\geq 0$ such that for all $n\geq m$,  $\pi^n\mathscr{A}_{\omega/t}$ is a Cherednik lattice of $\mathcal{A}_{\omega/t}(X)$.
\end{Lemma}
\begin{proof}
Choose a $G$-invariant affinoid subdomain $U\subset X$ which contains the Shilov boundary of $X$. By Lemma \ref{Lemma 7.6 ardakov}, it follows that $U$ is $\pi^n\mathscr{A}_{\omega/t}$-stable for sufficiently high $n$. On the other hand, the proof of Proposition \ref{prop Cherednik lattices} then shows that 
$\pi^n\mathscr{A}_{\omega/t}$ is a $U$-Cherednik lattice for big enough $n\geq 0$.   
\end{proof}
Hence, if $\mathscr{A}_{\omega/t}$ is a Cherednik lattice of $\mathcal{A}_{\omega/t}(X)$, Theorem \ref{teo simplification of definition of completed cherednik algebras} implies that we have the following identities of topological $K$-algebras:
\begin{equation*}
    \mathcal{H}_{t,c,\omega}(X,G)=\mathcal{H}_{t,c,\omega}(X,G)_{U}=\varprojlim_{m}\mathcal{H}_{t,c,\omega}(X,G)_{\pi^m\mathscr{A}_{\omega/t}}.
\end{equation*}

From now until the end of this section, fix a Cherednik lattice $\mathscr{A}_{\omega/t}$ of $\mathcal{A}_{\omega/t}(X)$. The goal of this section is showing that the inverse limit of Banach $K$-algebras:
\begin{equation*}
    \mathcal{H}_{t,c,\omega}(X,G)=\varprojlim_{m}\mathcal{H}_{t,c,\omega}(X,G)_{\pi^m\mathscr{A}_{\omega/t}},
\end{equation*}
is a Fréchet-Stein presentation of $\mathcal{H}_{t,c,\omega}(X,G)$. To simplify the calculations, we will assume $t=1$. We start by recalling a well-known fact on filtered algebras:
\begin{Lemma}\label{noetherianness of associated graded implies noetheriannes of filtered ring}
Let $R$ be a ring with a positive and exhaustive  filtration $F_{\bullet}R$. If the associated graded $\gr_{F}R$ is left (right) noetherian, then $R$ is left (right) noetherian. 
\end{Lemma}
\begin{proof}
This a consequence of \cite[Proposition D.1.1.]{hotta2007d}.
\end{proof}
This has the following straightforward consequence:
\begin{prop}\label{microlocal CHerednik algebras are noetherian}
The following hold for each $m\geq 0$:
\begin{enumerate}[label=(\roman*)]
    \item $\mathcal{H}_{1,c,\omega}(X,G)_{\pi^m\mathscr{A}_{\omega/t}}$ is two-sided noetherian.
    \item $\pi_{m}:\mathcal{H}_{1,c,\omega}(X,G)_{\pi^{m+1}\mathscr{A}_{\omega}}\rightarrow \mathcal{H}_{1,c,\omega}(X,G)_{\pi^{m}\mathscr{A}_{\omega}}$ has dense image.
\end{enumerate}
\end{prop}
\begin{proof}
We show the case $m=0$. As $\mathcal{H}_{1,c,\omega}(X,G)_{\mathscr{A}_{\omega}}=\widehat{H}_{1,c,\omega}(X,G)_{\mathscr{A}_{\omega}}[\pi^{-1}]$, and being noetherian is preserved under completion and localization, it is enough to show that $H_{1,c,\omega}(X,G)_{\mathscr{A}_{\omega}}$ is two-sided noetherian. By Proposition 
\ref{pbw uncompleted microlocal level} we have:
\begin{equation*}
    \gr_{F}H_{1,c,\omega}(X,G)_{\mathscr{A}_{\omega}}\cong G\ltimes \operatorname{Sym}_{A^{\circ}}(\mathscr{T}_{\mathscr{A}_{\omega}}).
\end{equation*}
As $K$ is discretely valued and $G$ is finite, the right hand side is noetherian. Thus, the proposition follows by Lemma \ref{noetherianness of associated graded implies noetheriannes of filtered ring}. For the second statement, we notice that $H_{1,c,\omega}(X,G)$ is dense in $\mathcal{H}_{1,c,\omega}(X,G)_{\pi^{m}\mathscr{A}_{\omega}}$ for every $m\geq 0$ by Lemma \ref{Lemma basic propertis complete cherednik algebras}. As the maps $\pi_{m}$ are the identity on $H_{1,c,\omega}(X,G)$, the result follows.
\end{proof}
The last step is showing that the following maps are two-sided flat: 
\begin{equation*}
    \pi_{m}:\mathcal{H}_{1,c,\omega}(X,G)_{\pi^{m+1}\mathscr{A}_{\omega}}\rightarrow \mathcal{H}_{1,c,\omega}(X,G)_{\pi^{m}\mathscr{A}_{\omega}}
\end{equation*}
The strategy will be adapting Sections 6.4 to 6.7 of \cite{ardakov2019} to our setting. Even though these sections are written in the language of noetherian, deformable, almost-commutative algebras, we will see that we can bridge the gaps by understanding the graded structure of the algebras $\gr_{F}H_{1,c,\omega}(X,G)_{\pi^m\mathscr{A}_{\omega}}=G\ltimes \operatorname{Sym}_{A^{\circ}}(\pi^{m}\mathscr{T}_{\mathscr{A}_{\omega}})$.
 \begin{obs}
Until the end of the section, we will write $R_{m}:=H_{1,c,\omega}(X,G)_{\pi^m\mathscr{A}_{\omega}}$ and $R:=R_{0}$ whenever it is convenient.     
 \end{obs}
Start by recalling the definition of a deformable $\mathcal{R}$-algebra:
\begin{defi}[{\cite[Section 3.5]{ardakov2013irreducible}}]
Let $R$ be a $\mathcal{R}$-algebra with a positive and exhaustive filtration $F_{\bullet}R$. We say $R$ is a deformable algebra if $\gr R$ is a flat $\mathcal{R}$-module. Let $R$ be a deformable algebra. For a non-negative integer $m\geq 0$, the  $m$-th deformation of $R$ is defined by the following rule:
   \begin{equation*}
       (R)_{m}=\sum_{i\geq 0}\pi^{im}F_{i}R.   
   \end{equation*}
  This is a filtered algebra, with filtration is given by the following rule:
  \begin{equation*}
      F_{j}(R)_{m}=(R)_{m}\cap F_{j}R=\sum_{i=0}^{j}\pi^{im}F_{i}R.
  \end{equation*}
\end{defi}
We may apply this definition to our setting as follows:
\begin{prop}\label{prop preparation  big flatness Theorem}
    The following hold:
    \begin{enumerate}[label=(\roman*)]
        \item The filtered algebra $H_{1,c,\omega}(X,G)_{\mathscr{A}_{\omega}}$ is a deformable $\mathcal{R}$-algebra.
        \item The identity $H_{1,c,\omega}(X,G)_{\pi^m\mathscr{A}_{\omega}}=(H_{1,c,\omega}(X,G)_{\mathscr{A}_{\omega}})_m$ holds for each $m\geq 0$.
    \end{enumerate}
\end{prop}
\begin{proof}
 As we have $\gr_{F}R=G\ltimes \operatorname{Sym}_{A^{\circ}}(\mathscr{T}_{\mathscr{A}_{\omega}})$, it follows that $R$ is deformable. For statement $(ii)$, we need to show that $(R)_{m+1}=R_{m+1}$. We can proceed inductively by showing that $F_{j}(R)_{m}=F_{j}R_{m}$ for each $j\geq 0$. The base case of $j=0$ follows at once from the fact that, by definition of the Dunkl-Opdam filtration, we have:
\begin{equation*}
    F_{0}(R)_{m}=F_{0}R=G\ltimes A^{\circ}=F_{0}R_{m}.
\end{equation*}
Assume now that $F_{j}(R)_{m}=F_{j}R_{m}$. We need to show that $F_{j+1}R_{m}\subset F_{j+1}(R)_{m}$. Let $v_1,\cdots,v_r$ be an $A^{\circ}$-basis of $\mathscr{T}_{\mathscr{A}_{\omega}}$, and $D_{v_{1}},\cdots, D_{v_{r}}$ be a choice of Dunkl-Opdam operators.
  As we have $\gr_{F}R_{m}=G\ltimes \operatorname{Sym}_{A^{\circ}}(\pi^{m}\mathscr{T}_{\mathscr{A}_{\omega}})$, it follows that every element in $F_{j+1}R_{m}\setminus F_{j}R_{m}$ can be written uniquely as a sum of $G\ltimes A^{\circ}$-linear combinations of elements of the form:
  \begin{equation*}
      (\pi^{m}D_{v_{1}})^{\alpha_{1}}\cdots (\pi^{m}D_{v_{r}})^{\alpha_{r}} \textnormal{,  }\alpha_{1}+\cdots+\alpha_{r}=j+1.
  \end{equation*}
  Our induction hypothesis implies that it is enough to show that all elements of this form are contained in $F_{j+1}(R)_{m}$. Notice that $\pi \in \mathcal{R}$ is in the center of $R$. Hence, as elements in $R$ we have:
  \begin{equation*}
      (\pi^{m}D_{v_{1}})^{\alpha_{1}}\cdots (\pi^{m}D_{v_{r}})^{\alpha_{r}}=\pi^{m(j+1)}D_{v_{1}}^{\alpha_{1}}\cdots D_{v_{r}}^{\alpha_{r}},
  \end{equation*}
and this is an element of $\pi^{m(j+1)}F_{j+1}R\subset F_{j+1}(R)_{m}$. The converse is analogous.
\end{proof}
We will use this proposition to compare the $\pi$-adic topology of the deformation $R_{1}$ with its subspace topology coming from the $\pi$-adic topology on $R$. 
\begin{Lemma}[{\cite[Lemma 6.4]{ardakov2019}}]\label{Lemma ardakow 6.4}
    Let $R$ be a deformable $\mathcal{R}$-algebra. Then the following hold:
    \begin{enumerate}[label=(\roman*)]
        \item $R_{1}\cap \pi^{n}R=\sum_{i\geq n}\pi^{i}F_{i}R$ for any $n\geq 0$.
        \item $(R_{n})_{m}=R_{nm}$ for any $n,m\geq 0$.
    \end{enumerate}
\end{Lemma}
The following proposition is an adaptation of \cite[Proposition 6.5]{ardakov2019}, in particular, we just need to check that the proof still holds in this context.
\begin{prop}\label{subspace topology}
 Let $I=R_{1}\cap \pi R$. The subspace filtration on $R_{1}$ arising from the $\pi$-adic filtration on $R$ and the $I$-adic filtration on $R_{1}$ are the same. 
\end{prop}
\begin{proof}
Let $D_{v_{1}},\cdots, D_{v_{r}}$ be a choice of Dunkl-Opdam operators associated to our basis of $\mathscr{T}_{\mathscr{A}_{\omega}}$. Its images in $\gr_{F}R$ are the elements $v_{1},\cdots, v_{r}$, and hence they generate $G\ltimes \operatorname{Sym}_{A^{\circ}}(\mathscr{T}_{\mathscr{A}_{\omega/t}})$ over $(\gr_{F}R)_{0}=G\ltimes A^{\circ}$. Thus, we have:
\begin{equation*}
    \pi\in I \textnormal{ and } \pi D_{v_{j}}\in I \textnormal{ for all } j\geq 1.
\end{equation*}
Furthermore, we know that the ordered products of standard Dunkl-Opdam operators form a basis of $R$ over $G\ltimes A^{\circ}$. Hence, as elements in $\mathcal{R}$ commute with Dunkl-Opdam operators, we know that every element in $\pi^{i}F_{i}R$ can be expressed uniquely as a linear combination with coefficients in $G\ltimes A^{\circ}$ of elements of the form:
\begin{equation*}
    \pi^{i}D_{v_{1}}^{\alpha_{1}}\cdots  D_{v_{r}}^{\alpha_{r}}=\pi^{\alpha_{0}}(\pi D_{v_{1}})^{\alpha_{1}}\cdots (\pi D_{v_{r}})^{\alpha_{r}},
\end{equation*}
for some natural numbers $\alpha_{j}$ with $0\leq j\leq r$ and $\sum_{j=0}^{r}\alpha_{j}=i$. Notice that as both $\pi$ and all elements of the form $\pi D_{v_{j}}$ are elements of $I$ for $1\leq j\leq r$, it follows that:
\begin{equation*}
    \pi^{\alpha_{0}}(\pi D_{v_{1}})^{\alpha_{1}}\cdots (\pi D_{v_{r}})^{\alpha_{r}},
\end{equation*}
is an element in $I^{i}$. Therefore, by Lemma \ref{Lemma ardakow 6.4}, for every $i\geq 1$ we have:
\begin{equation*}
    R_{1}\cap \pi^{i}R=\sum_{j\geq i}\pi^{j}F_{j}R\subset I^{i}\subset R_{1}\cap \pi^{i}R.
\end{equation*}
Therefore we have  $I^i\subset R_{1}\cap \pi^{i}R$ for each $i\geq 1$, as we wanted.
\end{proof}
Finally, we have the following:
\begin{teo}\label{teo complete Cher are F-S}
    Let $X=\Sp(A)$ be a  $G$-variety with an étale map $X\rightarrow \mathbb{A}^r_K$. For every Cherednik lattice $\mathscr{A}_{\omega/t}\subset \mathcal{A}_{\omega/t}(X)$, the identity:
    \begin{equation*}
        \mathcal{H}_{t,c,\omega}(X,G)=\varprojlim_{m}\mathcal{H}_{t,c,\omega}(X,G)_{\pi^m\mathscr{A}_{\omega/t}},
    \end{equation*}
is a two-sided Fréchet-Stein presentation.
\end{teo}
\begin{proof}
The proof of \cite[Theorem 6.6]{ardakov2019} carries over to this setting, with \cite[Lemma 6.4]{ardakov2019} being replaced by Proposition \ref{prop preparation  big flatness Theorem}, and \cite[Proposition 6.5]{ardakov2019} being replaced by Proposition \ref{subspace topology}.  
\end{proof}
In particular, there is a category of co-admissible  $\mathcal{H}_{t,c,\omega}(X,G)$-modules. This category will be discussed below. As preparation, we will  now 
describe the extension of scalars along the map of $K$-algebras $H_{t,c,\omega}(X,G)\rightarrow\mathcal{H}_{t,c,\omega}(X,G)$:
\begin{prop}\label{prop flatness of map to the frechet enveloping algebra}

The morphism of $K$-algebras:
\begin{equation*}
    i: H_{t,c,\omega}(X,G)\rightarrow\mathcal{H}_{t,c,\omega}(X,G),
\end{equation*}
is left and right faithfully flat with dense image.
\end{prop}
\begin{proof}
Let
    $B=H_{t,c,\omega}(X,G)$, $C=\mathcal{H}_{t,c,\omega}(X,G)$, and $C_m=\mathcal{H}_{t,c,\omega}(X,G)_{\pi^m\mathscr{A}_{\omega/t}}$. As $C=\varprojlim_m C_m$ is a Fréchet-Stein presentation \cite[Corollary 3.4]{schneider2002algebras} applies, so it 
 suffices to show that  $B\rightarrow C_m$ is faithfully flat for $m\geq 0$. As
the proof given in \cite[Theorem 3.1]{Ardakov_Bode_Wadsley_2021}, does not require $\gr_FB$ to be commutative, it goes through.
\end{proof}
\subsection{The sheaf of p-adic Cherednik algebras}\label{sheaf of p-adic}
Let $X$ be a smooth $G$-variety. In this section we will construct sheaves of $p$-adic Cherednik algebras  on $X/G_{\textnormal{ét}}$, and calculate their sheaf cohomology. As usual, we will start by the local case:
\begin{obs}
Let $X=\Sp(A)$ be $G$-variety with an étale map $X\rightarrow\mathbb{A}^r_K$. We fix  a choice of $t\in K^*$, $c\in \textnormal{Ref}(X,G)$, and $\omega\in\Omega_{X/K}^{2,cl}(X)^G$.
\end{obs}
We start by constructing a presheaf of $p$-adic Cherednik algebras on $X/G_{\textnormal{ét,aff}}$: 
\begin{Lemma}\label{lemma restriction maps sheaf of complete CHer alg}
Let $U_1\rightarrow U_2$ be a morphism in $X/G_{\textnormal{ét,aff}}$, and $V_i=U_i\times_{X/G}X$ for $i\in \{ 1,2\}$. There is a unique continuous morphism of $K$-algebras:
\begin{equation*}
     \mathcal{H}_{t,c,\omega}(V_2,G)\rightarrow  \mathcal{H}_{t,c,\omega}(V_1,G),
\end{equation*}
which extends the morphism $H_{t,c,\omega}(V_2,G)\rightarrow  H_{t,c,\omega}(V_1,G)$.
\end{Lemma}
\begin{proof}
Our definition of the $V_i$ implies that there is an étale and $G$-equivariant map $\overline{V_1}\rightarrow \overline{V_2}$. Thus, we have a commutative diagram of $K$-algebras:
\begin{equation*}
\begin{tikzcd}
	{G\ltimes \wideparen{\D}_{\omega/t}(\overline{V_2})} & {G\ltimes \wideparen{\D}_{\omega/t}(\overline{V_1})} \\
	{H_{t,c,\omega}(V_2,G) } & {H_{t,c,\omega}(V_1,G)}
	\arrow[from=1-1, to=1-2]
	\arrow[from=2-1, to=1-1]
	\arrow[from=2-1, to=2-2]
	\arrow[from=2-2, to=1-2]
\end{tikzcd}
\end{equation*}
where the upper horizontal map is continuous. By Remark \ref{scalateing the parameters of cherednik algebras}, completing the lower map with respect to the subspace topologies yields the desired map.
\end{proof}
We may use these maps to define a presheaf of completed Cherednik algebras:
\begin{defi}
We define the presheaf $\mathcal{H}_{t,c,\omega,X,G}$ on  $X/G_{\textnormal{ét,aff}}$ by setting:
\begin{equation*}
    \mathcal{H}_{t,c,\omega,X,G}(U)=\mathcal{H}_{t,c,\omega}(U\times_{X/G}X,G),
\end{equation*}
and restriction maps obtained by continuous extension of the ones in $H_{t,c,\omega,X,G}$.
\end{defi}
Our goal now is to show that $\mathcal{H}_{t,c,\omega,X,G}$ is a sheaf on  $X/G_{\textnormal{ét,aff}}$.
\begin{defi}
We define the following properties:
    \begin{enumerate}[label=(\roman*)]
    \item  A $\mathcal{R}$-module $M$ is of bounded $\pi$-torsion if $\pi^n M=0$ for some $n\geq 0$.
    \item A chain complex of $\mathcal{R}$-modules $\mathcal{C}^{\bullet}$ is called of bounded $\pi$-torsion cohomology if all its cohomology groups are of bounded $\pi$-torsion.
    \end{enumerate}
\end{defi}
The relevance these notions stems from the following results:
\begin{Lemma}[{\cite[Lemma 3.6]{ardakov2019}}]\label{Lemma AW19 3.6}
Let $\mathcal{C}^{\bullet}$ be a complex of flat $\mathcal{R}$-modules with bounded $\pi$-torsion cohomology. Then for each $n\in\mathbb{Z}$ we have: $H^n(\widehat{\mathcal{C}}^{\bullet})=H^n(\mathcal{C}^{\bullet})$.    
\end{Lemma}
\begin{Lemma}\label{Lemma approximation by mod pi quotients}
Let $\{U_i\}_{i=1}^n$ be an étale affinoid cover of $X$ . Then the complex:
\begin{equation*}
   \mathcal{C}^{\bullet}:= \left(0\rightarrow \OX_X(X)^{\circ}\rightarrow \prod_{1\leq i\leq n} \OX_{U_i}(U_i)^{\circ}\rightarrow \prod_{1\leq i,j\leq n} \OX_{U_i\cap U_j}(U_i\times_X U_j)^{\circ}\rightarrow \cdots \right),
\end{equation*}
is a complex of flat $\mathcal{R}$-modules with bounded $\pi$-torsion cohomology.
\end{Lemma}
\begin{proof}
By \cite[Corollary 3.2.6]{de1996etale}, we have an exact complex of Banach $K$-algebras:
\begin{equation*}  \mathcal{D}^{\bullet}:=\left(0\rightarrow \OX_X(X)\rightarrow \prod_{1\leq i\leq n} \OX_{U_i}(U_i)\rightarrow \prod_{1\leq i,j\leq n} \OX_{U_i\cap U_j}(U_i\times_X U_j)\rightarrow \cdots \right).
\end{equation*}
Then one can proceed as in the proof of \cite[Proposition 6.10.(iii)]{perf}.
\end{proof}
We will also need the following two lemmas:
\begin{Lemma}\label{Lemma affinoid subdomain containing shilov boundary}
Let  $\{U_i\}_{i=1}^{m}\in X/G_{\textnormal{ét,aff}}$, and define $V_i=U_i\times_{X/G}X\rightarrow X$. There is a $G$-invariant affinoid subdomain $W\subset \overline{X}$ such that for each  $1\leq i\leq m$ we have:
\begin{enumerate}[label=(\roman*)]
    \item $W$ contains the Shilov boundary of $X$.
    \item $W_i=W\times_{X}V_i\subset \overline{V}_i$ contains the Shilov boundary of $V_{i}$.
\end{enumerate}
\end{Lemma}
\begin{proof}
Choose an increasing  admissible cover $\overline{X}=\cup_n\overline{X}_n$ by $G$-invariant affinoid subdomains. By Lemmas \ref{Lemma 1 sheaf of cher algebras } and \ref{Lemma 2 sheaf of cher algebras}, we have
$\overline{V_i}=V_i\times_{X}\overline{X}$ for each $i$. Thus, we have an admissible affinoid cover $\overline{V_i}=\cup_n \left(V_i\times_{X}\overline{X}_n\right)$. By Proposition \ref{restriction rings are trivial},  for each $i$ there is some natural number $n_i$ such that $V\times_{X}\overline{X}_{n_i}$ contains the Shilov boundary of $V_i$. Again by the same Corollary, we make take some natural number $n_0$ such that  such that $\overline{X}_{n_0}$ contains the Shilov boundary of $X$. It follows that,  for  sufficiently high  $n$, the affinoid space 
$\overline{X}_{n}$ satisfies the properties we want. 
\end{proof}
\begin{Lemma}\label{Lemma quasi-coherentness at the integral level}
    Let $U\in  X/G_{\textnormal{ét,aff}}$  and set $V=\Sp(B)=U\times_{X/G}X$. Choose an affinoid subdomain $W=\Sp(C)\subset \overline{X}$ as in Lemma \ref{Lemma affinoid subdomain containing shilov boundary}, and let:
    \begin{equation*}
        Y=\Sp(D)=W\times_XV.
    \end{equation*}
    Choose a $W$-Cherednik lattice $\mathscr{A}_{\omega/t}$ in $X$ such that $\mathscr{B}_{\omega/t}=B^{\circ}\otimes_{A^{\circ}}\mathscr{A}_{\omega/t}$ is a $Y$-Cherednik lattice in $V$. For each $m\geq 0$ there is an identification of $B^{\circ}$-modules:
    \begin{equation*}
       \psi:B^{\circ}\otimes_{A^{\circ}}H_{t,c,\omega}(X,G)_{\pi^m\mathscr{A}_{\omega/t}}\cong H_{t,c,\omega}(V,G)_{\pi^m\mathscr{B}_{\omega/t}}.
    \end{equation*}
\end{Lemma}
\begin{proof}
Assume for simplicity that $t=1$. By Proposition \ref{prop restriction maps presheaf of cherednik algebras},  the map of filtered $K$-algebras $\varphi:G\ltimes \mathcal{D}_{\omega}(W)\rightarrow G\ltimes \mathcal{D}_{\omega}(Y)$
restricts to a map of filtered $K$-algebras $H_{1,c,\omega}(X,G)\rightarrow H_{1,c,\omega}(V,G)$. Let  $\mathscr{C}_{\omega}=C^{\circ}\otimes_{A^{\circ}}\mathscr{A}_{\omega/t}$ and $\mathscr{D}_{\omega}=D^{\circ}\otimes_{A^{\circ}}\mathscr{A}_{\omega/t}$. Our choice of $\mathscr{A}_{\omega}$ implies that for each $m\geq 0$, the map $\varphi$ restricts to a morphism:
\begin{equation*}
    \varphi:G\ltimes \mathcal{D}(\pi^m\mathscr{C}_{\omega})\rightarrow G\ltimes \mathcal{D}(\pi^m\mathscr{D}_{\omega}).
\end{equation*}
Hence, for each $m\geq 0$ we have maps of filtered $\mathcal{R}$-algebras:
\begin{equation*}\label{equation Lemma sheaf of p-adic CHer}
    H_{1,c,\omega}(X,G)_{\pi^m\mathscr{A}_{\omega}}\rightarrow H_{1,c,\omega}(V,G)_{\pi^m\mathscr{B}_{\omega}}.
\end{equation*}
By Proposition \ref{pbw uncompleted microlocal level}, taking the associated graded yields the following map:
\begin{equation*}
    G\ltimes \operatorname{Sym}_{A^{\circ}}(\mathscr{T}_{\mathscr{A}_{\omega/t}})\rightarrow  G\ltimes \operatorname{Sym}_{B^{\circ}}(\mathscr{T}_{\mathscr{B}_{\omega/t}})
\end{equation*}
so we may proceed as in the proof of Corollary \ref{clasical local version cherednik algebras are sheaves} to obtain the result.
\end{proof}
We are now in position to show the following:
\begin{teo}\label{completed Cherednik algebras are a sheaf}
$\mathcal{H}_{t,c,\omega,X,G}$ is a sheaf of Fréchet algebras on $X/G_{\textnormal{ét,aff}}$ with trivial higher \v{C}ech cohomology groups.
\end{teo}
\begin{proof}
Let $\{U_i\}_{i=1}^r$ be an étale affinoid cover of $X/G$, and let $V_i=U_i\times_{X/G} X$ for each $1\leq i\leq m$. We need to show exactness of  the following complex $\mathcal{C}^{\bullet}$:
\begin{multline*}\label{equation proof of cher algebras are a sheave}    0\rightarrow\mathcal{H}_{1,c,\omega}(X,G)\rightarrow \displaystyle\prod_{1\leq i\leq r}\mathcal{H}_{1,c,\omega}(V_{i},G)\rightarrow  \prod_{1\leq i,j\leq r}\mathcal{H}_{1,c,\omega}(V_i\times_{X} V_j,G)\rightarrow \cdots.
\end{multline*}
Choose $n\geq 1$, and consider the following family:
\begin{equation*}
    F_{n+1}:=\{ V_{i_1\cdots i_s}= V_{i_1}\times_X\cdots \times_X V_{i_s}, \textnormal{ such that } 1\leq i_1\leq \cdots\leq i_s\leq r \textnormal{, } s\leq n+1 \}.
\end{equation*}
Notice that $F_{n+1}$
is  finite, so we may choose an affinoid subdomain $W=\Sp(B)\subset X$ satisfying the conditions of Lemma \ref{Lemma affinoid subdomain containing shilov boundary} for $F_{n+1}$. Let $W_{i_1,\cdot,i_{s}}=W\times_{X}V_{i_1\cdots i_s}$ for all  $V_{i_1\cdots i_s}\in F_{n+1}$, and consider the complex $F_{n+1}\mathcal{C}^\bullet$:
\begin{equation*}    0\rightarrow\mathcal{H}_{1,c,\omega}(X,G)_W\rightarrow \cdots\rightarrow  \prod_{1\leq i_1\leq\cdots\leq i_{n+1} \leq r}\mathcal{H}_{1,c,\omega}(V_{i_1\cdots i_{n+1}},G)_{W_{i_1\cdots i_{n+1}}}\rightarrow 0.
\end{equation*}
The choice of $W$, together with Theorem \ref{teo simplification of definition of completed cherednik algebras}, imply that $F_{n+1}\mathcal{C}^i=\mathcal{C}^i$ for all $i\leq n+1$. In particular, $\textnormal{H}^i(\mathcal{C}^{\bullet})=\textnormal{H}^i(F_{n+1}\mathcal{C}^{\bullet})$ for all $i\leq n$. As we can choose $n$ to be arbitrarily large, it suffices to show that $\textnormal{H}^i(F_{n+1}\mathcal{C}^{\bullet})=0$ for all $i\leq n$.\\
Fix some $n\geq 0$ and let $V_{i_1\cdots i_s}=\Sp(A_{i_1\cdots i_s})$, and $W_{i_1\cdots i_s}=\Sp(B_{i_1\cdots i_s})$. Using finiteness of $F_{n+1}$ again, we can choose a $W$-Cherednik lattice $\mathscr{A}_{\omega}$ such that  $\mathscr{A}_{i_1\cdots i_s,\omega}=A^{\circ}_{i_1\cdots i_s}\otimes_{A^{\circ}}\mathscr{A}_{\omega}$  is a $W_{i_1\cdots i_s}$-Cherednik lattice in $V_{i_1\cdots i_s}$ for all $V_{i_1\cdots i_s}\in F_{n+1}$. For each $m\geq 1$, let $\mathcal{D}^{\bullet}_{n,m}$ be the following chain complex:
 \begin{multline*}\label{equation 2 sheaf proof}
   0\rightarrow\mathcal{H}_{1,c,\omega}(X,G)_{\pi^m\mathscr{A}_{\omega}}\rightarrow \cdots\\\rightarrow \prod_{1\leq i_1\leq\cdots\leq i_{n+1} \leq r}\mathcal{H}_{1,c,\omega}(V_{i_1,\cdots,i_{n+1}},G)_{\pi^m\mathscr{A}_{i_1\cdots i_{n+1},\omega}}\rightarrow 0.
 \end{multline*}
 Notice that $F_{n+1}\mathcal{C}^{\bullet}=\varprojlim_m \mathcal{D}^{\bullet}_{n,m}$. Furthermore, for each $m\geq 0$ the objects of the complex $\mathcal{D}^{\bullet}_{n,m}$ are Banach $K$-algebras, and for any $r\in\mathbb{Z}$ the map $\mathcal{D}^{r}_{n,m+1}\rightarrow \mathcal{D}^{r}_{n,m}$  has dense image. Thus, by \cite[Remark 13.2.4]{grothendieck1961elements} the inverse system $\left(\D^{\bullet}_{n,m} \right)_{n\geq 0}$, satisfies the Mittag-Leffler condition. Hence, by \cite[Theorem 3.5.8]{Weibel1994introduction}, for each $i\geq 0$ we have the following short exact sequence:
 \begin{equation*}
   0\rightarrow  R^1\varprojlim_m \textnormal{H}^{i-1}(\mathcal{D}^{\bullet}_{n,m})\rightarrow \textnormal{H}^i(F_{n+1}\mathcal{C}^{\bullet})\rightarrow \varprojlim_m \textnormal{H}^i(\mathcal{D}^{\bullet}_{n,m})\rightarrow 0.
 \end{equation*}
 Therefore, it suffices to show that $\textnormal{H}^i(\mathcal{D}^{\bullet}_{n,m})=0$ for $i\leq n$ and $m\geq 0$. Fix some non-negative integer $m\geq 0$. By Lemma \ref{Lemma approximation by mod pi quotients}, the chain complex:
\begin{equation}\label{equation final reduction  sheaf Theorem}
    0\rightarrow A^{\circ}\rightarrow \prod_{1\leq i\leq r} A^{\circ}_i\rightarrow \prod_{1\leq i,j\leq r} A^{\circ}_{i,j}\rightarrow\cdots,
\end{equation}
is a complex of flat $\mathcal{R}$-modules with bounded $\pi$-torsion cohomology. By Corollary \ref{coro basis of Cherednik algebras at the tintegral level}, the algebra $H_{1,c,\omega}(X,G)_{\pi^m\mathscr{A}_{\omega}}$
is free as an $A^{\circ}$-module. In particular, 
applying $-\otimes_{A^{\circ}}H_{1,c,\omega}(X,G)_{\pi^m\mathscr{A}_{\omega}}$ to $(\ref{equation final reduction  sheaf Theorem})$ yields a chain complex  of flat $\mathcal{R}$-modules with bounded $\pi$-torsion cohomology. Denote this complex by $\mathcal{E}^{\bullet}$, and let $\widehat{\mathcal{E}}_K^{\bullet}$ be the complex obtained after taking the $\pi$-adic completion of $\mathcal{E}^{\bullet}$, and then tensoring by $K$. It follows from Lemma \ref{Lemma AW19 3.6} that $\widehat{\mathcal{E}}_K^{\bullet}$ is an exact complex.
Notice that by Lemma \ref{Lemma quasi-coherentness at the integral level}, our choice of $\mathscr{A}_{\omega}$ implies that $\mathcal{D}_{n,m}^i=\widehat{\mathcal{E}}_K^{i}$
for all $i\leq n+1$. In particular, we have $\textnormal{H}^i(\mathcal{D}_{n,m}^{\bullet})=\textnormal{H}^i(\widehat{\mathcal{E}}_K^{\bullet})=0$ for all $i\leq n$, as we wanted to show.
\end{proof}
We may use this theorem to define sheaves of complete Cherednik algebras for more general $G$-varieties. In particular, we have the following:
\begin{teo}\label{teo sheaves of complete Cherednik algebras for general $G$-varieties}
Let $X$ be a smooth $G$-variety. Choose $t\in K^*$, $c\in \textnormal{Ref}(X,G)$ and $\omega\in\mathbb{H}^1(X,\Omega_{X/K}^{\geq 1})^{G}$. There is a unique sheaf of complete topological $K$-algebras $\mathcal{H}_{t,c,\omega,X,G}$ on $X/G_{\textnormal{ét}}$ satisfying the following: For each affinoid space $U\subset X/G_{\textnormal{ét}}$, such that there is an étale map $V=U\times_{X/G}X\rightarrow \mathbb{A}^n_K$ we have:
\begin{equation*}
    \mathcal{H}_{t,c,\omega,X,G}(U)=\mathcal{H}_{t,c,\omega}(V,G).
\end{equation*}
We call $\mathcal{H}_{t,c,\omega,X,G}$ a sheaf of $p$-adic Cherednik algebras on $X$.
\end{teo}
\begin{proof}
If $X=\Sp(A)$ is affinoid with an étale map $X\rightarrow\mathbb{A}^r_K$, then $\mathcal{H}_{t,c,\omega,X,G}$ is a sheaf on $X/G_{\textnormal{ét,Aff}}$
by Theorem \ref{completed Cherednik algebras are a sheaf}. As every étale cover may be refined to an étale affinoid cover, it follows that $\mathcal{H}_{t,c,\omega,X,G}$ extends uniquely to a sheaf on $X/G_{\textnormal{ét}}$. In the general case, as $X$ satisfies $(\operatorname{G-Aff})$, we may apply Proposition \ref{G-invariant coverings} to construct $\mathcal{H}_{t,c,\omega,X,G}$ on a cover of $X$. As $p$-adic Cherednik algebras only depend on the choice of the parameters $t,c,$ and $\omega$, it follows that $\mathcal{H}_{t,c,\omega,X,G}$ extends to a sheaf of topological $K$-algebras on $X/G_{\textnormal{ét}}$. 
\end{proof}
 \begin{obs}
 If $c\in \textnormal{Ref}(X,G)$ is the zero section, then  $\mathcal{H}_{t,c,\omega,X,G}=G\ltimes \wideparen{\mathcal{D}}_{\omega}$.
 \end{obs}
If $S(X,G)$ is empty, then $\mathcal{H}_{t,c,\omega,X,G}=G\ltimes \wideparen{\mathcal{D}}_{\omega/t}$ for every $t\in K^*$,  $\omega\in \mathbb{H}^1(X,\Omega_{X/K}^{\geq 1})^{G}$. In particular, if $X^{\operatorname{reg}}:=X\setminus X^G$, then $X^{\operatorname{reg}}/G$ is smooth, and 
the restriction of $\mathcal{H}_{t,c,\omega,X,G}$ to $X^{\operatorname{reg}}/G_{\textnormal{ét}}$ is canonically isomorphic to $G\ltimes \wideparen{\D}_{\omega/t}$.
\begin{obs}
When no confusion is possible, we will usually drop the $X$ and $G$ from the notation, and simply write $\mathcal{H}_{t,c,\omega}:=\mathcal{H}_{t,c,\omega,X,G}$.    
\end{obs}
\subsection{c-flatness of transition maps}\label{section flatness of transition maps}
For some of the applications (\emph{cf}. \cite{p-adicCatO}), it will be convenient to have a better understanding of the transition maps of $\mathcal{H}_{t,c,\omega}$.  Let $X=\Sp(A)$ be an affinoid $G$-variety with an étale map $X\rightarrow \mathbb{A}^r_K$, and let $\overline{X}$ be the complement of the union of all reflection hypersurfaces in $X$.  In this section, we will show that for any $G$-invariant affinoid subdomain $U=\Sp(B)\subset X$, the canonical map $\mathcal{H}_{t,c,\omega}(X,G)\rightarrow \mathcal{H}_{t,c,\omega}(U,G)$ is $c$-flat. In particular, let $\mathscr{A}_{\omega/t}$ be a Cherednik lattice of $\mathcal{A}_{\omega/t}(X)$ such that $\mathscr{B}_{\omega/t}=B^{\circ}\otimes_{A^{\circ}}\mathscr{A}_{\omega/t}$ is a Cherednik lattice of $\mathcal{A}_{\omega/t}(U)$. We will show that for big enough $n\geq 0$ the map:
\begin{equation*}
    \mathcal{H}_{t,c,\omega}(X,G)_{\pi^n\mathscr{A}_{\omega/t}}\rightarrow \mathcal{H}_{t,c,\omega}(U,G)_{\pi^n\mathscr{B}_{\omega/t}}
\end{equation*}
is a flat morphism of Banach $K$-algebras. We will only show this for right flatness, as the general case is analogous. This section is based upon ideas in  \cite[Section 4]{ardakov2019}.\\
Let $\mathbb{B}^1_K=\Sp(K\langle x\rangle)$ be the closed unit ball, and consider the smooth projection:
\begin{equation*}
    p:X\times \mathbb{B}^1_K\rightarrow X.
\end{equation*} 
The action of $G$ on $X$ induces an action on $X\times \mathbb{B}^1_K$ by acting trivially on $\mathbb{B}^1_K$. The projection $p:X\times \mathbb{B}^1_K\rightarrow X$ is $G$-equivariant, and applying the arguments from Lemmas \ref{Lemma 1 sheaf of cher algebras }, and \ref{Lemma 2 sheaf of cher algebras}, it follows that $p^{-1}(\overline{X})=\overline{X}\times \mathbb{B}^1_K= \overline{X\times \mathbb{B}^1_K}$.\\ 
Let $\mathscr{T}_{\mathscr{A}_{\omega}}$ be the image of $\mathscr{A}_{\omega}$
under the anchor map $\mathcal{A}_{\omega}(X)\rightarrow \mathcal{T}_{X/K}(X)$. Choose $f\in A^G$ such that for all $v\in \mathscr{T}_{\mathscr{A}_{\omega}}$ we have $v(f)\in A^{\circ}$, and choose $n\geq 0$ such that $\pi^nf\in A^{\circ}$. We define the following elements of $A^{\circ}\langle x\rangle$:
\begin{equation*}
  u_1=\pi^nx-\pi^nf, \quad u_2=\pi^nxf-\pi^n.
\end{equation*}
We will assume first that $U$ is an affinoid subdomain of the following forms:
\begin{equation}\label{equation c-flatness possible forms for U}
   U_1:= X\left(\frac{1}{f} \right), \textnormal{ } U_2:= X\left(f \right).
\end{equation}
Thus, the rigid functions on $U$ are of the form $B=A\langle x\rangle/u_iA\langle x\rangle$, for $i=1,2$. As $f\in A^G$, the closed immersion $U\rightarrow X\times \mathbb{B}^1_K$ is $G$-equivariant. By definition, $\mathscr{A}_{\omega}$ is a Cherednik lattice of $\mathcal{A}_{\omega}(X)$. Hence, there is some $G$-invariant affinoid subdomain $V=\Sp(C)\subset X$ containing the Shilov boundary of $X$ such that $\mathscr{A}_{\omega}$ is a $V$-Cherednik lattice of $\mathcal{A}_{\omega}(X)$.
By Lemma \ref{Lemma affinoid subdomain containing shilov boundary}, we may assume that $W=\Sp(D)=U\cap V$ contains the Shilov boundary of $U$, and by Lemma \ref{Lemma scalating the lattices}, we can assume that $\mathscr{B}_{\omega}:=B^{\circ}\otimes_{A^{\circ}}\mathscr{A}_{\omega}$ is a $W$-Cherednik lattice of $\mathcal{A}_{\omega}(U)$. Hence, we arrive at the following commutative diagram of $G$-equivariant maps:
\begin{center}
    % https://tikzcd.yichuanshen.de/#N4Igdg9gJgpgziAXAbVABwnAlgFyxMJZAJgBoBGAXVJADcBDAGwFcYkQANEAX1PU1z5CKMgAZqdJq3YA1HnxAZseAkVEUJDFm0QgAqvP7Kha0uJpbpugOqHFAlcOTkNFqTs4AdT3gC28AAJvX3ocAAsAIwjgACFuAD1yAH0AaTslQVUUF3NJbVlvP0Dg0MjouMTUngkYKABzeCJQADMAJwhfJBcQHAgkUV4W9s7EAFYaXq7BkDaOpAAWCb7EAYVZkbIe5fnp9aQAZiWkYl3hg6Ox07mxi53KbiA
\begin{tikzcd}
W \arrow[d] \arrow[r] & V\times \mathbb{B}^1_K \arrow[r] \arrow[d] & V \arrow[d] \\
U \arrow[r]           & X\times \mathbb{B}^1_K \arrow[r]           & X          
\end{tikzcd}
\end{center}
where all squares are cartesian, and the vertical maps are inclusions of affinoid subdomains which contain the Shilov boundary of the codomain.\\
As $\mathscr{A}_{\omega}$ is a $V$-Cherednik lattice of $\mathcal{A}_{\omega}(X)$, it follows that $\mathscr{C}_{\omega}=C^{\circ}\otimes_{A^{\circ}}\mathscr{A}_{\omega}$ is a free $(\mathcal{R},C^{\circ})$-Lie lattice of $\mathcal{A}_{\omega}(V)$. Denote its anchor map by $\tau:\mathscr{C}_{\omega}\rightarrow \Der_{\mathcal{R}}(C^{\circ})$. By \cite[Proposition 4.2]{ardakov2019}, the formulas:
\begin{equation}\label{equation extension of anchor map}
    \tau_1(v)(x)= \tau(v)(f) \textnormal{, }  \tau_2(v)(x)=-x^2(\tau(v)(f)) \textnormal{, for } v\in \mathscr{C}_{\omega}
\end{equation}
define unique extensions of $\tau$ to  $C^{\circ}$-linear maps  of $\mathcal{R}$-Lie algebras:
\begin{equation*}
    \tau_i:\mathscr{C}_{\omega}:=C^{\circ}\otimes_{A^{\circ}}\mathscr{A}_{\omega}\rightarrow \Der_{\mathcal{R}}(C^{\circ}\langle x\rangle).
\end{equation*}
These extensions induce $(\mathcal{R},C^{\circ}\langle x\rangle)$-Lie algebra structures on $C^{\circ}\langle x\rangle\otimes_{C^{\circ}}\mathscr{C}_{\omega}$ extending the bracket and anchor map of $\mathscr{C}_{\omega}$. We denote these Lie-Rinehart algebras by $\mathscr{C}_{\omega}^i$ for $i=1,2$ respectively. By \cite[Proposition 4.3]{ardakov2019}, the Lie-Rinehart algebras $\mathscr{C}_{\omega}^i$ satisfy the following properties:
\begin{enumerate}
\item $U(\mathscr{C}_{\omega}^i)\cong C^{\circ}\langle x\rangle \otimes_{C^{\circ}}U(\mathscr{C}_{\omega})$ as  $C^{\circ}\langle x\rangle$-modules.
    \item $\widehat{U}(\mathscr{C}_{\omega}^i)$ is a flat $\widehat{U}(\mathscr{C}_{\omega})$-module on both sides.
    \item Let $\mathscr{D}_{\omega}=D^{\circ}\otimes_{A^{\circ}}\mathscr{A}_{\omega}$.     There is a short exact sequence:
    \begin{equation}\label{equation short exact sequence c-flatness}
       0\rightarrow \widehat{U}(\mathscr{C}_{\omega}^i)_K\xrightarrow[]{u_i\cdot }\widehat{U}(\mathscr{C}_{\omega}^i)_K \rightarrow \widehat{U}(\mathscr{D}_{\omega})_K\rightarrow 0.
    \end{equation}
\end{enumerate}
We need to show that these properties also hold at the level of completed skew group algebras of twisted differential operators. Recall from Definition \ref{defi tdo} that we define $\mathcal{I}$  to be the two-sided ideal in $\widehat{U}(\mathscr{C}_{\omega})_K$ generated by:
\begin{equation*}
    (1_{\widehat{U}(\mathscr{C}_{\omega})_K}-1_{\mathscr{C}_{\omega}})\in \widehat{U}(\mathscr{C}_{\omega})_K.
\end{equation*}
First, notice that by construction of the completions of the algebras of twisted differential operators given in equation (\ref{equation FS presentation of completed TDO}), we have a canonical identification:
\begin{equation*}
    \widehat{U}(\mathscr{D}_{\omega})_K\otimes_{\widehat{U}(\mathscr{C}_{\omega})_K}\widehat{U}(\mathscr{C}_{\omega})_K/\mathcal{I}\cong \widehat{U}(\mathscr{D}_{\omega})_K/\mathcal{I}\cong \widehat{\D}(\mathscr{D}_{\omega})_K.
\end{equation*}
On the other hand, by \cite[Theorem 4.5]{ardakov2019}, the map $\widehat{U}(\mathscr{C}_{\omega})_K\rightarrow \widehat{U}(\mathscr{D}_{\omega})_K$ is two-sided flat. Therefore, applying  $-\otimes_{\widehat{U}(\mathscr{C}_{\omega})_K}\widehat{U}(\mathscr{C}_{\omega})_K/\mathcal{I}$ to (\ref{equation short exact sequence c-flatness}),  yields a short exact sequence of right $\widehat{U}(\mathscr{C}_{\omega}^i)_K/\mathcal{I}$-modules:
\begin{equation}\label{equation 2 short exact sequence c-flatness}
    0\rightarrow \widehat{U}(\mathscr{C}_{\omega}^i)_K/\mathcal{I}\xrightarrow[]{u_i\cdot }\widehat{U}(\mathscr{C}_{\omega}^i)_K/\mathcal{I} \rightarrow \widehat{\D}(\mathscr{D}_{\omega})_K\rightarrow 0.
\end{equation}
By construction, the action of $G$ on $C\langle x\rangle$ leaves $x$ fixed. Furthermore, as $f\in A^G$ the element $u_i$
is $G$-invariant for $i=1,2$. It follows that the action of Banach $K$-algebras of $G$ on $\widehat{U}(\mathscr{C}_{\omega})_K$ extends to an action of Banach $K$-algebras on $\widehat{U}(\mathscr{C}_{\omega}^i)_K$. Furthermore, by Lemma \ref{extension of the action of G to differential forms}, we know that $(1_{\widehat{U}(\mathscr{C}_{\omega})_K}-1_{\mathscr{C}_{\omega}})$ is also fixed by $G$. Hence, we get an action of $G$ by automorphisms of Banach $K$-algebras on  $\widehat{U}(\mathscr{C}_{\omega}^i)_K/\mathcal{I}$ for $i=1,2$. As the $u_i$ are fixed by $G$, the morphisms in (\ref{equation 2 short exact sequence c-flatness}) are $G$-equivariant, so we get a short exact sequence of $G\ltimes \widehat{U}(\mathscr{C}_{\omega}^i)_K/\mathcal{I}$-modules:
\begin{equation*}
    0\rightarrow G\ltimes\widehat{U}(\mathscr{C}_{\omega}^i)_K/\mathcal{I}\xrightarrow[]{u_i\cdot }G\ltimes\widehat{U}(\mathscr{C}_{\omega}^i)_K/\mathcal{I} \rightarrow G\ltimes\widehat{\D}(\mathscr{D}_{\omega})_K\rightarrow 0.
\end{equation*}
We can now introduce the following auxiliary object:
\begin{defi}
Let $\mathcal{H}^i$ be the $A^{\circ}$-submodule of $G\ltimes\widehat{U}(\mathscr{C}_{\omega}^i)_K/\mathcal{I}$ given by:
\begin{equation*}
    \mathcal{H}^i\cong A^{\circ}\langle x\rangle\otimes_{A^{\circ}} H_{1,c,\omega}(X,G)_{\mathscr{A}_{\omega}}.
\end{equation*}
\end{defi}
We now investigate the properties of $\mathcal{H}^i$:
\begin{prop}[{\emph{cf. }\cite[Proposition 4.3, Theorem 4.5]{ardakov2019}}]\label{prop two-sided flatness at the microlocal level}
The following hold:
\begin{enumerate}[label=(\roman*)]
    \item $\mathcal{H}^i$ is a $\mathcal{R}$-subalgebra of $G\ltimes\widehat{U}(\mathscr{C}_{\omega}^i)_K/\mathcal{I}$, and the canonical inclusion:
    \begin{equation*}
        H_{1,c,\omega}(X,G)_{\mathscr{A}_{\omega}}\rightarrow \mathcal{H}^i,
    \end{equation*}
   is a morphism of $\mathcal{R}$-algebras.
    \item $\widehat{\mathcal{H}}_K^i$ is two-sided flat as a $\mathcal{H}_{1,c,\omega}(X,G)_{\mathscr{A}_{\omega}}$-module.
    \item There is a short exact sequence of right $\widehat{\mathcal{H}}_K^i$-modules:
    \begin{equation*}
       0\rightarrow \widehat{\mathcal{H}}_K^i\xrightarrow[]{u_i\cdot }\widehat{\mathcal{H}}_K^i \rightarrow \mathcal{H}_{1,c,\omega}(U,G)_{\mathscr{B}_{\omega}}\rightarrow 0.
    \end{equation*}
    \item The map $\mathcal{H}_{1,c,\omega}(X,G)_{\mathscr{A}_{\omega}}\rightarrow\mathcal{H}_{1,c,\omega}(U,G)_{\mathscr{B}_{\omega}}$ is right flat.
\end{enumerate}
\end{prop}
\begin{proof}
As $A^{\circ}\langle x\rangle $ and $H_{1,c,\omega}(X,G)_{\mathscr{A}_{\omega}}$ are subalgebras of $G\ltimes\widehat{U}(\mathscr{C}_{\omega}^i)_K/\mathcal{I}$, it suffices to show that $\mathcal{H}^i$ is closed under the product of $G\ltimes\widehat{U}(\mathscr{C}_{\omega}^i)_K/\mathcal{I}$. Choose a Dunkl-Opdam operator $D_v\in H_{1,c,\omega}(X,G)_{\mathscr{A}_{\omega}}$ and $h=\sum_{n\geq 0}a_nx^{n}\in A^{\circ}\langle x\rangle$. Then, working as in Proposition \ref{initial relations in a Cherednik algebra} we have:
\begin{equation*}
    D_v\cdot h=h\cdot D_{v}+v(h)+\sum_{(Y,g)\in S(X,G)}\frac{2c(Y,g)}{1-\lambda_{Y,g}}\overline{\xi_{Y}}(v)(g(h)-h)g.
\end{equation*}
 Our choice of $f\in A$, together with the formulas in equation (\ref{equation extension of anchor map}) imply that $v(h)\in A^{\circ}\langle x \rangle$. Hence, it suffices to show that for every $(Y,g)\in S(X,G)$ we have:
 \begin{equation*}
   \frac{2c(Y,g)}{1-\lambda_{Y,g}}\overline{\xi_{Y}}(v)(g(h)-h)\in A^{\circ}\langle x \rangle.  
 \end{equation*}
As the action of $G$ on $A^{\circ}$ is continuous and fixes $x$, it is enough to show that:
 \begin{equation*}
     \frac{2c(Y,g)}{1-\lambda_{Y,g}}\overline{\xi_{Y}}(v)(g(a_n)-a_n)\in A^{\circ}, \textnormal{ for each } n\geq 0.
 \end{equation*}
As each $a_n$ is contained in $A^{\circ}$, this follows from Corollary \ref{coro basis of Cherednik algebras at the tintegral level}. Hence, $(i)$ holds. For $(ii)$, it suffices to show that the map between the $\pi$-adic completions:
\begin{equation*}
  \widehat{H}_{1,c,\omega}(X,G)_{\mathscr{A}_{\omega}}\rightarrow \widehat{\mathcal{H}}^i,  
\end{equation*}
is flat. Reducing this map modulo $\pi^n$ we get the following flat morphism:
\begin{equation}\label{equation morphisms in flatness proof}
    H_{1,c,\omega}(X,G)_{\mathscr{A}_{\omega}}/\pi^n\rightarrow (\mathcal{R}/\pi^n)[x]\otimes_{\mathcal{R}/\pi^n}\left(H_{1,c,\omega}(X,G)_{\mathscr{A}_{\omega}}/\pi^n\right),
\end{equation}
As $\widehat{H}_{1,c,\omega}(X,G)_{\mathscr{A}_{\omega}}$ is noetherian, it suffices to show that $\widehat{\mathcal{H}}^i$ is noetherian. By \cite[Tag 05GH]{stacks-project}, this is equivalent to the following quotient being noetherian:
\begin{equation*}
    \widehat{\mathcal{H}}^i/\pi=k[x]\otimes_kH_{1,c,\omega}(X,G)_{\mathscr{A}_{\omega}}/\pi,
\end{equation*}
being noetherian. This follows because the PBW filtration on $H_{1,c,\omega}(X,G)_{\mathscr{A}_{\omega}}/\pi$ induces an exhaustive filtration on 
$\widehat{\mathcal{H}}^i/\pi$, whose associated graded is noetherian (\emph{cf}. Corollary \ref{coro basis of Cherednik algebras at the tintegral level}). Hence, $\mathcal{H}_{1,c,\omega}(X,G)_{\mathscr{A}_{\omega}}\rightarrow\widehat{\mathcal{H}}_K^i$ is a flat morphism of $K$-algebras, as we wanted to show. Statement $(iii)$ follows from Lemma \ref{Lemma quasi-coherentness at the integral level}, as we have:
\begin{align*}
  \widehat{\mathcal{H}}_K^i/u_i=\left(A\langle x\rangle \widehat{\otimes}_A \mathcal{H}_{1,c,\omega}(X,G)_{\mathscr{A}_{\omega}}\right)/u_i=&(A\langle x\rangle /u_i)\widehat{\otimes}_A \mathcal{H}_{1,c,\omega}(X,G)_{\mathscr{A}_{\omega}}\\
  =&B\widehat{\otimes}_A \mathcal{H}_{1,c,\omega}(X,G)_{\mathscr{A}_{\omega}}\\=&\mathcal{H}_{1,c,\omega}(X,G)_{\mathscr{B}_{\omega}}.
\end{align*}
With the properties above available, $(iv)$ is equivalent to the corresponding fact for enveloping algebras of Lie algebroids, which is shown in \cite[Theorem 4.5]{ardakov2019}. 
\end{proof}
\begin{coro}\label{coro c-flatness in the simple case}
Let $X=\Sp(A)$ be a $G$-variety with an étale map $X\rightarrow \mathbb{A}^r_K$, and let $U=X\left(\frac{1}{f}\right)$ or $U=X\left(f\right)$. The map $\mathcal{H}_{1,c,\omega}(X,G)\rightarrow \mathcal{H}_{1,c,\omega}(U,G)$ is $c$-flat.
\end{coro}
\begin{proof}
Let $\mathscr{A}_{\omega}$ be a Cherednik lattice of $\mathcal{A}_{\omega}(X)$ such that $\mathscr{B}_{\omega}=B^{\circ}\otimes_{A^{\circ}}\mathscr{A}_{\omega}$ is a Cherednik lattice of $\mathcal{A}_{\omega}(U)$. Multiplying $\mathscr{A}_{\omega}$ by a power of $\pi$, we may assume that $v(f)\in A^{\circ}$ for all $v\in \mathscr{T}_{\mathscr{A}_{\omega}}$. Thus, the $\mathcal{H}_{1,c,\omega}(X,G)_{\pi^n\mathscr{A}_{\omega}}\rightarrow \mathcal{H}_{1,c,\omega}(X,G)_{\pi^n\mathscr{B}_{\omega}}$ are flat by Proposition \ref{prop two-sided flatness at the microlocal level}. Thus,  $\mathcal{H}_{1,c,\omega}(X,G)\rightarrow \mathcal{H}_{1,c,\omega}(U,G)$ is $c$-flat, as wanted.
\end{proof}
We now extend this to arbitrary $G$-invariant affinoid subdomains:
\begin{prop}\label{prop c-flatness of transition maps}
Let $X=\Sp(A)$ be a $G$-variety with an étale map $X\rightarrow \mathbb{A}^r_K$, and let $U\subset X$ be a $G$-invariant affinoid subdomain. The map:
\begin{equation*}
    \mathcal{H}_{t,c,\omega}(X,G)\rightarrow \mathcal{H}_{t,c,\omega}(U,G)
\end{equation*}
is $c$-flat. If $\{ U_i\}_{i=1}^n$ is an admissible cover of $X$ by $G$-equivariant affinoid subspaces, then the following map is $c$-faithfully flat:
\begin{equation*}
    \mathcal{H}_{t,c,\omega}(X,G)\rightarrow \prod_{i=1}^n\mathcal{H}_{t,c,\omega}(U_i,G).
\end{equation*}
\end{prop}
\begin{proof}
Let $U=X\left(\frac{g_1}{f},\cdots,\frac{g_n}{f}\right)$ be a $G$-invariant rational subdomain. In virtue of item $(iv)$ in Theorem \ref{teo existence of quotients of rigid spaces by finite groups}, we may assume that $f,g_1,\cdots,g_n\in A^G$. As $f$ does not vanish in $U$, it is a unit. Hence, it is bounded below. In particular, there is some $r\geq 0$ such that $U\subset X\left(\frac{1}{\pi^rf} \right)$, and we may factorize the open immersion $U\subset X$ into the following chain of open immersions:
\begin{equation*}
    U\rightarrow \cdots \rightarrow X\left(\frac{1}{\pi^rf}\right)\left(\frac{g_1}{f}\right) \rightarrow  X\left(\frac{1}{\pi^rf} \right)\rightarrow X.
\end{equation*}
Thus, the map $ \mathcal{H}_{t,c,\omega}(X,G)\rightarrow \mathcal{H}_{t,c,\omega}(U,G)$  is $c$-flat by Corollary \ref{coro c-flatness in the simple case}. Let $\{U_i\}_{i=1}^n$ be a $G$-invariant rational cover of $X$. By Theorem \ref{completed Cherednik algebras are a sheaf} we have an exact complex:
\begin{equation*}
    0\rightarrow \mathcal{H}_{t,c,\omega}(X,G)\rightarrow \prod_{i=1}^n\mathcal{H}_{t,c,\omega}(U_i,G)\rightarrow \cdots\rightarrow \mathcal{H}_{t,c,\omega}(\cap_{i=1}^nU_i,G)\rightarrow 0
\end{equation*}
such that each of the objects in the complex is a $c$-flat $\mathcal{H}_{t,c,\omega}(X,G)$-module. Furthermore, by the proofs of Theorem \ref{completed Cherednik algebras are a sheaf}, and Corollary \ref{coro c-flatness in the simple case}, it follows that for any finite index set $I\subset \{ 1,\cdots, n\}$ there are Fréchet-Stein presentations:
\begin{equation*}
    \mathcal{H}_{t,c,\omega}(X,G)=\varprojlim_m A_m, \textnormal{ } \mathcal{H}_{t,c,\omega}(\cap_{i\in I}U_i,G)=\varprojlim_m B^I_m,
\end{equation*}
such that the maps $A_m\rightarrow B^I_m$ are flat, and the following complexes are exact:
\begin{equation*}
    0\rightarrow A_m\rightarrow \prod_{i=1}^nB^{\{i\}}_m\rightarrow \cdots \rightarrow \prod_{\vert I\vert =k} B^{I}_m  \rightarrow \cdots\rightarrow B^{\{1,\cdots,n\}}_m\rightarrow 0.
\end{equation*}
But this implies that $A_m\rightarrow \prod_{i=1}^nB^{\{i\}}_m$ is faithfully flat. Thus, the map:
\begin{equation*}
    \mathcal{H}_{t,c,\omega}(X,G)\rightarrow \prod_{i=1}^n\mathcal{H}_{t,c,\omega}(U_i,G),
\end{equation*}
is $c$-faithfully flat. Let $U$ be a $G$-invariant affinoid subdomain of $X$. Then by Theorem \cite[Theorem 3.3.20]{bosch2014lectures}, $U$ is covered by finitely many $G$-invariant rational subdomains of $X$. Let $\{V_i\}_{i=1}^n$ be such a cover. Then we have:
\begin{equation*}
    \mathcal{H}_{t,c,\omega}(X,G)\rightarrow \mathcal{H}_{t,c,\omega}(U,G)\rightarrow \prod_{i=1}^n\mathcal{H}_{t,c,\omega}(V_i,G).
\end{equation*}
As the second map and the composition are $c$-faithfully flat, it follows that the first map is $c$-flat. The statement for $c$-faithful flatness follows analogously.
\end{proof}
As an upshot of the proof of the proposition, we get the following Corollary:
\begin{coro}
Let $U=\Sp(B)\subset X$ be a  $G$-invariant affinoid subdomain, and let $\mathscr{A}_{\omega/t}$ be a Cherednik lattice of $\mathcal{A}_{\omega/t}(X)$ such that $\mathscr{B}_{\omega/t}=B^{\circ}\otimes_{A^{\circ}}\mathscr{A}_{\omega/t}$ is a Cherednik lattice of $\mathcal{A}_{\omega/t}(U)$. For big enough $n\geq 0$ the map:
\begin{equation*}
    \mathcal{H}_{t,c,\omega}(X,G)_{\pi^n\mathscr{A}_{\omega/t}}\rightarrow \mathcal{H}_{t,c,\omega}(U,G)_{\pi^n\mathscr{B}_{\omega/t}}
\end{equation*}
is a two-sided flat morphism of $K$-algebras.
\end{coro}
\subsection{Localization of co-admissible modules}
We will now use the results of the previous section to develop  the basics of a theory of co-admissible modules over $p$-adic Cherednik algebras. In particular, we will define a localization process  that will allow us to extend every co-admissible $\mathcal{H}_{t,c,\omega}(X,G)$ to a sheaf of $\mathcal{H}_{t,c,\omega}$-modules. 
\begin{teo}\label{teo localization functor}
  Let $X=\Sp(A)$ be a $G$-variety with an étale map $X\rightarrow \mathbb{A}^r_K$, and a sheaf of $p$-adic Cherednik algebras $\mathcal{H}_{t,c,\omega}$. There is an exact embedding:
  \begin{equation*}
      \operatorname{Loc}(-):\mathcal{C}(\mathcal{H}_{t,c,\omega}(X,G))\rightarrow \Mod(\mathcal{H}_{t,c,\omega}),
  \end{equation*}
  defined on affinoid subdomains $U\subset X$ by the rule:
  \begin{equation*}
      \operatorname{Loc}(\mathcal{M})(U)=\mathcal{H}_{t,c,\omega}(U,G)\wideparen{\otimes}_{\mathcal{H}_{t,c,\omega}(X,G)}\mathcal{M}.
  \end{equation*}
$\operatorname{Loc}(\mathcal{M})$ has trivial higher \v{C}ech cohomology for any co-admissible module $\mathcal{M}$. 
\end{teo}
\begin{proof}
This is equivalent to \cite[Theorem 8.2]{ardakov2019}.
\end{proof}
\begin{defi}\label{defi co-admissible modules on affinoid spaces}
We denote the essential image of $\operatorname{Loc}(-)$ by $\mathcal{C}(\mathcal{H}_{t,c,\omega})$, and call its elements co-admissible $\mathcal{H}_{t,c,\omega}$-modules.
\end{defi}
Thus, we have achieved a local definition of co-admissible modules over a $p$-adic Cherednik algebra.
Ideally, we would like to generalize this notion to arbitrary smooth $G$-varieties. In order to do this, we would need a version of Kiehl's Theorem. That is, we would need to show that every sheaf of $\mathcal{H}_{t,c,\omega}$-modules which is locally co-admissible arises as the localization of a co-admissible $\mathcal{H}_{t,c,\omega}(X,G)$-module. In light of the structural results for $p$-adic Cherednik algebras shown in Sections \ref{section Completed Cherdnik algebras are Fréchet-Stein}, \ref{sheaf of p-adic}, and \ref{section flatness of transition maps}, it follows that the behavior of the sheaves $G\ltimes\wideparen{\D}_X$ and $\mathcal{H}_{t,c,\omega}$ is rather similar. Thus, one could expect that an approach similar to that of \cite{ardakov2019}  could work in this setting. Alas, the situation is not so simple, and there are certain non-trivial technical difficulties that need to be worked around in order to have a well-behaved theory
of co-admissible modules over sheaves of $p$-adic Cherednik algebras. In particular, the approach in \cite{ardakov2019} is based upon some auxiliary topologies, namely the admissible and accessible topologies (\emph{cf}. \cite[Definitions 3.2, 4.6]{ardakov2019}). These topologies do no have a clear analog in the setting of $p$-adic Cherednik algebras. In particular, consider the following family:
\begin{equation*}
    \{ U=\Sp(B)\subset X \textnormal{ }\vert \textnormal{ } \mathscr{B}_{\omega/t}:=B^{\circ}\otimes_{A^{\circ}}\mathscr{A}_{\omega/t} \textnormal{ is a Cherednik lattice of }U  \},
\end{equation*}
is not closed under intersections. In particular, it is not a site. In order to overcome this issues, one would need a more general version of Definition \ref{defi rigid cher algebras}, which allows for more general $G$-invariant affine formal models of $X$. This, in turn, would generate the need for a more general study of restriction algebras, so the contents of Section \ref{section restriction rings} would also need to be extended. On top of that, there is another layer of complexity arising from the fact that given two affinoid subdomains $U,V\subset X$,  it is not generally true that $S(U\cap V)\subset S(U)\cap S(V)$.\\
It is the believe of the author that all these difficulties can be handled, and  it is possible to define a \emph{bona fide} theory of co-admissible modules over sheaves of $p$-adic Cherednik algebras. However, the solutions to some of these issues are not immediate, so we will postpone the study of sheaves of co-admissible modules over $p$-adic Cherednik algebras to a later paper.


\begin{thebibliography}{12}
\bibitem{ardakov2013irreducible} K. Ardakov and S. Wadsley, On irreducible representations of compact p-adic analytic groups. \emph{Annals of Mathematics} 2013, 453-557, JSTOR.
\bibitem{ardakov2021equivariant} K. Ardakov. Equivariant $\D$-modules on rigid analytic spaces. \emph{arXiv} 170807475.
\bibitem{ardakov2019} K. Ardakov, S. Wadsley. $\wideparen{\D}$-modules on rigid analytic spaces I. \emph{J. Reine Angew. Math.} 747 (2019), 221--276.
\bibitem{ardakov2015d} K. Ardakov, S. Wadsley. $\wideparen{\D}$-modules on rigid analytic spaces II: Kashiwara's equivalence. \emph{J. Algebraic Geom.} 27 (2018), 647--701.
\bibitem{Ardakov_Bode_Wadsley_2021} K. Ardakov, A. Bode, S. Wadsley. $\wideparen{\D}$-modules on rigid analytic spaces III: Weak holonomicity and operations. \emph{arXiv} 1904.13280. 
\bibitem{beilinson1993proof} A. Beilinson and J. Bernstein, A proof of Jantzen conjectures \emph{Advances in Soviet mathematics}, vol. 16 (1993), Part 1, 1-50, American Mathematical Society Providence, RI. 
\bibitem{berkovich2012spectral} V. Berkovich, Spectral theory and analytic geometry over non-Archimedean fields \emph{Mathematical Surveys and Monographs}, vol. 33 (1990),American Mathematical Soc.
\bibitem{bezrukavnikov2009parabolic} R. Bezrukavnikov, P. Etingof, Parabolic induction and restriction functors for rational Cherednik algebras, \emph{Selecta Mathematica}, vol 14(2009), no 3, 397-425, Springer.
\bibitem{bhatt2022six} B. Bhatt and D. Hansen, The six functors for Zariski-constructible sheaves in rigid geometry \emph{Compositio Mathematica}, vol. 158 (2022), no. 2 437-482, London Mathematical Society.
\bibitem{bode2019completed} A. Bode. Completed tensor products and a global approach to $p$-adic analytic differential operators. \emph{Math. Proc. Camb. Philos. Soc.} 167 (2019), no. 2, 389--416.
\bibitem{bode2021operations} A. Bode. Six operations for $\wideparen{\D}$-modules on rigid analytic spaces. \emph{arXiv} 2110.09398.
\bibitem{Borel-Admiss} A. Borel, Admissible representations of a semi-simple group over a local field with vectors fixed under an Iwahori subgroup, \emph{Inventiones mathematicae}, vol 35 (1976).
\bibitem{BGR} S. Bosch, U. Güntzer, R. Remmert, Non-archimedean analysis, \emph{ Springer}, 1984.
\bibitem{bosch2014lectures} S. Bosch. Lectures on formal and rigid geometry. \emph{Lecture Notes in Mathematics}, 2015. Springer, Cham, 2014.
\bibitem{complexrefl} M. Broué, G. Malle, and Raphaël Rouquier, Complex reflection groups, braid groups, Hecke algebras, \emph{Journal für die reine und angewandte Mathematik}, no. 500 (1998), pp. 127-190.
\bibitem{Bush-Kutz1} C. Bushnell, P. Kutzko, Smooth representations of reductive $p$-adic groups: structure theory via types. \emph{Proceedings of the London Mathematical Society}, 77(3), pp. 582–634.
\bibitem{Bush-Kutz2} C. J. Bushnell, P. C. Kutzko,  Semisimple Types in $\operatorname{Gl}_n$, \emph{Compositio Mathematica}, no. 1 (1999), pp. 53-97.
\bibitem{Cherednikmcdonalds}I. Cherednik, Double affine Hecke algebras, Knizhnik- Zamolodchikov equations, and Macdonald's operators, IMRN, Duke Math. J. 9(1992), 171-180.
\bibitem{DAHA-Fourier} I. Cherednik, Double affine Hecke algebras and difference Fourier transforms, \emph{Inventiones Math.}  152:2 (2003), 213–303.
\bibitem{introduction DAHA} I. Cherednik, Introduction to double Hecke algebras \emph{arXiv preprint}, 2004, https://arxiv.org/abs/math/0404307.
\bibitem{DAHA CUP} I. Cherednik, Double Affine Hecke Algebras \emph{Cambridge University Press}, London Mathematical Society Lecture notes Series, 2005.
\bibitem{lecturesorbifolds}M. Davis, Lectures on orbifolds and reflection groups, \emph{Transformation Groups and
Moduli Spaces of Curves}, ALM 16, Higher Education Press,
Beijing 2010, pp. 63—93.
\bibitem{etingof2010lecture} P. Etingof, M. Xiaoguang, Lecture notes on Cherednik algebras \emph{arXiv preprint} 1001.0432. 
\bibitem{etingof2004cherednik} P. Etingof, Cherednik and Hecke algebras of varieties with a finite group action, \emph{arXiv preprint} 0406499.
\bibitem{etingof2002symplectic} P. Etingof and V. Ginzburg, Symplectic reflection algebras, Calogero-Moser space, and deformed Harish-Chandra homomorphism \emph{Invent. math.},vol 147 (2002), 243-348, Springer.
\bibitem{Fresnel2004} J. Fresnel an M. van der Put, Rigid Analytic Geometry and Its Applications \emph{Progress in Mathematics} 2003,Birkh{\"a}user, Boston.
\bibitem{fulton2013representation} W. Fulton and J. Harris, Representation theory: a first course \emph{Graduate Texts in Mathematics}, vol. 129 (2013) Springer Science \& Business Media.
\bibitem{ginzburg1998lectures} V. Ginzburg, Letcures on $\D$-modules , University of Chicago, 1998.
\bibitem{ginzburg2003category} V. Ginzburg, N. Guay, E. Opdam and R. Rouquier, On the category $\OX$ for rational Cherednik algebras \emph{Inent. Math.}, vol 154 (2003), 617-651, Springer.
\bibitem{grothendieck1961elements} A. Grothendieck, {\'E}l{\'e}ments de g{\'e}om{\'e}trie alg{\'e}brique: III. {\'E}tude cohomologique des faisceaux coh{\'e}rents, premi{\`e}re partie, \emph{Publications Math{\'e}matiques de l'IH{\'E}S}, vol. 11 (1961), pp. 5-167.
\bibitem{hansen2016quotients} D. Hansen, Quotients of adic spaces by finite groups, To appear in \emph{Math. Res. Letters}.
\bibitem{helemskii2012homology} A. Helemskii, The homology of Banach and topological algebras \emph{Mathematics and its Applications},2012, Springer Dordrecht.
\bibitem{houzel2006seminaire} C. Houzel (ed.). S\'eminaire Banach. \emph{Lecture Notes in Mathematics} 277. Springer-Verlag, 1972.
\bibitem{hotta2007d} R. Hotta, K. Takeuchi, T. Tanisaki. $\D$-modules, perverse sheaves, and representation theory. \emph{Progress in Mathematics}, 236. Birkhäuser, 2008.
\bibitem{H-L} R.B. Howlett, G. Lehrer, Induced cuspidal representations and generalised Hecke rings, \emph{Invent. Math.} 58 (1980), 37–64.
\bibitem{huber2013etale} R. Huber, Étale Cohomology of Rigid Analytic Varieties and Adic Spaces \emph{Aspects of Mathematics}, Vieweg Teubner Verlag Wiesbaden, Berlin, 1996.
\bibitem{Iwahori-Chevalley} N. Iwahori, On the structure of a Hecke ring of a Chevalley group over a finite field, \emph{J. Fac. Sci. Univ. Tokyo}  Sect. I 10 (1964), 215–236.
\bibitem{de1996etale} J. de Jong, M. Van der Put, Étale cohomology of rigid analytic spaces \emph{Documenta Mathematica}, vol. 1 (1996), 1-56,  Universiät Bielefeld, Fakultät für Mathematik.
\bibitem{kiehl1967Theorem} Kiehl, Reinhardt. Theorem A und Theorem B in der nichtarchimedischen Funktionentheorie. \emph{Invent. Math.}, vol 2(1967), 256-273. Springer.
\bibitem{le2007rigid} B. Le Stum, Rigid cohomology \emph{Cambridge Tracts in Mathematics}, 2007, Cambridge University Press.
\bibitem{Root-SystemsAV} E. Looijenga, Root systems and elliptic curves, \emph{ Inventiones mathematicae}, vol. 38 (1976).
\bibitem{HochschildDmodules} F. Peña, Hochschild cohomology of $\wideparen{\D}$-modules on rigid analytic spaces \emph{Arxiv} 2504.16707
\bibitem{HochschildDmodules2} F. Peña, Hochschild cohomology of $\wideparen{\D}$-modules on rigid analytic spaces II \emph{Arxiv} 2504.17167
\bibitem{p-adicCatO} F. Peña, Category $\OX$ for $p$-adic rational Cherednik algebras \emph{Arxiv} 2504.16699
\bibitem{prosmans2000homological} Prosmans, J.-P. Schneiders. A homological study of bornological spaces. Laboratoire analyse, g\'eom\'etrie et applications, Unit\'e mixte de recherche, Institut Galilee, Universit\'e Paris, CNRS, 2000.
\bibitem{rinehart1963differential} G. S. Rinehart. Differential forms on general commutative algebras. \emph{Trans. Amer. Math.} 108 (1963), 195--222.
\bibitem{rouquier1998complex} R. Rouquier, G. Malle and M. Broué, Complex reflection groups, braid groups, Hecke algebras \emph{Jurnal für die reine und angewandte Mathematik},vol. 1998, no. 500, 127-190, Walter de Gruyter GmbH \& Co. KG Berlin, Germany.
\bibitem{Schmidt2010VermaMO} T. Schmidt, Verma modules over p-adic Arens–Michael envelopes of reductive Lie algebras, \emph{Journal of Algebra}, vol. 390 (2010), pp. 160-180
\bibitem{schneider2013nonarchimedean} P. Schneider. Nonarchimedean functional analysis. \emph{Springer Monographs in Mathematics}, Springer-Verlag, Berlin, 2002.
\bibitem{schneider2002algebras} P. Schneider, J. Teitelbaum. Algebras of $p$-adic distributions and admissible representations. \emph{Invent. Math.} 153 (2003), no. 1, 145-196.
\bibitem{perf} P. Scholze,Perfectoid Spaces. \emph{Publ.math.IHES} vol. 116, pp. 245–313 (2012).
\bibitem{taylor1972homology} J. Taylor, Homology and cohomology for topological algebras \emph{Advances in Mathematics}, vol. 9 (1972), no.2 137-182, Elsevier.
\bibitem{Wedhorn2019adic} T. Wedhorn, Adic spaces \emph{arXiv} 1910.05934.
\bibitem{Weibel1994introduction} C. Weibel, An introduction to homological algebra \emph{Cambridge Studies in Advanced Mathematics},1994 no 38, Cambridge university press. 
\bibitem{stacks-project} The Stacks Project Authors. \emph{Stacks project}. \url{https://stacks.math.columbia.edu}, 2024.
\end{thebibliography}
\end{document}